\documentclass[10pt]{amsart}
\usepackage{amsmath,amssymb,latexsym,esint,cite,mathrsfs}
\usepackage{verbatim,wasysym,inputenc}
\usepackage[left=2.5cm,right=2.5cm,top=2.7cm,bottom=2.5cm]{geometry}
\usepackage{tikz,enumitem,graphicx, subfig, microtype, color}
\usepackage{epic,eepic}
\usepackage{mathrsfs}

\usepackage[colorlinks=true,urlcolor=blue, citecolor=red,linkcolor=blue,
linktocpage,pdfpagelabels, bookmarksnumbered,bookmarksopen]{hyperref}
\usepackage[hyperpageref]{backref}
\usepackage[english]{babel}
\allowdisplaybreaks
\newtheorem{Thm}{Theorem}[section]

\newtheorem{Lem}{Lemma}[section]
\newtheorem{Pro}{Proposition}[section]

\newtheorem{Rek}{Remark}[section]

\newcommand{\R}{\mathbb{R}}

\numberwithin{equation}{section} \numberwithin{figure}{section}

\begin{document}

\title[Critical Hartree equation]{Local Uniqueness of blow-up  solutions for critical Hartree equations in bounded domain}

\author[M.\ Squassina]{Marco Squassina}
\author[M. Yang]{Minbo Yang}
\author[S. Zhao]{Shunneng Zhao}

\address{Marco Squassina, \newline\indent
	Dipartimento di Matematica e Fisica \newline\indent
	Universit\`a Cattolica del Sacro Cuore, \newline\indent
	Via della Garzetta 48, 25133, Brescia, Italy}
\email{marco.squassina@unicatt.it}

\address{Minbo Yang,  Shunneng Zhao \newline\indent Department of Mathematics, Zhejiang Normal University, \newline\indent
	Jinhua, Zhejiang, 321004, People's Republic of China}
\email{M. Yang: mbyang@zjnu.edu.cn}
\email{S. Zhao: snzhao@zjnu.edu.cn}

\subjclass[2010]{35A02, 35J20,35B40, 35J60}
	\keywords{Hartree equation; Local Pohozaev identities; Blow-up analysis; Robin function; Local uniqueness.}
	\thanks{$^\dag$Minbo Yang is partially supported by NSFC (11971436, 12011530199) and ZJNSF(LZ22A010001, LD19A010001).}

\begin{abstract}
In this paper we are interested in the following critical Hartree equation
\begin{equation*}
\begin{cases}
-\Delta u
=\displaystyle{\Big(\int_{\Omega}\frac{u^{2_{\mu}^\ast} (\xi)}{|x-\xi|^{\mu}}d\xi\Big)u^{2_{\mu}^\ast-1}}+\varepsilon u
,~~~\text{in}~\Omega,\\
u=0,~~~~~~~~~~~~~~~~~~~~~~~~~~~~~~~~~~~~~~~~~~~\text{on}~\partial\Omega,
\end{cases}
\end{equation*}
where $N\geq4$, $0<\mu\leq4$, $\varepsilon>0$ is a small parameter, $\Omega$ is a bounded  domain in $\mathbb{R}^N$, and $2_{\mu}^\ast=\frac{2N-\mu}{N-2}$ is the critical exponent in the sense of the Hardy-Littlewood-Sobolev inequality. By establishing various versions of local Pohozaev identities and applying blow-up analysis, we first investigate the location of the blow-up  points for single bubbling solutions to above the Hartree equation. Next we prove the local uniqueness of the blow-up solutions that concentrates at the non-degenerate critical point of the Robin function for $\varepsilon$ small.

\end{abstract}

\maketitle

\begin{center}
	\begin{minipage}{9.5cm}
		\small
		\tableofcontents
	\end{minipage}
\end{center}
\medskip

\section{Introduction and main results}
In a
celebrated paper \cite{Brezis-1983}, Brezis and Nirenberg introduced the following Sobolev critical problem
\begin{equation}\label{eq1.1}
\begin{cases}
-\Delta u
=|u|^{2^{\ast}-2}u+\varepsilon u
,~~~~~~~~~~\text{in}~\Omega,\\
u=0,~~~~~~~~~~~~~~~~~~~~~~~~~~~~~~\text{on}~\partial\Omega,
\end{cases}
\end{equation}
where  $2^{\ast}=\frac{2N}{N-2}$ with $N\geq3$, $\varepsilon>0$ is a real positive parameter, $\Omega$ is a smooth bounded domain in $\mathbb{R}^N$.
The existence of a positive solution $u_{\varepsilon}$ to \eqref{eq1.1}, i.e., a solution which achieves the infimum
\begin{equation*}
S_{\varepsilon}:=\inf\limits_{u\in H_{0}^{1}(\Omega)\setminus\{0\}}\frac{\int_{\Omega}(|\nabla u|^2-\varepsilon u^2)dx}{\big(\int_{\Omega}|u|^{2^\ast}dx\big)^{\frac{2}{2^{\ast}}}}
\end{equation*}
has been proved by Brezis and Nirenberg in \cite{Brezis-1983} provided $\varepsilon\in(0,\lambda_1)$ in dimension $N\geq4$ and when $\varepsilon\in(\lambda_{\ast},\lambda_1)$ in dimension $N=3$, where $\lambda_1$ is the first eigenvalue of $-\Delta$ with Dirichlet boundary
condition and $\lambda_{\ast}\in(0,\lambda_1)$ depends on the domain $\Omega$.
On the other hand, when $\varepsilon=0$, problem~\eqref{eq1.1} becomes much more delicate. Pohozaev first proved in \cite{Pohozaev-1965} that \eqref{eq1.1} does not have any solutions
 in the case where $\Omega$ is a star-shaped domain. Bahri and Coron \cite{Bahri-1988} proved that \eqref{eq1.1} has a solution when $\Omega$ has a
nontrivial topology and $\varepsilon=0$.

As $\varepsilon\to0$,
Rey \cite{Rey-1990} proved that if a solution $u_{\varepsilon}$ of \eqref{eq1.1} satisfies
 \begin{equation}\label{S1}
|\nabla u_\varepsilon|^2\rightarrow S^{\frac{N}{2}}\delta_{x_0},~~~\text{as}~\varepsilon\rightarrow0,
\end{equation}
where $\delta_x$ denotes the Dirac mass at $x$ and $S$ the best Sobolev constant defined by
\begin{equation*}
S:=\inf\Big\{\frac{\|\nabla u\|_{L^2}}{\|u\|_{L^{2^\ast}}}~:u\in D^{1,2}(\mathbb{R}^N)\setminus\{0\}\Big\}.
\end{equation*}
Then $x_{0}\in \Omega$ is a critical point of Robin function $\mathcal{R}(x)$ (see \eqref{RB}). Conversely, if $N\geq5$ and $x_0$ is a nondegenerate critical point of $\mathcal{R}(x)$, then for $\varepsilon$ sufficiently small
 \eqref{eq1.1} has a family of solutions $u_\varepsilon$  satisfying \eqref{S1}.
Let $\Omega$ is a smooth bounded domain in $\mathbb{R}^N$ and $N\geq4$, Rey \cite{Rey-1989} (independently by Han \cite{Han-1991}) considered
\begin{equation}\label{eq1.2}
	\begin{cases}
		-\Delta u
		=N(N-2)u^{\frac{N+2}{N-2}}+\varepsilon u
		,~~~~~~~~~~~\text{in}~\Omega,\\
		u=0,~~~~~~~~~~~~~~~~~~~~~~~~~~~~~~~~~~~~~~~~~\text{on}~\partial\Omega,
	\end{cases}
\end{equation}
and studied the localtion of blow-up point for solutions to \eqref{eq1.2} and blowing up rate, namely,
\begin{equation*}
\begin{split}
\lim\limits_{\varepsilon\rightarrow 0}&\varepsilon\|u_{\varepsilon}\|_{L^{\infty}}^{\frac{2(N-4)}{N-2}}=\frac{(N-2)^3\omega_N}{2\rho_N}\mathcal{R}(x_0)~~~~~~~~\text{if }~N\geq5;\\&
\lim\limits_{\varepsilon\rightarrow 0}\varepsilon\ln\|u_{\varepsilon}\|_{L^{\infty}}
=4\omega_4\mathcal{R}(x_0)~~~~~~~~~~~~~~\text{if }~N=4;
\end{split}
\end{equation*}
where $\omega_N$ is a measure of the unit sphere of $\mathbb{R}^N$, $\rho_N=\int_{0}^{\infty}\frac{r^{N-1}}{(1+r^2)^{N-2}}dr$ and
\begin{equation}\label{RB}
\mathcal{R}(x):=H(x,x)
\end{equation}
is called the
Robin function of $\Omega$ at point $x$. The Green's function of the Dirichlet problem for the Laplacian is then defined by
   \begin{equation}\label{Robin}
G(x,y)=\frac{1}{(N-2)\omega_N|y-x|^{N-2}}-H(x,y),
\end{equation}
and it satisfies
    \begin{equation*}
\begin{cases}
 -\Delta G(x,\cdot)=\delta_x~~~~~~~~~~~~~~~~~~\text{in}~\Omega,\\
  G(x,\cdot)=0~~~~~~~~~~~~~~~~~~~~~~~~\text{on}~\partial\Omega.
 \end{cases}
    \end{equation*}
Musso and Pistoia in \cite{Musso-Pistoia-2002} and Bahri, Li and Rey in \cite{B-L-R} studied existence of solutions which blow-up at $k\geq1$ different points of $\Omega$.

To investigate the uniqueness of the blow-up solutions, Grossi in \cite{Grossi-2000} proved
	the uniqueness of the solutions to \eqref{eq1.1} under suitable assumptions on the the
	domain $\Omega$, see also \cite{Grossi-2002}.
	If $N\geq5$ and
	for $\varepsilon$ small enough, Cerqueti in \cite{Cerqueti-1999} proved that if the domain $\Omega$ is symmetric with respect to
	the coordinate hyperplanes ${x_k=0}$ and convex in the $x_k$-directions, there
	exists a unique solution $u_{\varepsilon}$ of \eqref{eq1.2} with the property that
	\begin{equation}\label{UNI}
		\lim\limits_{\varepsilon\rightarrow0}\frac{\int_{\Omega}|\nabla u_{\varepsilon}|^2dx}{\big(\int_{\Omega}|u_{\varepsilon}|^{2^{\ast}}dx\big)^{\frac{1}{2^{\ast}}}}=S,
	\end{equation}
	and this solution is
	nondegenerate. Later inspired by Li in\cite{Li-1995},
	Cerqueti and Grossi in \cite{Cerqueti-2001} follow
	closely the line of \cite{Li-1995} for the blow-up analysis which be used to prove the uniqueness
	result for the solutions of \eqref{eq1.2}, and they proved that
	all solutions of \eqref{eq1.2} satisfy the property \eqref{UNI} under the same hypothesis on the domain $\Omega$.
	In \cite{Glangetas-1993}, Glangetas considered the problem \eqref{eq1.2} and it is shown that if $N\geq5$, the uniqueness of solutions $u_\varepsilon$ of \eqref{eq1.2} with the property that \eqref{S1} for $\varepsilon$ small enough, where $x_0$ is a nondegenerate critical point of Robin function $\mathcal{R}(x)$.
	Recently, considering the uniqueness result of Glangetas in \cite{Glangetas-1993}, Cao, Luo and Peng\cite{Cao-2021} proved that if $\varepsilon$ is small,  problem~\eqref{eq1.1} has a unique solution  provided the domain $\Omega$ is convex and $N\geq6$.
	For other related results, we refer the readers to \cite{ Cao-1998,Cao--Heinz2003,Cao-2019,Ddeng-L-Y2015,Guo-2017} and their references for the existence and uniqueness of
	solutions for nonlinear elliptic equations.

	There is wide literature about the study of the asymptotic behavior of the solutions for the almost critical problem
	\begin{equation}\label{SC}
		\begin{cases}
			-\Delta u
			=u^{2^{\ast}-1-\varepsilon}
			,~~~~~~~~~~~~~~~\text{in}~\Omega,\\
			u=0,~~~~~~~~~~~~~~~~~~~~~~~~~~~~~~~\text{on}~\partial\Omega,
		\end{cases}
	\end{equation}
	Atkinson and Peletier \cite{ATKINSON-1986} studied the asymptotic behavior of subcritical solutions $u_\varepsilon$ to \eqref{SC}.
	Brezis and Peletier \cite{Brezis-Peletier} used the method of PDE to obtain the same results as that in \cite{ATKINSON-1986} for the spherical domains. Wei in \cite{Wei-1998} further locate the blow-up point $x_0$ and to give a precise asymptotic expansion of the least energy solutions for problem \eqref{SC}. Rey in \cite{OREY-1991} and Musso and Pistoia in \cite{Musso-Pistoia-2003} proved, for $\varepsilon>0$ small enough, a positive solutions with two positive blow-up points provided the domain $\Omega$ have a small hole. For $\varepsilon<0$, Del Pino, Felmer and Musso in \cite{DFM1} established a positive solutions which blows-up at  two positive points when the domain $\Omega$ have a hole and for $\varepsilon$ small enough. Del Pino, Felmer and Musso in \cite{DFM3} found solutions with three or more positive blow-up points under suitable assumptions on the domain $\Omega$. Towers of positive bubbles for problem \eqref{SC} were constructed by Del Pino, Dolbeault and Musso in\cite{DFM5} under suitable assumptions on the nondegeneracy of Robin's function $\mathcal{R}(x)$ and Green's function.

In this paper we are interested in the following critical Hartree equation
\begin{equation}\label{NCE}
	\begin{cases}
		-\Delta u
		=\displaystyle{\Big(\int_{\Omega}\frac{u^{2_{\mu}^\ast}(\xi)}{|x-\xi|^{\mu}}d\xi\Big)u^{2_{\mu}^\ast-1}}+\varepsilon u
		,~~\text{in}~\Omega,\\
		u=0,~~~~~~~~~~~~~~~~~~~~~~~~~~~~~~~~~~~~~~~~~\text{on}~\partial\Omega,
	\end{cases}
\end{equation}
where $N\geq4$, $0<\mu\leq4$, $\varepsilon>0$ is a small parameter, $\Omega$ is a smooth and bounded domain in $\mathbb{R}^N$ and the exponent $2_{\mu}^\ast:=\frac{2N-\mu}{N-2}$ is critical in the sense of the Hardy-Littlewood-Sobolev inequality. To under the  critical growth of the nonlocal problem, we need to recall the famous Hardy-Littlewood-Sobolev inequality.
\begin{Pro}\label{Pro1.1}
\cite{Lieb-1983-AM} Let $\theta,r>1$ and $0<\mu<N$ with $\frac{1}{\theta}+\frac{1}{r}=2-\frac{\mu}{N}$. Let $f\in L^\theta(\mathbb{R}^N)$
and $g\in L^r(\mathbb{R}^N)$, there exists a sharp constant $C(\theta,r,\mu,N)$ independent of $f,g$, such that
\begin{equation}\label{eq1.4}
\int_{\mathbb {R}^N}\int_{\mathbb {R}^N}{\frac{f(x)g(\xi)}{|x-\xi|^{\mu}}}dxd\xi\leq C(\theta,r,\mu,N)\|f\|_{\theta}\|g\|_{r}.
\end{equation}
If $\theta=r=\frac{2N}{2N-\mu}$, then
\begin{equation*}
C(\theta,r,\mu,N)=C_{N,\mu}=\pi^{\frac{\mu}{2}}\frac{\Gamma(\frac{N-\mu}{2})}
{\Gamma(N-\frac{\mu}{2})}\Big\{\frac{\Gamma(N)}
{\Gamma(\frac{N}{2})}\Big\}^{\frac{N-\mu}{N}}.
\end{equation*}
There is equality in \eqref{eq1.4}
 if and only if $f\equiv(const.)g$ and
\begin{equation*}
g(x)=\bar{A}(\bar{\gamma}^2+|x-\bar{a}|^2)^{-\frac{N+\mu}{2}}
\end{equation*}
for some $\bar{A}\in\mathbb{C}$, $0\neq\bar{\gamma}\in\mathbb{R}$ and $\bar{a}\in\mathbb{R}^N$.
\end{Pro}
\noindent
According to Proposition \ref{Pro1.1}, the functional
$$
\int_{\R^{N}}\int_{\R^{N}}\frac{|u(x)|^{p}|v(y)|^{p}}{|x-y|^{\mu}}dxdy
$$
is well defined in $H^1(\R^N)\times H^1(\R^N)$ if
$\frac{2N-\mu}{N}\leq p\leq\frac{2N-\mu}{N-2}$.
Here, it is quite natural to call $\frac{2N-\mu}{N}$ the lower Hardy-Littlewood-Sobolev critical exponent and $2_{\mu}^{\ast}:=\frac{2N-\mu}{N-2}$ the upper Hardy-Littlewood-Sobolev critical exponent. In the following, we use $S_{H,L}$ to denote best constant defined by
\begin{equation}\label{eq1.12}
S_{H,L}:=\inf_{u\in D^{1,2}(\mathbb{R}^N)\setminus\{0\}}\frac{\int_{\mathbb{R}^N}|\nabla u|^2dx}{\big(\int_{\mathbb{R}^N}\int_{\mathbb{R}^N}\frac{|u(x)|^{2_{\mu}^\ast}
|u(\xi)|^{2_{\mu}^\ast}}{|x-\xi|^\mu}dxd\xi\big)^{\frac{N-2}{2N-\mu}}}.
\end{equation}
In this way, we know that \eqref{NC} is closely related to the nonlocal Euler-Lagrange equation
\begin{equation}\label{Intcase}
\displaystyle-\Delta u
=\left(\int_{\mathbb{R}^{N}}\frac{u^{2_{ \mu}^{\ast}}(\xi)}{|x-\xi|^{\mu}}d\xi\right)
u^{2_{\mu}^{\ast}-1},~~~\text{in}~~\mathbb{R}^{N}.
\end{equation}
For the critical nonlocal equation \eqref{Intcase},  Du and Yang in \cite{Du-Yang2018-DCDS} and Guo, Hu, Peng and Shuai in\cite{Guo2019} studied equation \eqref{Intcase} with critical exponent $\frac{2N-\mu}{N-2}$ by analyzing the corresponding integral system. They also classified the uniqueness of the positive solutions and concluded that every positive
solution of \eqref{Intcase} must assume the form
\label{REL} (see \cite{Du-Yang2018-DCDS, Gao-Yang})
\begin{equation*}
 	\begin{split}
 		\bar{U}_{z,\lambda}(x)&=S^{\frac{(N-\mu)(2-N)}{4(N-\mu+2)}}C _{N,\mu}^{\frac{2-N}{2(N-\mu+2)}}\big[N(N-2)\big]^{\frac{N-2}{4}}U_{z,\lambda}(x),
 	\end{split}
 \end{equation*}
where (see \cite{Talenti-1976-AMPA}),
\begin{equation*}
	U_{z,\lambda}(x):=\frac{\lambda^{\frac{N-2}{2}}}{(1+\lambda^2|x-z|^2)^{\frac{N-2}{2}}}~~x\in\mathbb{R}^N,~z\in\mathbb{R}^N, ~\lambda\in\mathbb{R}^{+},
\end{equation*}
is the unique family of positive solutions of
\begin{equation}\label{eq14.1}
	-\Delta u=N(N-2)u^{2^\ast-1},~~~~~\text{in}~~\mathbb{R}^N.
\end{equation}

In a recent paper \cite{Gao-Yang},  Gao and Yang considered the Hartree type Brezis-Nirenberg problem \eqref{NC}.
 They proved a Brezis-Nirenberg type result saying that:
if $N\geq4$, \eqref{NC} has a nontrivial solution for $\varepsilon>0$;
if $N=3$, then there exists $\lambda_{\ast}$ such that \eqref{NCE} has a nontrivial solution for $\varepsilon>\lambda_{\ast}$, where $\varepsilon$ is not an eigenvalue of $-\Delta$ with homogeneous Dirichlet boundary data; if $N\geq3$ and $\varepsilon\leq0$, \eqref{NCE} admits no solutions when $\Omega$ is star-shaped. More recently, Yang and Zhao in \cite{Yang-Zhao-2021} proved that the solution $u_\varepsilon$ of \eqref{NCE} blows up exactly at a critical point of the Robin function that cannot be on the boundary of $\Omega$ via the Lyapunov-Schmit reduction method.  Existence of bubbling solutions for equation \eqref{NCE} were constructed by Yang, Ye and Zhao in\cite{Yang-Zhao-2022} under suitable assumptions on the nondegeneracy of Robin's function $\mathcal{R}(x)$.

Naturally, one would like to know whether the  local uniqueness results of the blow-up solutions hold true for the Hartree equation and if it is possible to prove the location of blow-up point for the critical problem via local Pohozaev identities. For $N\geq4$ and $\varepsilon>0$ is small, one of the main purposes of this paper is to locate the blow-up point of single bubbling solutions for the following critical Hartree equation by local Pohozaev identities and blow-up analysis,
\begin{equation}\label{NC}
	\begin{cases}
		-\Delta u
		=\displaystyle{\Big(\int_{\Omega}\frac{u^{2_{\mu}^\ast}(\xi)}{|x-\xi|^{\mu}}d\xi\Big)u^{2_{\mu}^\ast-1}}+\varepsilon u
		,~~\text{in}~\Omega,\\
		u=0,~~~~~~~~~~~~~~~~~~~~~~~~~~~~~~~~~~~~~~~~~~\text{on}~\partial\Omega,
	\end{cases}
\end{equation}
and study the local uniqueness of the blow-up solutions for problem \eqref{NC} provided $N\geq6$ and $\varepsilon$ small enough.

Before stating the main results, it is useful to introduce some notations. We denote by
\begin{equation}\label{AHL}
	A_{H,L}:=\big[N(N-2)\big]^{\frac{N-\mu+2}{2}}C _{N,\mu}^{-1}S_{H,L}^{\frac{\mu-N}{2}}.
\end{equation}
We know that $U_{z,\lambda}(x)$ is the solution of
\begin{equation*}
	-\Delta U_{z,\lambda}=A_{H,L}\displaystyle{\Big(\int_{\mathbb{R}^N}\frac{U_{z,\lambda}^{2_{\mu}^\ast}(\xi)}{|x-\xi|^{\mu}}d\xi\Big)U_{z,\lambda}^{2_{\mu}^\ast-1}},~~~\text{in}~~\mathbb{R}^N.
\end{equation*}
We denote by $PU_{z,\lambda}$ the projection of a function $U_{z,\lambda}$ onto $H_{0}^{1}(\Omega)$, namely,
\begin{equation}\label{PUX}
		\Delta PU_{z,\lambda}
		=\Delta U_{z,\lambda},
		~~~~\text{in}~\Omega,~~~~
		PU_{z,\lambda}=0,~~~~~\text{on}~\partial\Omega.
\end{equation}
Let us set
$$\psi_{z,\lambda}=U_{z,\lambda}-PU_{z,\lambda}.$$
We remark that $\psi_{z,\lambda}$ is a harmonic function such that
$$\psi_{z,\lambda}=U_{z,\lambda},~~~~\text{on}~~\partial\Omega.$$
A first result that we obtain is the following.
\begin{Thm}\label{Thm1.2}
Assume that $N\geq4$ and $u_{\varepsilon}$ is a sequence of solutions of $H_{0}^1(\Omega)$ satisfying
\begin{equation}\label{T2}
\|u_{\varepsilon}\|_{L^{\infty}}\rightarrow+\infty ~~\text{as}~~\varepsilon\rightarrow0~~\text{and}~~|u_{\varepsilon}(x)|\leq CU_{x_{\varepsilon},\lambda_{\varepsilon}}(x),
\end{equation}
with its maximum at $x_{\varepsilon}$ and $\lambda_{\varepsilon}^{\frac{N-2}{2}}=\max_{x\in\Omega}u_{\varepsilon}(x)=u_{\varepsilon}(x_{\varepsilon})$.
Then there exists $x_0\in\Omega$ such that as $\varepsilon\rightarrow0$, $x_{\varepsilon}\rightarrow x_0$, and $x_0$ is a critical point of Robin function $\mathcal{R}$, i.e.,
$\nabla\mathcal{R}(x_{0})=0$.
\end{Thm}
\begin{Rek}\rm 
The above results have been be proved by Yang and Zhao \cite{Yang-Zhao-2021} by using reduction arguments under
different conditions, in this paper we will prove this theorem via the local Pohozaev identity \eqref{PH}.
\end{Rek}

In \cite{Yang-Zhao-2022}, authors constructed the existence of single bubbling solutions for \eqref{NC} via the Lyapunov-Schmit reduction method. Along with this interesting results, we will obtain a type of local uniqueness results of these. More precisely, we can prove the following result.
\begin{Thm}\label{Thm1.4}
Let $N\geq6$ and $\mu\in(0,4)$. Suppose that $N,\mu$ satisfy the assumptions:
\begin{equation*}\small
	(**)\hspace{6mm}
	\left\lbrace
	\begin{aligned}
		&\mu\in (0, 4),   \quad\mu\hspace{1mm}\mbox{is sufficiently close to}\hspace{1mm} 0\hspace{1mm},  ~~~~~\quad  \mbox{if} \hspace{1mm}N>6, \\
		&\mu\in (0, 4), \quad  \mu\hspace{1mm}\mbox{is sufficiently close to}\hspace{1mm}   0\hspace{1mm} or \hspace{1mm}4,  \quad \mbox{if} \hspace{1mm} N=6.
	\end{aligned}
	\right.
\end{equation*}
Suppose that $\{u_{\varepsilon}^{(j)}\}(j=1,2)$ are two families of functions of $H_{0}^1(\Omega)$ such that $u_{\varepsilon}^{(j)}$ is a solution of \eqref{NC} and satisfies  condition \eqref{T2}.
If $x_0\in\Omega$ is an isolated non-degenerate critical point of the Robin function $\mathcal{R}(x)$, then
there exists $\varepsilon_0^{\prime}>0$ such that for any $\varepsilon\in(0,\varepsilon_0^{\prime}]$,
such type of solutions
\begin{equation*}
u_{\varepsilon}^{(j)}=PU_{x_{\varepsilon}^{(j)},\lambda_{\varepsilon}^{(j)}}+w_{\varepsilon}^{(j)},~~~~j=1,2,
 \end{equation*}
are unique, that is,
$u_{\varepsilon}^{(1)}=u_{\varepsilon}^{(2)}$, $x_{\varepsilon}^{(1)}=x_{\varepsilon}^{(2)}$, $\lambda_{\varepsilon}^{(1)}=\lambda_{\varepsilon}^{(2)}$ and $w_{\varepsilon}^{(1)}=w_{\varepsilon}^{(2)}$.
\end{Thm}
\begin{Rek}\label{NU}\rm
We remark that there are some restriction on the dimension $N$ and parameter $\mu$, since some  estimates do not work well in applying the local Pohozaev identities and applying blow-up analysis. For example, we note that in the case that $N=6$ and $\mu=4$,  it is difficult to prove that $c_0=0$  in Lemma~\ref{Lem6.6}  by \eqref{559} and \eqref{5510}  (see proof of~Lemma~\ref{Lem6.6} below). Moreover, for $N=5$, if $x_0$ is a nondegerate critical point of Robin function $\mathcal{R}$, from Lemma~\ref{Lem5.3}, we have
\begin{equation}\label{AA011}
|x_{\varepsilon}-x_0|=O\big(\frac{1}{\lambda_{\varepsilon}}\big).
\end{equation}
However, by \eqref{AA011}, we can not derive the estimates of \eqref{6.91} and \eqref{6.102} in Lemma~\ref{6.91}.
\end{Rek}
The proof of the main results is mainly inspired by \cite{Ddeng-L-Y2015,Guo-2017},  let $u_{\varepsilon}^{(1)}$ and $u_{\varepsilon}^{(2)}$ be two different positive solutions of \eqref{eq1.1}.
Set
\begin{equation*}
\eta_{\varepsilon}(x):=\frac{u_{\varepsilon}^{(1)}(x)-u_{\varepsilon}^{(2)}(x)}{\|u_{\varepsilon}^{(1)}(x)-u_{\varepsilon}^{(2)}(x)\|_{L^{\infty}}},
\end{equation*}
then for any fixed $\theta\in(0,1)$ and small $\varepsilon$, we want to prove $|\eta_{\varepsilon}(x)|<\theta$ for all $x\in\Omega$,  which is incompatible with the fact $\|\eta_{\varepsilon}\|_{L^{\infty}}=1$.
Compared with the local Brezis-Nirenberg problem, the appearance of nonlocal critical term in problem \eqref{NC} brings new difficulties. For example, the corresponding local Pohozaev identities will have various new terms involving volume integral, which causes new difficulties in estimates of the order of each terms in the local Pohozaev identities precisely. To apply the blow-up analysis, we need to use some nondegeneracy results, for which little is known. In a recent preprint, Yang and Zhao \cite{Yang-Zhao-2021} prove the nondegeneracy results for some special $N$ and $\mu$.

\begin{Lem}\label{Lemma4.5}
	Assume that $N\geq3$ and  $\bar{U}_{z,\lambda}$  be the corresponding family of unique family of positive solutions of \eqref{Intcase}.  Suppose that $N, \mu$ satisfy the assumptions:
	\begin{equation*}
		\left\lbrace
		\begin{aligned}
			&\mu\in (0, N),   \quad\mu\hspace{1mm}\mbox{is sufficiently close to}\hspace{1mm} 0\hspace{1mm}or \hspace{1mm} N,  \hspace{1mm} \mbox{if} \hspace{1mm}N=3\hspace{1mm} or \hspace{1mm} 4, \\
			&\mu\in (0, 4),   \quad\mu\hspace{1mm}\mbox{is sufficiently close to}\hspace{1mm} 0\hspace{1mm},  \quad  \mbox{if} \hspace{1mm}N\geq5 \hspace{1mm}\mbox{but}\hspace{1mm} N\neq6, \\
			&\mu\in (0, 4], \quad  \mu\hspace{1mm}\mbox{is sufficiently close to}\hspace{1mm}   0\hspace{1mm}, \hspace{1mm}4\hspace{1mm}or=4,  \quad \mbox{if} \hspace{1mm} N=6.
		\end{aligned}
		\right.
	\end{equation*}
	Then the linearized operator of \eqref{Intcase} at $\bar{U}_{z,\lambda}$ defined by
	\begin{equation*}
		\begin{split}
			L\phi=-\Delta\phi-2_{\mu}^{\ast}\Big(\frac{1}{|x|^{\mu}}\ast(\bar{U}_{z,\lambda}^{2_{\mu}^{\ast}-1}\phi)\Big)\bar{U}_{z,\lambda}^{2_{\mu}^{\ast}-1}
			-(2_{\mu}^{\ast}-1)\Big(\frac{1}{|x|^{\mu}}\ast \bar{U}_{z,\lambda}^{2_{\mu}^{\ast}}\Big)\bar{U}_{z,\lambda}^{2_{\mu}^{\ast}-2}\phi
		\end{split}
	\end{equation*}
	only admits solutions in $D^{1,2}(\mathbb{R}^N)$ of the form
	\begin{equation*}
		\phi=\bar{a}D_{\lambda}\bar{U}_{z,\lambda}+\mathbf{b}\cdot\nabla \bar{U}_{z,\lambda},
	\end{equation*}
	where $\bar{a}\in \mathbb{R}$, $\mathbf{b}\in\mathbb{R}^N$.
\end{Lem}

\noindent
{\bf Notation.}
In what follows we let
\begin{equation*}
\|u\|_{H_{0}^1}=\big(\int_{\Omega}|\nabla u|^2dx\big)^{\frac{1}{2}},~~~
			\|u\|_{L^q}=\big(\int_{\Omega}|u|^qdx\big)^{\frac{1}{q}}.
		\end{equation*}
as the standard norm in the Sobolev space $H_{0}^1(\Omega)$ and $L^q(\Omega)$-standard norm for any $q\in[1,+\infty)$, respectively. Moreover, $A=o(\bar{\alpha})$ means $A/\bar{\alpha}\rightarrow0$ as $\varepsilon\rightarrow0$ and $A=O(\bar{\alpha})$
		means that $|A/\bar{\alpha}|\leq C$.

\par
This paper is organized as follows: in section 2, we first construct the local Pohozaev type identities for critical Hartree equation and
give the proof of Theorem \ref{Thm1.2}. In section 3, we give the proofs of some crucial estimates for blow-up solutions and Green's function, and completed the proof of Theorem~\ref{Thm1.4}. The proof of Theorem~\ref{Thm1.4} requires some technical computations which are given in section 5 and Appendixs A-D.

\section{Local Pohozaev identities and blow-up  points}
The first goal of this section is to establish the local Pohozaev type identities for the critical Hartree equation. As an application of the local Pohozaev type identities, we are able to locate the blow-up  points in Theorem~\ref{Thm1.2}.
\subsection{Local Pohozaev type identities}
\begin{Lem}\label{Lem2.1}
Suppose that $u_{\varepsilon}$ be a solution of the equation~\eqref{NC}.
Then, for any bounded domain $\Omega^{\prime}\subset\Omega$, one has the following identity holds:
\begin{equation}\label{eq2.6}
\begin{split}
&-\int_{\partial \Omega^{\prime}}\frac{\partial u_{\varepsilon}}{\partial\nu}\big\langle x-x_\varepsilon,\nabla u_{\varepsilon}\big\rangle ds+\frac{1}{2}\int_{\partial \Omega^{\prime}}|\nabla u_{\varepsilon}|^2\big\langle x-x_\varepsilon,\nu\big\rangle ds-\frac{N-2}{2}\int_{\partial\Omega^{\prime}} \frac{\partial u_{\varepsilon}}{\partial\nu}u_{\varepsilon}ds\\&=
\Big(\frac{N-2}{2}-\frac{N}{2_\mu^{\ast}}\Big)\int_{\Omega^{\prime}}\int_{\Omega\setminus\Omega^{\prime}}\frac{u_{\varepsilon}^{2_{\mu}^\ast}(x)u_{\varepsilon}^{2_{\mu}^\ast}(\xi)}
{|x-\xi|^{\mu}}d\xi dx
+\frac{\mu}{2_\mu^\ast}\int_{\Omega^{\prime}}\int_{\Omega\setminus\Omega^{\prime}}x\cdot(x-\xi)\frac{u_{\varepsilon}^{2_{\mu}^\ast}(x)u_{\varepsilon}^{2_{\mu}^\ast}(\xi)}
{|x-\xi|^{\mu+2}}d\xi dx\\&
\hspace{5mm}+\frac{1}{2_\mu^\ast}\int_{\partial \Omega^{\prime}}\int_{ \Omega}\frac{u_{\varepsilon}^{2_{\mu}^\ast}(x)u_{\varepsilon}^{2_{\mu}^\ast}(\xi)}
{|x-\xi|^{\mu}}\big\langle x-x_\varepsilon,\nu\big\rangle d\xi ds+\frac{\varepsilon}{2}\int_{\partial \Omega^{\prime}}u_{\varepsilon}^2\big\langle x-x_\varepsilon,\nu\big\rangle ds-\varepsilon\int_{\Omega^{\prime}} u_{\varepsilon}^2dx,
\end{split}
\end{equation}
and
\begin{equation}\label{PH}
\begin{split}
-\int_{\partial\Omega^{\prime}}&\frac{\partial u_{\varepsilon}}{\partial x_j}\frac{\partial u_{\varepsilon}}{\partial\nu}ds+\frac{1}{2}\int_{\partial \Omega^{\prime}}|\nabla u_{\varepsilon}|^2\nu_jds
=\frac{2
}{2_\mu^\ast}
\int_{\partial\Omega^{\prime}}\int_{ \Omega^{\prime}}\frac{|u_{\varepsilon}(x)|^{2_{\mu}^\ast}|u_{\varepsilon}(\xi)|^{2_{\mu}^\ast}}
{|x-\xi|^{\mu}}\nu_jd\xi ds+\frac{\varepsilon}{2}\int_{\partial \Omega^{\prime}} u_{\varepsilon}^2\nu_jds\\&+\frac{1}{2_\mu^{\ast}}
\int_{\partial\Omega^{\prime}}\int_{ \Omega\setminus\Omega^{\prime}}\frac{|u_{\varepsilon}(x)|^{2_{\mu}^\ast}|u_{\varepsilon}(\xi)|^{2_{\mu}^\ast}}
{|x-\xi|^{\mu}}\nu_jd\xi ds+\frac{\mu}{2_\mu^{\ast}}\int_{\Omega^{\prime}}\int_{\Omega\setminus\Omega^{\prime}}(x_j-\xi_j)\frac{|u_{\varepsilon}(\xi)|^{2_{\mu}^\ast}|u_{\varepsilon}(x)|^{2_{\mu}^\ast}}
{|x-\xi|^{\mu+2}}d\xi dx,
\end{split}
\end{equation}
for $j=1,\cdots,N$,
where $\nu=\nu(x)$ denotes the unit outward normal to the boundary $\partial\Omega^{\prime}$.
\end{Lem}
\begin{proof}
By elliptic regularity theory, we know that the solution $u_{\varepsilon}$ of \eqref{NC} is of $ C^2$.  Without loss of generality, we may suppose that $x_\varepsilon=0$.
 Since $u_\varepsilon$ satisfies
\begin{equation}\label{eq2.7}
-\Delta u_{\varepsilon}
=\Big(\int_{\Omega}\frac{u_{\varepsilon}^{2_{\mu}^\ast}(\xi)}
{|x-\xi|^{\mu}}d\xi\Big)u_{\varepsilon}^{2_{\mu}^\ast-1}+\varepsilon u_{\varepsilon}.
\end{equation}
Then we multiply the equation \eqref{eq2.7} by $\langle x,\nabla u_{\varepsilon}\rangle$ and integrating on $\Omega^{\prime}$,  we obtain
\begin{equation}\label{eq2.8}
-\int_{\Omega^{\prime}}\Delta u_{\varepsilon}\langle x,\nabla u_{\varepsilon}\rangle dx
=\int_{\Omega^{\prime}}\langle x,\nabla u_{\varepsilon}\rangle\Big(\int_{\Omega}\frac{u_{\varepsilon}^{2_{\mu}^\ast}(\xi)}
{|x-\xi|^{\mu}}d\xi\Big)u_{\varepsilon}^{2_{\mu}^\ast-1}(x)dx+\varepsilon\int_{\Omega^{\prime}}\langle x,\nabla u_{\varepsilon}\rangle u_{\varepsilon}(x)dx.
\end{equation}
Notice that
\begin{equation*}
\begin{split}
\int_{\Omega^{\prime}}&\big\langle x,\nabla u_{\varepsilon}(x)\big\rangle\Big(\int_{\Omega}\frac{u_{\varepsilon}^{2_{\mu}^\ast}(\xi)}
{|x-\xi|^{\mu}}d\xi\Big)u_{\varepsilon}^{2_{\mu}^\ast-1}(x)dx\\&
=\int_{\Omega^{\prime}}\big\langle x,\nabla u_{\varepsilon}(x)\big\rangle\Big(\int_{\Omega^{\prime}}\frac{u_{\varepsilon}^{2_{\mu}^\ast}(\xi)}
{|x-\xi|^{\mu}}d\xi\Big)u_{\varepsilon}^{2_{\mu}^\ast-1}(x)dx+\int_{\Omega^{\prime}}\big\langle x,\nabla u_{\varepsilon}(x)\big\rangle\Big(\int_{\Omega\setminus\Omega^{\prime}}\frac{u_{\varepsilon}^{2_{\mu}^\ast}(\xi)}
{|x-\xi|^{\mu}}d\xi\Big)u_{\varepsilon}^{2_{\mu}^\ast-1}(x)dx.
\end{split}
\end{equation*}
We calculate the first term on the right-hand side to obtain
\begin{equation*}
\begin{split}
2_{\mu}^\ast\int_{\Omega^{\prime}}&\big\langle x,\nabla u_{\varepsilon}(x)\big\rangle\Big(\int_{\Omega^{\prime}}\frac{u_{\varepsilon}^{2_{\mu}^\ast}(\xi)}
{|x-\xi|^{\mu}}d\xi\Big)u_{\varepsilon}^{2_{\mu}^\ast-1}(x)dx\\&
=-N\int_{\Omega^{\prime}}\int_{\Omega^{\prime}}\frac{u_{\varepsilon}^{2_{\mu}^\ast}(x)u_{\varepsilon}^{2_{\mu}^\ast}(\xi)}
{|x-\xi|^{\mu}}dxd\xi
+\mu\int_{\Omega^{\prime}}\int_{\Omega^{\prime}}x\cdot(x-\xi)\frac{u_{\varepsilon}^{2_{\mu}^\ast}(\xi)}
{|x-\xi|^{\mu+2}}u_{\varepsilon}^{2_{\mu}^\ast}(x)d\xi dx\\&~~~+\int_{\partial \Omega^{\prime}}\int_{ \Omega^{\prime}}\frac{u_{\varepsilon}^{2_{\mu}^\ast}(x)u_{\varepsilon}^{2_{\mu}^\ast}(\xi)}
{|x-\xi|^{\mu}}\big\langle x,\nu\big\rangle d\xi ds.
\end{split}
\end{equation*}
Similarly, we can deduce
\begin{equation*}
\begin{split}
2_{\mu}^\ast\int_{\Omega^{\prime}}&\langle \xi,\nabla u_{\varepsilon}(\xi)\rangle\Big(\int_{\Omega^{\prime}}\frac{u_{\varepsilon}^{2_{\mu}^\ast}(x)}
{|x-\xi|^{\mu}}dx\Big)u_{\varepsilon}^{2_{\mu}^\ast-1}(\xi)d\xi\\&
=-N\int_{\Omega^{\prime}}\int_{\Omega^{\prime}}\frac{u_{\varepsilon}^{2_{\mu}^\ast}(x)u_{\varepsilon}^{2_{\mu}^\ast}(\xi)}
{|x-\xi|^{\mu}}dxd\xi
+\mu\int_{\Omega^{\prime}}\int_{\Omega^{\prime}}\xi\cdot(\xi-x)\frac{u_{\varepsilon}^{2_{\mu}^\ast}(x)}
{|x-\xi|^{\mu+2}}u_{\varepsilon}^{2_{\mu}^\ast}(\xi)dxd\xi\\&~~~+\int_{\partial \Omega^{\prime}}\int_{ \Omega^{\prime}}\frac{u_{\varepsilon}^{2_{\mu}^\ast}(x)u_{\varepsilon}^{2_{\mu}^\ast}(\xi)}
{|x-\xi|^{\mu}}\langle \xi,\nu\rangle dxds.
\end{split}
\end{equation*}
Thus we can prove that
\begin{equation}\label{poh1}
\begin{split}
\int_{\Omega^{\prime}}&\big\langle x,\nabla u_{\varepsilon}(x)\big\rangle\Big(\int_{\Omega^{\prime}}\frac{u_{\varepsilon}^{2_{\mu}^\ast}(\xi)}
{|x-\xi|^{\mu}}d\xi\Big)u_{\varepsilon}^{2_{\mu}^\ast-1}(x)dx\\&
=\frac{\mu-2N}{22_{\mu}^\ast}\int_{\Omega^{\prime}}\int_{\Omega^{\prime}}\frac{u_{\varepsilon}^{2_{\mu}^\ast}(x)u_{\varepsilon}^{2_{\mu}^\ast}(\xi)}
{|x-\xi|^{\mu}}dxd\xi
+\frac{1}{2_{\mu}^\ast}\int_{\partial \Omega^{\prime}}\int_{ \Omega^{\prime}}\frac{u_{\varepsilon}^{2_{\mu}^\ast}(x)u_{\varepsilon}^{2_{\mu}^\ast}(\xi)}
{|x-\xi|^{\mu}}\big\langle x,\nu\big\rangle d\xi ds.
\end{split}
\end{equation}
For the second term, integration by parts, we have
\begin{equation}\label{poh2}
\begin{split}
2_{\mu}^\ast&\int_{\Omega^{\prime}}\big\langle x,\nabla u_{\varepsilon}(x)\big\rangle\Big(\int_{\Omega\setminus\Omega^{\prime}}\frac{u_{\varepsilon}^{2_{\mu}^\ast}(\xi)}
{|x-\xi|^{\mu}}d\xi\Big)u_{\varepsilon}^{2_{\mu}^\ast-1}(x)dx\\&=-N\int_{\Omega^{\prime}}\int_{\Omega\setminus\Omega^{\prime}}\frac{u_{\varepsilon}^{2_{\mu}^\ast}(x)u_{\varepsilon}^{2_{\mu}^\ast}(\xi)}
{|x-\xi|^{\mu}}d\xi dx
+\mu\int_{\Omega^{\prime}}\int_{\Omega\setminus\Omega^{\prime}}x\cdot(x-\xi)\frac{u_{\varepsilon}^{2_{\mu}^\ast}(\xi)}
{|x-\xi|^{\mu+2}}u_{\varepsilon}^{2_{\mu}^\ast}(x)d\xi dx\\&~~~+\int_{\partial \Omega^{\prime}}\int_{ \Omega\setminus\Omega^{\prime}}\frac{u_{\varepsilon}^{2_{\mu}^\ast}(x)u_{\varepsilon}^{2_{\mu}^\ast}(\xi)}
{|x-\xi|^{\mu}}\langle x,\nu\rangle d\xi ds.
\end{split}
\end{equation}
On the other hand, we have
\begin{equation}\label{poh3}
-\int_{\Omega^{\prime}}\Delta u_{\varepsilon}\big\langle x,\nabla u_{\varepsilon}\big\rangle dx=\frac{2-N}{2}\int_{\Omega^{\prime}}|\nabla u_{\varepsilon}|^{2}dx+\frac{1}{2}\int_{\partial \Omega^{\prime}}\big\langle x,\nu\big\rangle|\nabla u_{\varepsilon}|^2dx-\int_{\partial\Omega^{\prime}}\frac{\partial u_{\varepsilon}}{\partial\nu}\big\langle x,\nabla u_{\varepsilon}\big\rangle ds
\end{equation}
and
\begin{equation}\label{poh4}
\int_{\Omega^{\prime}}\big\langle x,\nabla u_{\varepsilon}\big\rangle u_{\varepsilon}dx=\frac{1}{2}\int_{\partial \Omega^{\prime}} u_{\varepsilon}^2\big\langle x,\nu\big\rangle ds-\frac{N}{2}\int_{\Omega^{\prime}} u_{\varepsilon}^2dx.
\end{equation}
In view of Green's formulas, we have
\begin{equation}\label{poh5}
\begin{split}
\int_{\Omega^{\prime}}|\nabla u_{\varepsilon}|^2&=-\int_{\Omega^{\prime}}u_{\varepsilon}\Delta u_{\varepsilon}dx+\int_{\partial \Omega^{\prime}}\frac{\partial u_{\varepsilon}}{\partial\nu}u_{\varepsilon}ds\\&
=\int_{\Omega^{\prime}}\int_{\Omega}\frac{u_{\varepsilon}^{2_{\mu}^\ast}(x)u_{\varepsilon}^{2_{\mu}^\ast}(\xi)}
{|x-\xi|^{\mu}}dxd\xi+\varepsilon\int_{\Omega^{\prime}}u_{\varepsilon}^2dx+\int_{\partial \Omega^{\prime}}\frac{\partial u_{\varepsilon}}{\partial\nu}u_{\varepsilon}ds.
\end{split}
\end{equation}
Hence by \eqref{eq2.8}, \eqref{poh1},  \eqref{poh2},  \eqref{poh3},  \eqref{poh4} and  \eqref{poh5} imply that \eqref{eq2.6}.
\par
To prove \eqref{PH},
We multiply \eqref{eq2.7} by $\frac{\partial u_{\varepsilon}}{\partial x_j}$ and integrating on $\Omega^{\prime}$,  we have
\begin{equation}\label{PH1}
-\int_{\Omega^{\prime}}\Delta u_{\varepsilon}\frac{\partial u_{\varepsilon}}{\partial x_j}dx
=\int_{\Omega^{\prime}}\frac{\partial u_{\varepsilon}}{\partial x_j}\Big(\int_{\Omega}\frac{|u_{\varepsilon}(\xi)|^{2_{\mu}^\ast}}
{|x-\xi|^{\mu}}d\xi\Big)|u_{\varepsilon}(x)|^{2_{\mu}^\ast-1}dx+\varepsilon\int_{\Omega^{\prime}}\frac{\partial u_{\varepsilon}}{\partial x_j}u_{\varepsilon}dx.
\end{equation}
Similar to the above argument, we have
\begin{equation*}
\begin{split}
\int_{\Omega^{\prime}}\frac{\partial u_{\varepsilon}}{\partial x_j}\Big(\int_{\Omega^{\prime}}\frac{|u_{\varepsilon}(\xi)|^{2_{\mu}^\ast}}
{|x-\xi|^{\mu}}d\xi\Big)&|u_{\varepsilon}(x)|^{2_{\mu}^\ast-1}dx
=
-(2_{\mu}^\ast-1)\int_{\Omega^{\prime}}\frac{\partial u_{\varepsilon}}{\partial x_j}\Big(\int_{\Omega^{\prime}}\frac{|u_{\varepsilon}(\xi)|^{2_{\mu}^\ast}}
{|x-\xi|^{\mu}}d\xi\Big)|u_{\varepsilon}(x)|^{2_{\mu}^\ast-1}dx\\&~~~
+\mu\int_{\Omega^{\prime}}\int_{\Omega^{\prime}}(x_j-\xi_j)\frac{|u_{\varepsilon}(\xi)|^{2_{\mu}^\ast}|u_{\varepsilon}(x)|^{2_{\mu}^\ast}}
{|x-\xi|^{\mu+2}}d\xi dx+\int_{\partial \Omega^{\prime}}\int_{ \Omega^{\prime}}\frac{|u_{\varepsilon}(x)|^{2_{\mu}^\ast}|u_{\varepsilon}(\xi)|^{2_{\mu}^\ast}}
{|x-\xi|^{\mu}}\nu_jd\xi ds.
\end{split}
\end{equation*}
Then, we can deduce
\begin{equation}\label{ph00}
\begin{split}
\int_{\Omega^{\prime}}&\frac{\partial u_{\varepsilon}}{\partial x_j}\Big(\int_{\Omega^{\prime}}\frac{|u_{\varepsilon}(\xi)|^{2_{\mu}^\ast}}
{|x-\xi|^{\mu}}d\xi\Big)|u_{\varepsilon}(x)|^{2_{\mu}^\ast-1}dx
\\&
=\frac{N-2}{2N-\mu}
\int_{\partial\Omega^{\prime}}\int_{ \Omega^{\prime}}\frac{|u_{\varepsilon}(x)|^{2_{\mu}^\ast}|u_{\varepsilon}(\xi)|^{2_{\mu}^\ast}}
{|x-\xi|^{\mu}}\nu_jd\xi ds+\frac{\mu(N-2)}{2N-\mu}\int_{\Omega^{\prime}}\int_{\Omega^{\prime}}(x_j-\xi_j)\frac{|u_{\varepsilon}(\xi)|^{2_{\mu}^\ast}|u_{\varepsilon}(x)|^{2_{\mu}^\ast}}
{|x-\xi|^{\mu+2}}d\xi dx.
\end{split}
\end{equation}
Similarly, we also have
\begin{equation*}
\begin{split}
\int_{\Omega^{\prime}}&\frac{\partial u_{\varepsilon}}{\partial \xi_j}\Big(\int_{\Omega^{\prime}}\frac{|u_{\varepsilon}(x)|^{2_{\mu}^\ast}}
{|x-\xi|^{\mu}}dx\Big)|u_{\varepsilon}(\xi)|^{2_{\mu}^\ast-1}d\xi
\\&
=\frac{N-2}{2N-\mu}
\int_{\partial\Omega^{\prime}}\int_{ \Omega^{\prime}}\frac{|u_{\varepsilon}(\xi)|^{2_{\mu}^\ast}|u_{\varepsilon}(x)|^{2_{\mu}^\ast}}
{|x-\xi|^{\mu}}\nu_jdx ds+\frac{\mu(N-2)}{2N-\mu}\int_{\Omega^{\prime}}\int_{\Omega^{\prime}}(\xi_j-x_j)\frac{|u_{\varepsilon}(x)|^{2_{\mu}^\ast}|u_{\varepsilon}(\xi)|^{2_{\mu}^\ast}}
{|x-\xi|^{\mu+2}}dxd\xi.
\end{split}
\end{equation*}
Hence we can get
\begin{equation}\label{PH2}
\begin{split}
\int_{\Omega^{\prime}}&\frac{\partial u_{\varepsilon}}{\partial x_j}\Big(\int_{\Omega^{\prime}}\frac{|u_{\varepsilon}(\xi)|^{2_{\mu}^\ast}}
{|x-\xi|^{\mu}}d\xi\Big)|u_{\varepsilon}(x)|^{2_{\mu}^\ast-1}dx
=\frac{2(N-2)}{2N-\mu}
\int_{\partial\Omega^{\prime}}\int_{ \Omega^{\prime}}\frac{|u_{\varepsilon}(x)|^{2_{\mu}^\ast}|u_{\varepsilon}(\xi)|^{2_{\mu}^\ast}}
{|x-\xi|^{\mu}}\nu_jd\xi ds.
\end{split}
\end{equation}
Now similar to the calculations of \eqref{ph00}, we know
\begin{equation}\label{ph01}
\begin{split}
\int_{\Omega^{\prime}}&\frac{\partial u_{\varepsilon}}{\partial x_j}\Big(\int_{\Omega\setminus\Omega^{\prime}}\frac{|u_{\varepsilon}(\xi)|^{2_{\mu}^\ast}}
{|x-\xi|^{\mu}}d\xi\Big)|u_{\varepsilon}(x)|^{2_{\mu}^\ast-1}dx
\\&
=\frac{N-2}{2N-\mu}
\int_{\partial\Omega^{\prime}}\int_{ \Omega\setminus\Omega^{\prime}}\frac{|u_{\varepsilon}(x)|^{2_{\mu}^\ast}|u_{\varepsilon}(\xi)|^{2_{\mu}^\ast}}
{|x-\xi|^{\mu}}\nu_jd\xi ds+\frac{\mu(N-2)}{2N-\mu}\int_{\Omega^{\prime}}\int_{\Omega\setminus\Omega^{\prime}}(x_j-\xi_j)\frac{|u_{\varepsilon}(\xi)|^{2_{\mu}^\ast}|u_{\varepsilon}(x)|^{2_{\mu}^\ast}}
{|x-\xi|^{\mu+2}}d\xi dx.
\end{split}
\end{equation}
So, by \eqref{PH1}, \eqref{PH2} and \eqref{ph01}, we can prove \eqref{PH}. This finishes the proof.
\end{proof}

\subsection{Location of the blow-up  point}
We first prove the following lemma.
\begin{Lem}\label{Lem21}
Assume that $N\geq4$ and $u_{\varepsilon}$  is a sequence of solutions of problem \eqref{NC} satisfying the
assumptions of Theorem~\ref{Thm1.2}. Then there holds $\lambda_{\varepsilon}d_{\varepsilon}\rightarrow+\infty$ for $\varepsilon$ small enough.
\end{Lem}
\begin{proof}
Assume that $\lambda_{\varepsilon}d_{\varepsilon}\rightarrow \tilde{c}<+\infty$ as $\varepsilon\rightarrow0$ and $u_{\varepsilon}$ is a solution of \eqref{NC} with $\lambda_{\varepsilon}^{\frac{N-2}{2}}=\max\limits_{x\in\Omega}u_{\varepsilon}(x)=u_{\varepsilon}(x_{\varepsilon})\rightarrow+\infty$ as $\varepsilon\rightarrow0$. Set $v_{\varepsilon}=\lambda_{\varepsilon}^{-\frac{N-2}{2}}u_{\varepsilon}(\lambda_{\varepsilon}^{-1}x+x_{\varepsilon})$. Then $v_{\varepsilon}(x)$ satisfies
\begin{equation*}
\begin{cases}
-\Delta v_{\varepsilon}(x)=\displaystyle\Big(\int_{\Omega_{\varepsilon}}\frac{v_{\varepsilon}^{2_{\mu}^{\ast}}(\xi)}{|x-\xi|^{\mu}}d\xi\Big)v_{\varepsilon}^{2_{\mu}^{\ast}-1}(x)+\frac{\varepsilon}{\lambda_{\varepsilon}^{2}}v_{\varepsilon}~~~~\text{in}~~\Omega_{\varepsilon}:
=\big\{x:\frac{x}{\lambda_{\varepsilon}}+x_{\varepsilon}\in\Omega\big\},\\
v_{\varepsilon}(x)=0~~~~~~~~~~~~~~~~~~~~~~~~~~~~~~~~~~~~~~~~~~~~~~~~~~~~\text{on}~~\partial\Omega_{\varepsilon},\\
v_{\varepsilon}(0)=\max\limits_{x\in\Omega_{\varepsilon}}v_{\varepsilon}(x)=1.
\end{cases}
\end{equation*}
As $\varepsilon\rightarrow0$, by the elliptic regularity, we have $v_{\varepsilon}\rightarrow v$ in $C^2_{loc}(\R^N_{+})$ and $v$ satisfies
\begin{equation*}
\begin{cases}
-\Delta v(x)=\displaystyle\Big(\int_{\R^N_{+}}\frac{v^{2_{\mu}^{\ast}}(\xi)}{|x-\xi|^{\mu}}d\xi\Big)v^{2_{\mu}^{\ast}-1}(x),~v>0,~~~~\text{in}~~\R_{+}^N:=\big\{x\in\R^N:x_N>0\big\},\\
v(0)=\max\limits_{x\in\R^N_{+}}v(x)=1, ~~~v\in H_{0}^1(\R^N_{+}).
\end{cases}
\end{equation*}
It follows from the Pohozaev identity that $v\equiv0$, which contradicts with $v(0)=1$.
\end{proof}

We are ready to give the estimate of $u_{\varepsilon}$ away from $x_{\varepsilon}$.
\begin{Lem}\label{Lem2.2}
Assume that $N\geq4$  and  $u_{\varepsilon}$ is a sequence of solutions of problem \eqref{NC} satisfying the
assumptions of Theorem~\ref{Thm1.2} and $x\in\Omega\setminus B_{R\lambda_{\varepsilon}^{-1}}(x_\varepsilon)$ for $R>0$ is any fixed large constant. Then
\begin{equation}\label{L1}
u_{\varepsilon}(x)=\frac{G(x,x_{\varepsilon})}{\lambda_{\varepsilon}^{\frac{N-2}{2}}}A_{N,\mu}
+O\Big(\frac{\varepsilon }{\lambda_{\varepsilon}^{\frac{N-2}{2}}d^{N-2}}+\frac{1}{\lambda_{\varepsilon}^{\frac{N+2}{2}}d^{N}}+\frac{1}{\lambda_{\varepsilon}^{\frac{N}{2}}d^{N-1}}\Big)~~~~\text{in}~~\Omega\setminus B_{R\lambda_{\varepsilon}^{-1}}(x_\varepsilon)
\end{equation}
and
\begin{equation}\label{L2}
\nabla u_{\varepsilon}(x)=\frac{\nabla G(x,x_{\varepsilon})}{\lambda_{\varepsilon}^{\frac{N-2}{2}}}A_{N,\mu}
+O\Big(\frac{\varepsilon }{\lambda_{\varepsilon}^{\frac{N-2}{2}}d^{N-1}}+\frac{1}{\lambda_{\varepsilon}^{\frac{N+2}{2}}d^{N+1}}+\frac{1}{\lambda_{\varepsilon}^{\frac{N}{2}}d^{N}}\Big)~~~\text{in}~~\Omega\setminus B_{R\lambda_{\varepsilon}^{-1}}(x_\varepsilon).
\end{equation}
Here  $d=|x_{\varepsilon}-x|$ and $A_{N,\mu}=\displaystyle{\int_{ B_{\frac{1}{2}d\lambda_{\varepsilon}}(0)}\int_{B_{\frac{1}{2}d\lambda_{\varepsilon}}(0)}\frac{v_{\varepsilon}^{2_{\mu}^\ast}(\xi)v_{\varepsilon}^{2_{\mu}^\ast-1}(x)}{|x-\xi|^{\mu}}d\xi dx}$.
\end{Lem}
\begin{proof}
By the potential theory and \eqref{NC}, we have
\begin{equation}\label{ee29}
u_{\varepsilon}(x)=\int_{\Omega}G(x,z)\Big(\displaystyle{\Big(\int_{\Omega}\frac{u_{\varepsilon}^{2_{\mu}^\ast}(\xi)}{|z-\xi|^{\mu}}d\xi\Big)u_{\varepsilon}^{2_{\mu}^\ast-1}(z)}+\varepsilon u_{\varepsilon}(z)\Big)dz.
\end{equation}
First we remark that, as a consequence of the moving sphere method, the Talenti bubbles satisfy
	\begin{equation}\label{e29}
		\int_{\mathbb{R}^N}\frac{ U_{x_{\varepsilon},\lambda_{\varepsilon}}^{2_{\mu}^\ast}(\xi)}{|x-\xi|^\mu}d\xi=\frac{N(N-2)}{A_{H,L}}U_{x_{\varepsilon},\lambda_{\varepsilon}}^{2^{\ast}-2_{\mu}^\ast}, \end{equation}
	(see \cite{GLMY}[Proof Theorem 1.2] for example).
Combining \eqref{T2}, \eqref{e29} and $G(x,z)=O\big(\frac{1}{|z-x|^{N-2}}\big)$, we know
\begin{equation}\label{L71}
\begin{split}
\int_{\Omega\setminus B_{\frac{d}{2}}(x_{\varepsilon})}&\int_{\Omega }\frac{u_{\varepsilon}^{2_{\mu}^\ast}(\xi)u_{\varepsilon}^{2_{\mu}^\ast-1}(z)G(x,z)}{|x-\xi|^{\mu}}dx d\xi\leq C\int_{\Omega\setminus B_{\frac{d}{2}}(x_{\varepsilon})}\int_{\mathbb{R}^N }\frac{U_{x_\varepsilon,\lambda_{\varepsilon}}^{2_{\mu}^\ast}(\xi)U_{x_{\varepsilon},\lambda_{\varepsilon}}^{2_{\mu}^\ast-1}(z)G(x,z)}{|x-\xi|^{\mu}}dzd\xi\\&
=O\Big(\frac{1}{\lambda_{\varepsilon}^{\frac{N+2}{2}}}\int_{\big(\Omega\setminus B_{\frac{d}{2}}(x_{\varepsilon})\big)\setminus B_{2d}(x)}\frac{1}{|z-x|^{N-2}
|z-x_{\varepsilon}|^{N+2}}dz\\&~~~+\frac{1}{\lambda_{\varepsilon}^{\frac{N+2}{2}}}\int_{\big(\Omega\setminus B_{\frac{d}{2}}(x_{\varepsilon})\big)\cap B_{2d}(x)}\frac{1}{|z-x|^{N-2}
|z-x_{\varepsilon}|^{N+2}}dz\Big)\\&
=O\Big(\frac{1}{\lambda_{\varepsilon}^{\frac{N+2}{2}}d^{N}}\Big),
\end{split}
\end{equation}
where $d=|x_{\varepsilon}-x|$.
Similar to the above estimates, we can also obtain
\begin{equation}\label{LL3}
\begin{split}
\int_{ B_{\frac{d}{2}}(x_{\varepsilon})}&\int_{\Omega\setminus B_{\frac{d}{2}}(x_{\varepsilon}) }\frac{u_{\varepsilon}^{2_{\mu}^\ast}(\xi)u_{\varepsilon}^{2_{\mu}^\ast-1}(z)G(x,z)}{|x-\xi|^{\mu}}d\xi dz\leq C\int_{B_{\frac{d}{2}}(x_{\varepsilon})}\int_{\mathbb{R}^N }\frac{U_{x_\varepsilon,\lambda_{\varepsilon}}^{2_{\mu}^\ast}(\xi)U_{x_{\varepsilon},\lambda_{\varepsilon}}^{2_{\mu}^\ast-1}(z)G(x,z)}{|x-\xi|^{\mu}}d\xi dz\\&
=O\Big( \lambda_{\varepsilon}^{\frac{N+2}{2}}\int_{ B_{\frac{d}{2}}(x_{\varepsilon})}\frac{1}{|z-x|^{N-2}
(1+\lambda_{\varepsilon}^2|z-x_{\varepsilon}|^2)^{\frac{N+2}{2}}}dz\Big)\\&
=O\Big(\frac{1}{\lambda_{\varepsilon}^{\frac{N}{2}}d^{N-1}}\Big),
\end{split}
\end{equation}
Furthermore, we have
\begin{equation}\label{L7}
\begin{split}
\int_{ B_{\frac{d}{2}}(x_{\varepsilon})}&\int_{B_{\frac{d}{2}}(x_{\varepsilon})}\frac{u_{\varepsilon}^{2_{\mu}^\ast}(\xi)u_{\varepsilon}^{2_{\mu}^\ast-1}(z)G(x,z)}{|x-\xi|^{\mu}}d\xi dz=G(x,x_{\varepsilon})\int_{ B_{\frac{d}{2}}(x_{\varepsilon})}\int_{B_{\frac{d}{2}}(x_{\varepsilon})}\frac{u_{\varepsilon}^{2_{\mu}^\ast}((\xi)u_{\varepsilon}^{2_{\mu}^\ast-1}(z)}{|x-\xi|^{\mu}}d\xi dz\\&~~~+\int_{ B_{\frac{d}{2}}(x_{\varepsilon})}\int_{B_{\frac{d}{2}}(x_{\varepsilon})}\frac{u_{\varepsilon}^{2_{\mu}^\ast}(\xi)u_{\varepsilon}^{2_{\mu}^\ast-1}(z)\big( G(x,z)-G(x,x_{\varepsilon})\big)}{|x-\xi|^{\mu}}d\xi dz\\&
=\frac{G(x,x_{\varepsilon})}{\lambda_{\varepsilon}^{\frac{N-2}{2}}}\int_{ B_{\frac{d\lambda_{\varepsilon}}{2}}(0)}\int_{B_{\frac{d\lambda_{\varepsilon}}{2}}(0)}\frac{v_{\varepsilon}^{2_{\mu}^\ast}(\xi)v_{\varepsilon}^{2_{\mu}^\ast-1}(z)}{|x-\xi|^{\mu}}d\xi dz
\\&~~~+\underbrace{\frac{1}{\lambda_{\varepsilon}^{\frac{N-2}{2}}}\int_{ B_{\frac{d\lambda_{\varepsilon}}{2}}(0)}\int_{B_{\frac{d\lambda_{\varepsilon}}{2}}(0)}\frac{v_{\varepsilon}^{2_{\mu}^\ast}(\xi)v_{\varepsilon}^{2_{\mu}^\ast-1}(z)\big( G(x,\lambda_{\varepsilon}^{-1}z+x_{\varepsilon})-G(x,x_{\varepsilon})\big)}{|x-\xi|^{\mu}}d\xi dz}\limits_{:=\mathcal{I}}\\&
:=\frac{G(x,x_{\varepsilon})}{\lambda_{\varepsilon}^{\frac{N-2}{2}}}A_{N,\mu}+O\Big(\frac{1}{\lambda_{\varepsilon}^{\frac{N}{2}}d^{N-1}}\Big),
\end{split}
\end{equation}
where since
\begin{equation}\label{G3}
\begin{split}
\mathcal{I}
&=O\Big(\frac{1}{\lambda_{\varepsilon}^{\frac{N-2}{2}}}\int_{ B_{\frac{d\lambda_{\varepsilon}}{2}}(0)}\int_{B_{\frac{d\lambda_{\varepsilon}}{2}}(0)}\frac{v_{\varepsilon}^{2_{\mu}^\ast}(\xi)v_{\varepsilon}^{2_{\mu}^\ast-1}(z)}{|x-\xi|^{\mu}}\cdot\frac{|z|}{d^{N-1}\lambda_{\varepsilon}} d\xi dz\Big)\\&
=O\Big(\frac{1}{\lambda_{\varepsilon}^{\frac{N}{2}}d^{N-1}}\int_{ B_{\frac{d\lambda_{\varepsilon}}{2}}(0)} \int_{\mathbb{R}^N}\frac{U_{0,1}^{2_{\mu}^\ast}(\xi)U_{0,1}^{2_{\mu}^\ast-1}(z)|z| }{|x-\xi|^{\mu}}d\xi dz\Big)\\&
=O\Big(\frac{1}{\lambda_{\varepsilon}^{\frac{N}{2}}d^{N-1}}\int_{ B_{\frac{d\lambda_{\varepsilon}}{2}}(0)}\frac{|y|}{(1+|z|^2)^{\frac{N+2}{2}}}dz\Big)\\&
=O\Big(\frac{1}{\lambda_{\varepsilon}^{\frac{N}{2}}d^{N-1}}\Big).
\end{split}
\end{equation}
On the other hand, by \eqref{T2}, $G(x,y)=O\big(\frac{1}{|y-x|^{N-2}}\big)$ and the definition of $U_{x_{\varepsilon},\lambda_{\varepsilon}}$, we can deduce
\begin{equation}\label{GU1}
\begin{split}
\varepsilon\int_{\Omega}G(x,z)u_{\varepsilon}(z)dz
&=O\Big[\frac{\varepsilon }{\lambda_{\varepsilon}^{\frac{N-2}{2}}}\Big(\int_{ B_{\frac{d}{2}}(x_{\varepsilon})}\frac{1}{|z-x|^{N-2}}\frac{1}{|z-x_{\varepsilon}|^{N-2}}dz
+ \int_{\Omega\setminus B_{\frac{d}{2}}(x_{\varepsilon})}\frac{1}{|z-x|^{N-2}}\frac{1}{|z-x_{\varepsilon}|^{N-2}}dz\Big)\Big]\\&
=O\Big(\frac{\varepsilon}{\lambda_{\varepsilon}^{\frac{N-2}{2}}d^{N-2}}\Big).
\end{split}
\end{equation}
It follows from \eqref{ee29}-\eqref{G3} that the inequality \eqref{L1}.

To prove \eqref{L2}, we know
\begin{equation*}
\nabla u_{\varepsilon}(x)=\int_{\Omega}\nabla_{x}G(x,z)\Big(\displaystyle{\big(\int_{\Omega}\frac{u_{\varepsilon}^{2_{\mu}^\ast}(\xi)}{|z-\xi|^{\mu}}d\xi\big)u_{\varepsilon}^{2_{\mu}^\ast-1}}(z)+\varepsilon u_{\varepsilon}(z)\Big)dz.
\end{equation*}
Similar to estimate of $u_{\varepsilon}(x)$, we can also obtain the inequality \eqref{L2}.
Hence we finish the proof of Lemma~\ref{Lem2.2}.
\end{proof}
We are going to prove Theorem~\ref{Thm1.2} by applying Lemmas~\ref{Lem21}, \ref{Lem2.2} and the local Pohozaev identity \eqref{eq2.6}.\\
\textbf{Proof of~Theorem~\ref{Thm1.2}.}
We will prove the theorem by excluding the case $x_{0}\in\partial\Omega$. In fact, takeing $d_{\varepsilon}=\frac{1}{2}d(x_{\varepsilon},\partial\Omega)$, and by Lemma~\ref{Lem21}, we have
$$\lambda_{\varepsilon}d_{\varepsilon}\rightarrow+\infty,~~~\text{as}~~\varepsilon\rightarrow0.$$
Then, by repeating the similar calculations of \eqref{R151} in Lemma~\ref{Lem53}, we know
\begin{equation}\label{G0}
A_{N,\mu}=\frac{N(N-2)}{A_{H,L}}\int_{\mathbb{R}^N}U_{0,1}^{2^\ast-1}dx
+o\big(1\big).
\end{equation}
By the Hardy-Littlewood-Sobolev inequality, we have
\begin{equation}\label{G23}
\begin{split}	
\int_{B_{d_{\varepsilon}}(x_\varepsilon)}&\int_{\Omega\setminus B_{d_{\varepsilon}}(x_\varepsilon)}(x_j-\xi_j)\frac{|u_{\varepsilon}(\xi)|^{2_{\mu}^\ast}|u_{\varepsilon}(x)|^{2_{\mu}^\ast}}
{|x-\xi|^{\mu+2}}dxd\xi\\&	
=O\Big(\lambda_{\varepsilon}^{2N-\mu}\Big)\int_{B_{d_{\varepsilon}}(x_\varepsilon)}\int_{\Omega\setminus B_{d_{\varepsilon}}(x_\varepsilon)}\frac{1}{\big(1+\lambda_{\varepsilon}|\xi-x_{\varepsilon}|\big)^{2N-\mu}}\frac{1}{|x-\xi|^{\mu+1}}\frac{1}{\big(1+\lambda_{\varepsilon}|x-x_{\varepsilon}|\big)^{2N-\mu}}dxd\xi\\&
=O\Big(\frac{1}{\lambda_{\varepsilon}^{N-\frac{\mu+1}{2}}d_{\varepsilon}^{N-\frac{\mu-1}{2}}}\Big).
\end{split}
\end{equation}
Also, we have
\begin{equation}\label{G231}
\begin{split}
\int_{\partial B_{\delta}(x_\varepsilon)}\int_{ \Omega}\frac{|u_{\varepsilon}(x)|^{2_{\mu}^\ast}|u_{\varepsilon}(\xi)|^{2_{\mu}^\ast}}
{|x-\xi|^{\mu}}\nu_jd\xi ds&\leq C\int_{\partial B_{\delta}(x_\varepsilon)}\int_{ \mathbb{R}^N}\frac{|U_{x_\varepsilon,\lambda_{\varepsilon}}(x)|^{2_{\mu}^\ast}|U_{x_\varepsilon,\lambda_\varepsilon}(\xi)|^{2_{\mu}^\ast}}
{|x-\xi|^{\mu}}\nu_jds\\&=C
\frac{N(N-2)}{A_{H,L}}\int_{\partial B_{\delta}(x_\varepsilon)}U_{x_{\varepsilon},\lambda_{\varepsilon}}^{2^{\ast}}\nu_jds.
\end{split}
\end{equation}
In view of Lemma~\ref{Lem2.2}, we know the estimates \eqref{L1} and \eqref{L2} hold on $\partial B_{d_{\varepsilon}}(x_{\varepsilon})$.
By\eqref{G23} and \eqref{G231},  taking $\Omega^{\prime}=B_{d_{\varepsilon}}(x_\varepsilon)$ in the local Pohozaev identity \eqref{PH}  in Lemma \ref{Lem2.1}, we have
\begin{equation*}
\begin{split}
\int_{B_{d_{\varepsilon}}(x_{\varepsilon})}\frac{\partial G(x,x_{\varepsilon})}{\partial x_j}\frac{\partial G(x,x_{\varepsilon})}{\partial\nu}ds-\frac{1}{2}\int_{\partial B_{d_{\varepsilon}}(x_{\varepsilon})}|\nabla G(x,x_{\varepsilon})|^2\nu_jds
=O\Big(\frac{\varepsilon}{d_{\varepsilon}^{N-1}}+\frac{1}{\lambda_{\varepsilon}d_{\varepsilon}^N}\Big).
\end{split}
\end{equation*}
Since we have the identity (see \cite{Cao-Peng-Yan2021})
\begin{equation}\label{GREEN}
	\begin{split}
		\int_{\partial B_{\delta}(x_{\varepsilon})}\frac{\partial G(x,x_{\varepsilon})}{\partial x_j}\frac{\partial G(x_{\varepsilon},x)}{\partial\nu}ds-\frac{1}{2}\int_{\partial B_{\delta}(x_{\varepsilon})}|\nabla G(x,x_{\varepsilon})|^2\nu_jds
		=\frac{\partial H(x,x_{\varepsilon})}{\partial x_j}\big|_{x=x_{\varepsilon}},
	\end{split}
\end{equation}
then we know
\begin{equation}\label{fai}
\frac{\partial H(x,x_{\varepsilon})}{\partial x_j}\Big|_{x=x_{\varepsilon}}=O\Big(\frac{\varepsilon}{d_{\varepsilon}^{N-1}}+\frac{1}{\lambda_{\varepsilon}d_{\varepsilon}^N}\Big),~~j=1,\dots,N,
\end{equation}
However, recall the following estimate established in\cite{Rey-1990,Cao-Peng-Yan2021}
\begin{equation}\label{221}
	\nabla\mathcal{R}(x_{\varepsilon})= \frac{2}{\omega_N}\frac{1}{(2d_{\varepsilon})^{N-1}}\frac{\tilde{x}-x_{\varepsilon}}{d_{\varepsilon}}+O\big(\frac{1}{d_{\varepsilon}^{N-2}}\big),
	~~~~\text{as}~~d_{\varepsilon}\rightarrow0,
\end{equation}
where $\tilde{x}\in\partial\Omega$ is the unique point, satisfying $d(x_{\varepsilon},\partial\Omega)=|x_{\varepsilon}-\tilde{x}|$.
The estimates in \eqref{fai} and \eqref{221},
lead to a contradiction as $\varepsilon\rightarrow0$ immediatelly.

From the above arguments, we know there must hold $x_0\in\Omega$. We have the following estimate,  the proof is postponed to Lemma~\ref{Lem53} in the Appendix,
\begin{equation}\label{G0}
A_{N,\mu}=\frac{N(N-2)}{A_{H,L}}\int_{\mathbb{R}^N}U_{0,1}^{2^\ast-1}dx
+O\Big(\frac{1}{\lambda_{\varepsilon}^2}\Big).
\end{equation}
By Lemma~\ref{Lem2.2} and \eqref{G0}, we get by taking $\Omega^{\prime}=B_{\delta}(x_\varepsilon)$ in the local Pohozaev identity \eqref{PH}  in Lemma \ref{Lem2.1},
\begin{equation}\label{G1}
\text{LHS of}~\eqref{PH}=\int_{\partial B_{\delta}(a_{\varepsilon})}\frac{\partial G(x,x_{\varepsilon})}{\partial x_j}\frac{\partial G(x,x_{\varepsilon})}{\partial\nu}ds-\frac{1}{2}\int_{\partial B_{\delta}(x_{\varepsilon})}|\nabla G(x,x_{\varepsilon})|^2\nu_jds+O\Big(\varepsilon+\frac{1}{\lambda_{\varepsilon}}\Big).
\end{equation}
On the other hand, by Hardy-Littlewood-Sobolev inequality, we can also find
	\begin{equation*}
		\begin{split}	\int_{B_{\delta}(x_\varepsilon)}&\int_{\Omega\setminus B_{\delta}(x_\varepsilon)}(x_j-\xi_j)\frac{|u_{\varepsilon}(\xi)|^{2_{\mu}^\ast}|u_{\varepsilon}(x)|^{2_{\mu}^\ast}}
{|x-\xi|^{\mu+2}}dxd\xi=O\Big(\int_{B_{\delta}(x_\varepsilon)}\int_{\Omega\setminus B_{\delta}(x_\varepsilon)}\frac{|u_{\varepsilon}(\xi)|^{2_{\mu}^\ast}|u_{\varepsilon}(x)|^{2_{\mu}^\ast}}
{|x-\xi|^{\mu+1}}dxd\xi\Big)\\&	
=O\Big(\lambda_{\varepsilon}^{2N-\mu}\Big)\int_{B_{\delta}(x_\varepsilon)}\int_{\Omega\setminus B_{\delta}(x_\varepsilon)}\frac{1}{\big(1+\lambda_{\varepsilon}|\xi-x_{\varepsilon}|\big)^{2N-\mu}}\frac{1}{|x-\xi|^{\mu+1}}\frac{1}{\big(1+\lambda_{\varepsilon}|x-x_{\varepsilon}|\big)^{2N-\mu}}dxd\xi\\&
=O\Big(\frac{1}{\lambda_{\varepsilon}^{N-\frac{\mu+1}{2}}}\Big).
\end{split}
\end{equation*}
It follows from \eqref{PH} and \eqref{G1} that
\begin{equation}\label{G22}
\begin{split}
\int_{\partial B_{\delta}(x_{\varepsilon})}\frac{\partial G(x,x_{\varepsilon})}{\partial x_j}\frac{\partial G(x,x_{\varepsilon})}{\partial\nu}ds-\frac{1}{2}\int_{\partial B_{\delta}(x_{\varepsilon})}|\nabla G(x,x_{\varepsilon})|^2\nu_jds
=O\Big(\varepsilon+\frac{1}{\lambda_{\varepsilon}}\Big).
\end{split}
\end{equation}

Hence \eqref{G22} and \eqref{GREEN} imply that
\begin{equation}\label{HR}
\frac{\partial H(x,x_{\varepsilon})}{\partial x_j}\Big|_{x=x_{\varepsilon}}=O\Big(\varepsilon+\frac{1}{\lambda_{\varepsilon}}\Big),~~j=1,\dots,N.
\end{equation}
This means that $\nabla\mathcal{R}(x_{0})=0$ as $\varepsilon\rightarrow0$.
Thus the conclusion follows.
$\hfill{} \Box$

\section{Local uniqueness of the blow-up  solutions}

\subsection{Estimates for blow-up solutions and Green's function}
Before we prove that local uniqueness of such type of solutions, we need some preparations. The following lemma plays a crucial role.
\begin{Lem}\label{Lem5.2}
Assume that $N\geq6$, $\mu\in(0,4]$ and $u_{\varepsilon}$ is a sequence of solutions of problem \eqref{NC} in $H_{0}^1(\Omega)$. Then we have
\begin{equation}\label{L11} u_{\varepsilon}(x)=\frac{G(x,x_{\varepsilon})}{\lambda_{\varepsilon}^{\frac{N-2}{2}}}A_{N,\mu}
+O\Big(\frac{\ln\lambda_{\varepsilon}}{\lambda_{\varepsilon}^{\frac{N+2}{2}}}\Big)~~~~~~~~~~~\text{in}~~ \Omega\setminus B_{\tau}(x_\varepsilon),
\end{equation}
and
\begin{equation}\label{L111} \nabla u_{\varepsilon}(x)=\frac{\nabla_{x}G(x,x_{\varepsilon})}{\lambda_{\varepsilon}^{\frac{N-2}{2}}}A_{N,\mu}
	+O\Big(\frac{\ln\lambda_{\varepsilon}}{\lambda_{\varepsilon}^{\frac{N+2}{2}}}\Big)~~~~~~~\text{in}~~ \Omega\setminus B_{\tau}(x_\varepsilon),
\end{equation}
where $A_{N,\mu}$ from Lemma~\ref{Lem2.2} and $d=|x_{\varepsilon}-x|$.
\end{Lem}
\begin{proof}
We know that the solution of \eqref{NC} can be rewritten as:
\begin{equation}\label{G7}
u_{\varepsilon}(x)=\int_{\Omega}\int_{\Omega}\frac{u_{\varepsilon}^{2_{\mu}^\ast}(\xi)u_{\varepsilon}^{2_{\mu}^\ast-1}(z)G(x,z)}{|x-\xi|^{\mu}}d\xi dz+\varepsilon\int_{\Omega} u_{\varepsilon}(z)G(x,z)dz.
\end{equation}

From the estimate in \eqref{L71}, we know
\begin{equation}\label{5.2}
\begin{split}
\int_{\Omega\setminus B_{\frac{\tau}{2}}(x_{\varepsilon})}\int_{B_{\frac{\tau}{2}}(x_{\varepsilon}) }\frac{u_{\varepsilon}^{2_{\mu}^\ast}(\xi)u_{\varepsilon}^{2_{\mu}^\ast-1}(z)G(x,z)}{|x-\xi|^{\mu}}d\xi dz
=O\Big(\frac{1}{\lambda_{\varepsilon}^{\frac{N+2}{2}}}\Big)
\end{split}
\end{equation}
and
\begin{equation}\label{G5}
\begin{split}
\int_{\Omega\setminus B_{\frac{\tau}{2}}(x_{\varepsilon})}\int_{\Omega\setminus B_{\frac{\tau}{2}}(x_{\varepsilon}) }\frac{u_{\varepsilon}^{2_{\mu}^\ast}(\xi)u_{\varepsilon}^{2_{\mu}^\ast-1}(z)G(x,z)}{|x-\xi|^{\mu}}d\xi dz
=O\Big(\frac{1}{\lambda_{\varepsilon}^{\frac{N+2}{2}}}\Big).
\end{split}
\end{equation}
Since
\begin{equation*}
G(x,z)=G(x,x_{\varepsilon})+\big\langle\nabla G(x,x_\varepsilon),z-x_{\varepsilon}\big\rangle+O\big(|z-x_{\varepsilon}|^2\big),
\end{equation*}
then we know
\begin{equation}\label{L7}
\begin{split}
\int_{ B_{\frac{\tau}{2}}(x_{\varepsilon})}\int_{B_{\frac{\tau}{2}}(x_{\varepsilon})}&\frac{u_{\varepsilon}^{2_{\mu}^\ast}(\xi)u_{\varepsilon}^{2_{\mu}^\ast-1}(z)G(x,z)}{|x-\xi|^{\mu}}d\xi dz=G(x,x_{\varepsilon})\int_{ B_{\frac{\tau}{2}}(x_{\varepsilon})}\int_{B_{\frac{\tau}{2}}(x_{\varepsilon})}\frac{u_{\varepsilon}^{2_{\mu}^\ast}(\xi)u_{\varepsilon}^{2_{\mu}^\ast-1}(z)}{|x-\xi|^{\mu}}d\xi dz\\&~~~+\underbrace{\int_{ B_{\frac{\tau}{2}}(x_{\varepsilon})}\int_{B_{\frac{\tau}{2}}(x_{\varepsilon})}\frac{u_{\varepsilon}^{2_{\mu}^\ast}(\xi)u_{\varepsilon}^{2_{\mu}^\ast-1}(z)\big\langle\nabla G(x,x_\varepsilon),z-x_{\varepsilon}\big\rangle}{|x-\xi|^{\mu}}d\xi dz}\limits_{:=A_1}\\&~~~+\underbrace{\int_{ B_{\frac{\tau}{2}}(x_{\varepsilon})}\int_{B_{\frac{\tau}{2}}(x_{\varepsilon})}\frac{u_{\varepsilon}^{2_{\mu}^\ast}(\xi)u_{\varepsilon}^{2_{\mu}^\ast-1}(x)|x-x_{\varepsilon}|^2}{|z-\xi|^{\mu}}d\xi dz}\limits_{:=A_2}\\&
=\frac{G(x,x_{\varepsilon})}{\lambda_{\varepsilon}^{\frac{N-2}{2}}}A_{N,\mu}+
O\Big(\frac{\ln\lambda_{\varepsilon}}{\lambda_{\varepsilon}^{\frac{N+2}{2}}}\Big),
\end{split}
\end{equation}
where
\begin{equation*}
\begin{split}
A_1
&=O\Big(\int_{B_{\frac{\tau}{2}}(x_{\varepsilon})}\int_{\mathbb{R}^N}\frac{ U_{x_{\varepsilon},\lambda_{\varepsilon}}^{2_{\mu}^\ast}(\xi) PU_{x_{\varepsilon},\lambda_{\varepsilon}}^{2_{\mu}^\ast-1}(z)\big\langle\nabla G(x,x_\varepsilon),z-x_{\varepsilon}\big\rangle}{|x-\xi|^{\mu}}d\xi dz\Big)\\&~~~+O\Big(\int_{B_{\frac{\tau}{2}}(x_{\varepsilon})}\int_{\mathbb{R}^N}\frac{ U_{x_{\varepsilon},\lambda_{\varepsilon}}^{2_{\mu}^\ast}(\xi) PU_{x_{\varepsilon},\lambda_{\varepsilon}}^{2_{\mu}^\ast-2}(\xi)w_{\varepsilon}\big|z-x_{\varepsilon}\big|}{|x-\xi|^{\mu}}d\xi dz\Big)\\&
= O\Big(\frac{N(N-2)}{A_{H,L}}\int_{B_{\frac{\tau}{2}}(x_{\varepsilon})} U_{x_{\varepsilon},\lambda_{\varepsilon}}^{2^\ast-1}\langle\nabla G(x,x_\varepsilon),z-x_{\varepsilon}\rangle dz+\frac{N(N-2)}{\mathcal{A}_{H,L}}\int_{B_{\frac{\tau}{2}}(x_{\varepsilon})} U_{x_{\varepsilon},\lambda_{\varepsilon}}^{2^\ast-2}w_{\varepsilon}\big|z-x_{\varepsilon}\big|dz\Big)
\\&=O\Big(\int_{B_{\frac{\tau}{2}}(x_{\varepsilon})} U_{x_{\varepsilon},\lambda_{\varepsilon}}^{2^\ast-2}(z)w_{\varepsilon} |z-x_{\varepsilon}|dz\Big)
=O\Big(\frac{1}{\lambda_{\varepsilon}^2}\big(\int_{0}^{\frac{\tau\lambda_{\varepsilon}}{2}} \frac{r^{N-1}}{(1+r^2)^{\frac{6N}{N+2}}}dr\big)^{\frac{N+2}{2N}}\|w_{\varepsilon} \|_{H_{0}^1}\Big)\\&=O\Big(\frac{1}{\lambda_{\varepsilon}^2}\|w_{\varepsilon} \|_{H_{0}^1}\Big),
\end{split}
\end{equation*}
since $G(x,x_{\varepsilon})=G(x_{\varepsilon},x)$. Similarly, we can calculate that
\begin{equation*}\label{G4}
\begin{split}
A_2=&O\Big(\int_{ B_{\frac{\tau}{2}}(x_{\varepsilon})}\int_{\mathbb{R}^N}\frac{U_{x_\varepsilon,\lambda_{\varepsilon}}^{2_{\mu}^\ast}(\xi)U_{x_\varepsilon,\lambda_{\varepsilon}}^{2_{\mu}^\ast-1}(z)|z-x_{\varepsilon}|^2}{|x-\xi|^{\mu}}d\xi dz\Big)\\&
=O\Big(\frac{N(N-2)}{\mathcal{A}_{H,L}}\int_{ B_{\frac{\tau}{2}}(x_{\varepsilon})}U_{x_\varepsilon,\lambda_{\varepsilon}}^{2^\ast-1}(z)|z-x_{\varepsilon}|^2dz\Big)\\&
=O\Big(\frac{\ln\lambda_{\varepsilon}}{\lambda_{\varepsilon}^{\frac{N+2}{2}}}\Big).
\end{split}
\end{equation*}
Finally, similar to the calculation of \eqref{GU1}, from $\lambda_{\varepsilon}\sim\varepsilon^{-\frac{1}{N-4}}$ from \cite[subsection~2.2]{Yang-Zhao-2022}, we obtain
\begin{equation}\label{G9}
\begin{split}
\varepsilon&\int_{\Omega}G(x,z)u_{\varepsilon}(z)dz
=O\Big(\frac{1}{\lambda_{\varepsilon}^{\frac{N+2}{2}}}\Big).
\end{split}
\end{equation}
Then \eqref{G7}, \eqref{5.2}, \eqref{G5}, \eqref{L7} and \eqref{G9} imply that \eqref{L11}. we Finally, we get \eqref{L111} from the fact that
\begin{equation}\label{G8}
	\nabla u_{\varepsilon}(x)=\int_{\Omega}\int_{\Omega}\frac{u_{\varepsilon}^{2_{\mu}^\ast}(\xi)u_{\varepsilon}^{2_{\mu}^\ast-1}(z)\nabla_{x} G(x,z)}{|x-\xi|^{\mu}}d\xi dz+\varepsilon\int_{\Omega} u_{\varepsilon}(z)\nabla_xG(x,z)dz.
\end{equation}
Hence the proof is finished.
\end{proof}

\begin{Lem}\label{Lem5.3}
Assume that $N\geq6$, $\mu\in(0,4]$ and  $u_{\varepsilon}$ is a solution of \eqref{NC}. Then we have
\begin{equation}\label{R1} \nabla\mathcal{R}(x_{\varepsilon})=O\Big(\frac{\ln\lambda_{\varepsilon}}{\lambda_{\varepsilon}^{2}}\Big)
\end{equation}
and
\begin{equation}\label{R2}
\varepsilon=\frac{1}{\lambda_{\varepsilon}^{N-4}}\Big(A_{0}+\frac{1}{\lambda_{\varepsilon}^{2}}\Big),
\end{equation}
where $A_0$ is a strictly positive constant.
\end{Lem}
\begin{proof}
 Similar to the arguments of \eqref{HR}, applying the Pohozaev identity in Lemma~\ref{eq2.6} and  Lemma~\ref{Lem5.2},  we can obtain \eqref{R1} by taking $\Omega^{\prime}=B_{\tau}(x_\varepsilon)$ . Next we shall prove \eqref{R2}.  By Lemma~\ref{Lem2.2}, we find
\begin{equation}\label{R3}
u_{\varepsilon}(x)=\frac{G(x,x_{\varepsilon})}{\lambda_{\varepsilon}^{\frac{N-2}{2}}}A_{N,\mu}
+O\Big(\frac{\varepsilon }{\lambda_{\varepsilon}^{\frac{N-2}{2}}}+\frac{1}{\lambda_{\varepsilon}^{\frac{N}{2}}}\Big),x\in\Omega\setminus B_{\tau}(x_{\varepsilon})
\end{equation}
and
\begin{equation}\label{R4}
\nabla u_{\varepsilon}(x)=\frac{\nabla G(x,x_{\varepsilon})}{\lambda_{\varepsilon}^{\frac{N-2}{2}}}A_{N,\mu}
+O\Big(\frac{\varepsilon }{\lambda_{\varepsilon}^{\frac{N-2}{2}}}+\frac{1}{\lambda_{\varepsilon}^{\frac{N}{2}}}\Big),~~~x\in\Omega\setminus B_{\tau}(x_{\varepsilon}),
\end{equation}
By \eqref{R3} and \eqref{R4},  taking $\Omega^{\prime}=B_{\tau}(x_\varepsilon)$ in the local Pohozaev identity \eqref{eq2.6}, we obtain
 \begin{equation*}
\begin{split}
\int_{B_{\tau}(x_\varepsilon)}\int_{\Omega\setminus B_{\tau}(x_\varepsilon)}\frac{u_{\varepsilon}^{2_{\mu}^\ast}(x)u_{\varepsilon}^{2_{\mu}^\ast}(\xi)}
{|x-\xi|^{\mu}}d\xi dx&\leq C\Big(\int_{B_{\tau}(x_\varepsilon)}U_{x_\varepsilon,\lambda_{\varepsilon}}^{2^\ast}(x)dx\Big)^{\frac{2N-\mu}{2N}}\Big(\int_{\Omega\setminus B_{\tau}(x_\varepsilon)}u_{\varepsilon}^{2^\ast}(x)dx\Big)^{\frac{2N-\mu}{2N}}\\&
=O\Big(\frac{1}{\lambda_{\varepsilon}^N}\Big).
\end{split}
\end{equation*}
Similarly, we can also calculate that
\begin{equation*}
\begin{split}
&\int_{\partial B_{\tau}(x_\varepsilon)}\int_{ \Omega}\frac{u_{\varepsilon}^{2_{\mu}^\ast}(x)u_{\varepsilon}^{2_{\mu}^\ast}(\xi)}
{|x-\xi|^{\mu}}\big\langle x-x_\varepsilon,\nu\big\rangle d\xi ds=O\Big(\frac{1}{\lambda_{\varepsilon}^N}\Big),\int_{B_{\tau}(x_\varepsilon)}\int_{\Omega\setminus B_{\tau}(x_\varepsilon)}x\cdot(x-\xi)\frac{u_{\varepsilon}^{2_{\mu}^\ast}(x)u_{\varepsilon}^{2_{\mu}^\ast}(\xi)}
{|x-\xi|^{\mu+2}}d\xi dx=O\Big(\frac{1}{\lambda_{\varepsilon}^N}\Big),\\&
\int_{\partial B_{\tau}(x_\varepsilon)}u_{\varepsilon}^2\big\langle x-x_\varepsilon,\nu\big\rangle ds=O\Big(\frac{\varepsilon}{\lambda_{\varepsilon}^{N-2}}\Big).
\end{split}
\end{equation*}
Inserting \eqref{R3} and \eqref{R4} into the Pohozaev identity \eqref{eq2.6}, we know
\begin{equation}\label{R5}
\begin{split}
\frac{A_{N,\mu}^2}{\lambda_{\varepsilon}^{N-2}}&\bigg[-\int_{\partial B_{\tau}(x_{\varepsilon})}\frac{\partial G(x,x_{\varepsilon})}{\partial\nu}\big\langle x-x_\varepsilon,\nabla G(x,x_{\varepsilon})\big\rangle ds+\frac{1}{2}\int_{\partial B_{\tau}(x_{\varepsilon})}|\nabla G(x,x_{\varepsilon})|^2\big\langle x-x_\varepsilon,\nu\big\rangle ds\\&-\frac{N-2}{2}\int_{\partial B_{\tau}(x_{\varepsilon})} \frac{\partial  G(x,x_{\varepsilon})}{\partial\nu} G(x,x_{\varepsilon})\bigg]\\&
=-\varepsilon\int_{B_{\tau}(x_{\varepsilon})} u_{\varepsilon}^2(x)dx+O\Big(\frac{\varepsilon }{\lambda_{\varepsilon}^{N-2}}+\frac{1}{\lambda_{\varepsilon}^{N}}\Big).
\end{split}
\end{equation}
Applying the following identity (see \cite{Cao-2021})
\begin{equation*}
\begin{split}
-\int_{\partial B_{\tau}(x_{\varepsilon})}&\frac{\partial G(x,x_{\varepsilon})}{\partial \nu}\big\langle x-x_{\varepsilon},\nabla G(x,x_{\varepsilon})\big\rangle ds+\frac{1}{2}\int_{\partial B_{\tau}(x_{\varepsilon})}|\nabla G(x,x_{\varepsilon})|^2\big\langle x-x_{\varepsilon},\nu\big\rangle ds\\&-\frac{N-2}{2}\int_{\partial B_{\tau}(x_{\varepsilon})}G(x,x_{\varepsilon})\frac{\partial G(x,x_{\varepsilon})}{\partial\nu}ds=-\frac{N-2}{2}\mathcal{R}(x_{\varepsilon}),
\end{split}
\end{equation*}
we obtain from \eqref{R5} that
\begin{equation}\label{HH1}
\frac{(N-2)H(x_{\varepsilon},x_{\varepsilon})A_{N,\mu}^2}{2\lambda_{\varepsilon}^{N-2}}=\varepsilon\int_{B_{\tau}(x_{\varepsilon})} u_{\varepsilon}^2+O\Big(\frac{\varepsilon }{\lambda_{\varepsilon}^{N-2}}+\frac{1}{\lambda_{\varepsilon}^{N}}\Big).
\end{equation}
On one hand, by Lemma~\ref{Lem53}, we know
\begin{equation}\label{R15}
A_{N,\mu}=\frac{N(N-2)}{A_{H,L}}\int_{\mathbb{R}^N}U_{0,1}^{2^\ast-1}
+O\Big(\frac{1}{\lambda^2}\Big).
\end{equation}
On the other hand, by $PU_{x_{\varepsilon},\lambda_{\varepsilon}}\leq U_{x_{\varepsilon},\lambda_{\varepsilon}}$, we have
\begin{equation}\label{R6}
\begin{split}
\varepsilon\int_{B_{\tau}(x_{\varepsilon})} u_{\varepsilon}^2(x)&dx=\varepsilon\Big[\int_{B_{\tau}(x_{\varepsilon})}(PU_{x_{\varepsilon},\lambda_{\varepsilon}}(x))^2dx+O\big(\int_{B_{\tau}(x_{\varepsilon})}PU_{x_{\varepsilon},\lambda_{\varepsilon}}w_{\varepsilon}+\|w_{\varepsilon}\|_{H_{0}^1}^2\big)\Big]\\&
=\frac{\varepsilon}{\lambda^2}\int_{\mathbb{R}^N}U_{0,1}^2dx+O\Big(\frac{\varepsilon}{\lambda^{N-2}}\Big)+
O\Big(\big(\int_{B_{\tau}(x_{\varepsilon})}U_{x_{\varepsilon},\lambda_{\varepsilon}}^2dx\big)^{\frac{1}{2}}\|w_{\varepsilon}\|_{H_{0}^1}+\|w_{\varepsilon}\|_{H_{0}^1}^2\Big)\\&
=\frac{\varepsilon}{\lambda^2}\int_{\mathbb{R}^N}U_{0,1}^2dx+O\Big(\frac{\varepsilon}{\lambda^{4}}\Big).
\end{split}
\end{equation}
Therefore, together with \eqref{HH1}, \eqref{R15} and \eqref{R6}, we can deduce
\begin{equation}\label{RR}
\frac{N^2(N-2)^3\mathcal{R}(x_{\varepsilon})}{2A_{H,L}\lambda_{\varepsilon}^{N-2}}\Big(\int_{\mathbb{R}^N}U_{0,1}^{2^\ast-1}dx
+O(\frac{1}{\lambda^2})\Big)^2=\frac{\varepsilon}{\lambda_{\varepsilon}^2}\Big(\int_{\R^N}U_{0,1}^2dx+O(\frac{1}{\lambda^2})\Big)+O\Big(\frac{\varepsilon }{\lambda_{\varepsilon}^{N-2}}+\frac{1}{\lambda_{\varepsilon}^{N}}\Big),
\end{equation}
which implies that \eqref{R2} is true.

\end{proof}

\subsection{The local uniqueness result}
The purpose of this subsection is devoted to complete the proof of Theorem~\ref{Thm1.4}. There are some prelimilaries to be done before we go into the proof. First of all, we let $u_{\varepsilon}^{(1)}$ and $u_{\varepsilon}^{(2)}$ be two different solutions of \eqref{eq1.1}. We will use $x_{\varepsilon}^{(j)}$ and $\lambda_{\varepsilon}^{(j)}$ to denote the center and the height of the bubbles appearing in $u_{\varepsilon}^{(j)}(j=1,2)$, respectively.
\par
Let
\begin{equation}\label{6.1}
\eta_{\varepsilon}(x):=\frac{u_{\varepsilon}^{(1)}(x)-u_{\varepsilon}^{(2)}(x)}{\|u_{\varepsilon}^{(1)}(x)-u_{\varepsilon}^{(2)}(x)\|_{L^{\infty}}},
\end{equation}
then $\eta_{\varepsilon}(x)$ satisfies $\|\eta_{\varepsilon}\|_{L^{\infty}}$=1 and
\begin{equation}\label{6.2}
-\Delta\eta_{\varepsilon}(x)=f\big(x,u_{\varepsilon}^{(1)},u_{\varepsilon}^{(2)}\big),
\end{equation}
where
\begin{equation*}
f\big(x,u_{\varepsilon}^{(1)},u_{\varepsilon}^{(2)}\big)=(2_{\mu}^{\ast}-1)\Big(\int_{\Omega}\frac{\big(u_{\varepsilon}^{(1)}(\xi)\big)^{2_{\mu}^\ast}}{|x-\xi|^{\mu}}d\xi\Big)C_{\varepsilon}(x)\eta_{\varepsilon}(x)+ 2_{\mu}^\ast\Big(\int_{\Omega}\frac{D_{\varepsilon}(\xi)\eta_{\varepsilon}(\xi)}{|x-\xi|^{\mu}}d\xi\Big)\big(u_{\varepsilon}^{(2)}(x)\big)^{2_{\mu}^\ast-1}+\varepsilon\eta_{\varepsilon}
\end{equation*}
with
\begin{equation}\label{l6.4}
C_{\varepsilon}(x)=\int_{0}^1\Big(tu_{\varepsilon}^{(1)}(x)+(1-t)u_{\varepsilon}^{(2)}(x)\Big)^{2_{\mu}^{\ast}-2}dt,~~~D_{\varepsilon}(\xi)=\int_{0}^1\Big(tu_{\varepsilon}^{(1)}(\xi)+(1-t)u_{\varepsilon}^{(2)}(\xi)\Big)^{2_{\mu}^{\ast}-1}dt.
\end{equation}
\begin{Lem}\label{Lem5.4}
For $N\geq6$, $\mu\in(0,4]$, it holds
\begin{equation}\label{AA02}
|x_{\varepsilon}^{(1)}-x_{\varepsilon}^{(2)}|=O\Big(\frac{\ln\lambda_{\varepsilon}^{(1)}}{(\lambda_{\varepsilon}^{(1)})^{2}}\Big)~~~\text{and}~~~|\lambda_{\varepsilon}^{(1)}-\lambda_{\varepsilon}^{(2)}|=O\Big(\frac{\ln\lambda_{\varepsilon}^{(1)}}{(\lambda_{\varepsilon}^{(1)})^{2}}\Big).
\end{equation}
\end{Lem}
\begin{proof}
First we remark that
$$
\mathcal{R}(x_{\varepsilon})=\mathcal{R}(x_0)+\big\langle\nabla\mathcal{R}(x_0),x_{\varepsilon}-x_0\big\rangle+O\big(\nabla^2\mathcal{R}(x_0)|x_{\varepsilon}-x_0|^2\big).
$$
Combining \eqref{R1} and $x_0$ is a nondegerate critical point of Robin function $\mathcal{R}$, we see that for $N\geq6$
\begin{equation}\label{AA01}
|x_{\varepsilon}-x_0|=O\Big(\frac{\ln\lambda_{\varepsilon}}{\lambda_{\varepsilon}^{2}}\Big),
\end{equation}
A direct computations, we get \eqref{AA02} from \eqref{R2} and \eqref{AA01}.
\end{proof}

\begin{Lem}\label{L4.2} For $N\geq6$, $\mu\in(0,4)$, it holds
\begin{equation}\label{6.91}
C_{\varepsilon}(x)=U_{x_{\varepsilon}^{(1)},\lambda_{\varepsilon}^{(1)}}^{2_{\mu}^{\ast}-2}+O\Big(\frac{\ln\lambda_{\varepsilon}^{(1)}}{\lambda_{\varepsilon}^{(1)}}\Big)U_{x_{\varepsilon}^{(1)},\lambda_{\varepsilon}^{(1)}}^{2_{\mu}^{\ast}-2}
+O\Big(\sum_{j=1}^{2}|w_{\varepsilon}^{(j)}|^{2_{\mu}^{\ast}-2}\Big)
\end{equation}
and
\begin{equation}\label{6.102}
D_{\varepsilon}(x)=U_{x_{\varepsilon}^{(1)},\lambda_{\varepsilon}^{(1)}}^{2_{\mu}^{\ast}-1}+O\Big(\frac{\ln\lambda_{\varepsilon}^{(1)}}{\lambda_{\varepsilon}^{(1)}}\Big)U_{x_{\varepsilon}^{(1)},\lambda_{\varepsilon}^{(1)}}^{2_{\mu}^{\ast}-1}
+O\Big(\sum_{j=1}^{2}|w_{\varepsilon}^{(j)}|^{2_{\mu}^{\ast}-1}\Big).
\end{equation}
\end{Lem}
\begin{proof}
In view of Lemma~\ref{Lem5.4}, we first remark that
\begin{equation}\label{UUA}
\begin{split}
|U_{a_{\varepsilon}^{(1)},\lambda_{\varepsilon}^{(1)}}(x)-U_{x_{\varepsilon}^{(2)},\lambda_{\varepsilon}^{(2)}}(x)|&=O\Big(|x_{\varepsilon}^{(1)}-x_{\varepsilon}^{(2)}|\cdot \big(\nabla_{x} U_{x,\lambda_{\varepsilon}^{(1)}}(x)|_{x=x_{\varepsilon}^{(1)}}\big)+|\lambda_{\varepsilon}^{(1)}-\lambda_{\varepsilon}^{(2)}|\cdot \big(\nabla_{x} U_{x_{\varepsilon}^{(1)},\lambda}(x)|_{\lambda=\lambda_{\varepsilon}^{(1)}}\big)\Big)
\\&=
O\Big( U_{x_{\varepsilon}^{(1)},\lambda_{\varepsilon}^{(1)}}\Big(\lambda_{\varepsilon}^{(1)}|x_{\varepsilon}^{(1)}-x_{\varepsilon}^{(2)}|+\frac{\lambda_{\varepsilon}^{(1)}-\lambda_{\varepsilon}^{(2)}}{\lambda_{\varepsilon}^{(1)}}\Big)\Big)\\&
=O\Big(\frac{\ln\lambda_{\varepsilon}^{(1)}}{\lambda_{\varepsilon}^{(1)}}\Big)U_{x_{\varepsilon}^{(1)},\lambda_{\varepsilon}^{(1)}}(x),
\end{split}
\end{equation}
which implies that
\begin{equation}\label{Uu121}
u_{\varepsilon}^{(1)}-u_{\varepsilon}^{(2)}=O\Big(\frac{\ln\lambda_{\varepsilon}^{(1)}}{\lambda_{\varepsilon}^{(1)}}\Big)U_{x_{\varepsilon}^{(1)},\lambda_{\varepsilon}^{(1)}}(x)+
O\Big(\sum_{j=1}^{2}|w_{\varepsilon}^{(j)}|\Big).
\end{equation}
Then \eqref{Uu121} can deduce that \eqref{6.91} and \eqref{6.102}.
\end{proof}

From \cite{WY-2010}, we have the following estimate:
\begin{Lem}\label{Lem6.1}
For any constant $0<\sigma\leq N-2$, there is a constant $C>0$ such that
\begin{equation*}
\int_{\mathbb{R}^N}\frac{1}{|y-x|^{N-2}}\frac{1}{(1+|x|)^{2+\sigma}}dx\leq
\begin{cases}
\frac{C}{(1+|y|)^\sigma},~~~~~~~~\text{if}~~\sigma<N-2,\\
\frac{C}{(1+|y|)^\sigma}\ln|y|,~~~\text{if}~~\sigma=N-2.
\end{cases}
\end{equation*}
\end{Lem}\begin{Lem}\label{L6.2}
For any constant $\sigma\geq N-2-\frac{\mu}{2}$ and $\mu\in(0,4]$, there is a constant $C>0$ such that
\begin{equation*}
\int_{\mathbb{R}^N}\frac{1}{|y-x|^{\frac{2N(N-2)}{2N-\mu}}}\frac{1}{(1+|x|)^{\frac{2N(2+\sigma)}{2N-\mu}}}dx\leq
\begin{cases}
\frac{C}{(1+|y|)^\frac{N(2\sigma+\mu)}{2N-\mu}},~~~~~~~~\text{if}~~\sigma\geq N-2-\frac{\mu}{2}~~\text{and}~~0<\mu<4,\\
\frac{C}{(1+|y|)^\frac{N(2\sigma+\mu)}{2N-\mu}}\ln|y|,~~~\text{if}~~\sigma=N-2-\frac{\mu}{2}~~\text{and}~~\mu=4.
\end{cases}
\end{equation*}
\end{Lem}
\begin{proof}
We just need to obtain the estimate for $|y|\geq2$, the other is similar. Let $d=\frac{1}{2}|y|$. Then we have
\begin{equation*}
\begin{split}
\int_{B_{d}(0)}\frac{1}{|y-x|^{\frac{2N(N-2)}{2N-\mu}}}\frac{1}{(1+|x|)^{\frac{2N(2+\sigma)}{2N-\mu}}}dx&\leq\frac{C}{d^{\frac{2N(N-2)}{2N-\mu}}}\int_{B_{d}(0)}\frac{1}{(1+|x|)^{\frac{2N(2+\sigma)}{2N-\mu}}}dx\\&
\leq\frac{C}{d^{\frac{2N(N-2)}{2N-\mu}}}\frac{1}{(1+d)^{\frac{N(4+2\sigma-2N+\mu)}{2N-\mu}}}\leq\frac{C}{d^{\frac{N(2\sigma+\mu)}{2N-\mu}}},
\end{split}
\end{equation*}
for any $\sigma>N-2-\frac{\mu}{2}$.
And we have
\begin{equation*}
\begin{split}
\int_{B_{d}(y)}\frac{1}{|y-x|^{\frac{2N(N-2)}{2N-\mu}}}\frac{1}{(1+|x|)^{\frac{2N(2+\sigma)}{2N-\mu}}}dx&\leq\frac{C}{d^{\frac{2N(2+\sigma)}{2N-\mu}}}\int_{B_{d}(y)}\frac{1}{|y-x|^{\frac{2N(N-2)}{2N-\mu}}}dx\\&
\leq\frac{C}{d^{\frac{2N(2+\sigma)}{2N-\mu}}}d^{\frac{N(4-\mu)}{2N-\mu}}\leq\frac{C}{d^{\frac{N(2\sigma+\mu)}{2N-\mu}}}.
\end{split}
\end{equation*}
Assume that $x\in\R^N\setminus \big(B_d(0)\cup B_{d}(y)\big)$. Then we know
\begin{equation*}
|x-y|\geq\frac{1}{2}|y|,~~~|x|\geq\frac{1}{2}|y|.
\end{equation*}
Hence by direct computation, we have
\begin{equation*}
\begin{split}
\int_{\R^N\setminus \big(B_d(0)\cup B_{d}(y)\big)}\frac{1}{|y-x|^{\frac{2N(N-2)}{2N-\mu}}}\frac{1}{(1+|x|)^{\frac{2N(2+\sigma)}{2N-\mu}}}dx&\leq\int_{\R^N\setminus \big(B_d(0)\cup B_{d}(y)\big)}\frac{1}{|x|^{\frac{2N(N-2)}{2N-\mu}}}\frac{1}{(1+|x|)^{\frac{2N(2+\sigma)}{2N-\mu}}}dx\\&
\leq\frac{C}{d^{\frac{2N(N-2)}{2N-\mu}}}\int_{\R^N\setminus \big(B_d(0)\cup B_{d}(y)\big)}\frac{1}{(1+|x|)^{\frac{2N(2+\sigma)}{2N-\mu}}}dx\\&\leq\frac{C}{d^{\frac{N(2\sigma+\mu)}{2N-\mu}}},
\end{split}
\end{equation*}
for any $\sigma\geq N-2-\frac{\mu}{2}$. This finishes the proof.

\end{proof}

\begin{Lem}\label{Lem6.2}
For $\eta_{\varepsilon}(x)$ defined by \eqref{6.1}, we have
\begin{equation}\label{e6.3}
\int_{\Omega}\eta_{\varepsilon}(x)dx=\frac{\ln\lambda_{\varepsilon}^{(1)}}{(\lambda_{\varepsilon}^{(1)})^{N-2}}~~\text{and}~~ \eta_{\varepsilon}(x)=\frac{\ln\lambda_{\varepsilon}^{(1)}}{(\lambda_{\varepsilon}^{(1)})^{N-2}},~~\text{in}~~\Omega\setminus B_{\delta}(x_{\varepsilon}^{(1)}),
\end{equation}
where $\delta>0$ is a any small fixed constant.
\end{Lem}
\begin{proof}
By the potential theory, we know
 $$\eta_{\varepsilon}(x)=(2_{\mu}^{\ast}-1)\eta_{\varepsilon,1}(x)+2_{\mu}^{\ast}\eta_{\varepsilon,2}(x)+\varepsilon\int_{\Omega}G(x,z)\eta_{\varepsilon}(z)dz, $$
 where
\begin{equation*}
\begin{split}
&\eta_{\varepsilon,1}(x)=\int_{\Omega}G(x,z)\Big(\int_{\Omega}\frac{(u_{\varepsilon}^{(1)}(\xi))^{2_{\mu}^\ast}}{|x-\xi|^{\mu}}d\xi\Big)C_{\varepsilon}(z)\eta_{\varepsilon}(x)dz,~~\eta_{\varepsilon,2}(x)=\int_{\Omega}G(x,z) \Big(\int_{\Omega}\frac{D_{\varepsilon}(\xi)\eta_{\varepsilon}(\xi)}{|x-\xi|^{\mu}}d\xi\Big)(u_{\varepsilon}^{(2)}(z))^{2_{\mu}^\ast-1}dz.
\end{split}
\end{equation*}
First, we can deduce
\begin{equation}\label{ced}
|C_{\varepsilon}(z)|\leq C\frac{(\lambda_{\varepsilon}^{(1)})^\frac{4-\mu}{2}}{\big(1+\lambda_{\varepsilon}^{(1)}|z-x_{\varepsilon}^{(1)}|\big)^{4-\mu}}~~~\text{and}~~~|D_{\varepsilon}(\xi)|\leq C\frac{(\lambda_{\varepsilon}^{(1)})^\frac{N-\mu+2}{2}}{\big(1+\lambda_{\varepsilon}^{(1)}|\xi-x_{\varepsilon}^{(1)}|\big)^{N-\mu+2}}.
\end{equation}
Combining  $|\eta_{\varepsilon}|\leq1$, \eqref{T2}, \eqref{e29} and Lemma~\ref{Lem6.1}, then we get
\begin{equation}\label{L62}
\begin{split}
\eta_{\varepsilon,1}&\leq C \int_{\Omega}\frac{1}{|z-x|^{N-2}}\Big(\int_{\mathbb{R}^N}\frac{\big(U_{x_{\varepsilon}^{(1)},\lambda_{\varepsilon}^{(1)}}(\xi)\big)^{2_{\mu}^{\ast}}}{|x-\xi|^{\mu}}d\xi\Big)C_{\varepsilon}(z)dz\\&
\leq C \int_{\Omega}\frac{1}{|z-\lambda_{\varepsilon}^{(1)}x|^{N-2}}\frac{1}{\big(1+|z-\lambda_{\varepsilon}^{(1)}x_{\varepsilon}^{(1)}|^2\big)^2}dx\\&
\leq C\int_{\Omega}\frac{1}{|\lambda_{\varepsilon}^{(1)}(x-x_{\varepsilon}^{(1)})-z |^{N-2}}\frac{1}{\big(1+|z|\big)^4}dz\\&
\leq C\frac{1}{\big(1+\lambda_{\varepsilon}^{(1)}|x-x_{\varepsilon}^{(1)}|\big)^2}.\\&
\end{split}
\end{equation}
Next repeating the above process, we know
\begin{equation*}
\begin{split}
\eta_{\varepsilon,1}=O\Big(\frac{1}{\big(1+\lambda_{\varepsilon}^{(1)}|x-x_{\varepsilon}^{(1)}|\big)^4}\Big).
\end{split}
\end{equation*}
Then we can proceed as in the above argument for finite number of times to prove
\begin{equation}\label{L4.5}
\begin{split}
\eta_{\varepsilon,1}=O\Big(\frac{\ln \lambda_{\varepsilon}^{(1)}}{\big(1+\lambda_{\varepsilon}^{(1)}|x-x_{\varepsilon}^{(1)}|\big)^{N-2}}\Big).
\end{split}
\end{equation}
Next, we find
\begin{equation}\label{u12}
\begin{split}
|U_{x_{\varepsilon}^{(1)},\lambda_{\varepsilon}^{(1)}}(x)-U_{x_{\varepsilon}^{(2)},\lambda_{\varepsilon}^{(2)}}(x)|=
O\Big(\frac{\ln\lambda_{\varepsilon}}{(\lambda_{\varepsilon}^{(1)})^{2}}\Big)U_{x_{\varepsilon}^{(1)},\lambda_{\varepsilon}^{(1)}}(x).
\end{split}
\end{equation}
Hence by Lemma~\ref{L6.2} and Hardy-Littlewood-Sobolev inequality, we can calculate that for $d(\Omega):=diam(\Omega)$
\begin{equation*}
\begin{split}
\eta_{\varepsilon,2}&\leq C \int_{\Omega}\Big(\int_{\Omega}\frac{D_{\varepsilon}(\xi)}{|x-\xi|^{\mu}}d\xi\Big)\big(U_{x_{\varepsilon}^{(2)},\lambda_{\varepsilon}^{(2)}}(z)\big)^{2_{\mu}^{\ast}-1}G_{\varepsilon}(x,z)dz\\&
\leq\frac{C}{(\lambda_{\varepsilon}^{(1)})^{\frac{N-2}{2}}}\Big(\int_{0}^{d(\Omega)}\frac{r^{N-1}}{(1+r^2)^{\frac{N(N-\mu+2)}{2N-\mu}}}dr\Big)^{\frac{2N-\mu}{2N}}
\Big(\int_{\Omega}\frac{1}{|z-x|^{\frac{2N(N-2)}{2N-\mu}}}\frac{(\lambda_{\varepsilon}^{(1)})^{\frac{N(N-\mu+2)}{2N-\mu}}}{\big(1+(\lambda_{\varepsilon}^{(1)})^2|z-x_{\varepsilon}^{(1)}|^2\big)^{\frac{N(N-\mu+2)}{2N-\mu}}}dz\Big)^{\frac{2N-\mu}{2N}}\\&
\leq C
\Big(\int_{\Omega}\frac{1}{\big|\lambda_{\varepsilon}^{(1)}(x-x_{\varepsilon}^{(1)})-z\big|^{\frac{2N(N-2)}{2N-\mu}}}\frac{1}{\big(1+|z|\big)^{\frac{2N(N-\mu+2)}{2N-\mu}}}dx\Big)^{\frac{2N-\mu}{2N}}\\&
\leq\frac{C}{\big(1+\lambda_{\varepsilon}^{(1)}|x-x_{\varepsilon}^{(1)}|\big)^N}.
\end{split}
\end{equation*}
Finally, we have
\begin{equation*}
\varepsilon\int_{\Omega}G_{\varepsilon}(x,z)\eta_{\varepsilon}(z)dz=O(\varepsilon).
\end{equation*}
Therefore, together with the estimates of $\eta_{\varepsilon,1}$ and $\eta_{\varepsilon,2}$, we can deduce
\begin{equation}\label{538}
\begin{split}
|\eta_{\varepsilon}(x)|=O\Big(\frac{\ln \lambda_{\varepsilon}^{(1)}}{\big(1+\lambda_{\varepsilon}^{(1)}|x-x_{\varepsilon}^{(1)}|\big)^{N-2}}
\Big)+O(\varepsilon).
\end{split}
\end{equation}
Hence \eqref{e6.3} can be deduced by \eqref{538}.
\end{proof}

According to the above nondegeneracy result in Lemma \ref{Lemma4.5}, we have the following crucial lemma.
\begin{Lem}\label{Lem6.3}
Suppose that the exponents $N$, $\mu$ satisfy the assumptions $(**)$  in Theorem \ref{Thm1.4}. Let $\tilde{\eta}_{\varepsilon}(x)=\eta_{\varepsilon}\big(\frac{x}{\lambda_{\varepsilon}^{(1)}}+x_{\varepsilon}^{(1)}\big)$.
Then we have that
\begin{equation}\label{6.3}
\Big|\tilde{\eta}_{\varepsilon}(x)-\sum_{k=0}^{N}c_k\phi_k(x)\Big|=o\Big(\frac{1}{\lambda_{\varepsilon}^{(1)}}\Big),~~\text{uniformly in}~~C^{1}(B_{R}(0)~\text{for any}~R>0,
\end{equation}
where $c_k$, $k=1,\cdots,N$ are some constants and
\begin{equation*}
\phi_0=\frac{\partial U_{0,\lambda}}{\partial\lambda}\Big|_{\lambda=1},~~~\phi_k=\frac{\partial U_{0,1}}{\partial x_k},~k=1,\cdots,N.
\end{equation*}
\end{Lem}
\begin{proof}
Since $|\tilde{\eta}_{\varepsilon}|\leq1$, by the regularity theorem \cite{Gilbarg-1983}, we know that
\begin{equation*}
	\tilde{\eta}_{\varepsilon}\in C^{1}(B_{\tilde{R}}(0))~~~\text{and}
	~\|\tilde{\eta}_{\varepsilon}\|_{C^{1,\alpha}(B_{\tilde{R}}(0))}\leq C,
\end{equation*}
for any fixed large $\tilde{R}$ and $\alpha\in(0,1)$. Hence we may assume that $\tilde{\eta}_{\varepsilon}\rightarrow\tilde{\eta}_0$ in $C^1(B_{\tilde{R}}(0))$ for any large $\tilde{R}>0$.
Now by a direct calculation, we have
\begin{equation*}
	\begin{split}
		-\Delta\tilde{\eta}_{\varepsilon}(x)&=-\frac{1}{(\lambda_{\varepsilon}^{(1)})^2}\Delta\eta_{\varepsilon}\big(\frac{x}{\lambda_{\varepsilon}^{(1)}}+x_{\varepsilon}^{(1)}\big)
		\\&=E_{\varepsilon,1}\tilde{\eta}_{\varepsilon}\big(x\big)+ E_{\varepsilon,2}(x)+\frac{\varepsilon}{(\lambda_{\varepsilon}^{(1)})^2}\tilde{\eta}_{\varepsilon}\big(x\big),
	\end{split}
\end{equation*}
where
\begin{equation*}
	\begin{split}
		E_{\varepsilon,1}(x)&=\frac{2_{\mu}^{\ast}-1}{(\lambda_{\varepsilon}^{(1)})^{N-\mu+2}}\Big(\int_{\Omega}\frac{\big(u_{\varepsilon}^{(1)}((\lambda_{\varepsilon}^{(1)})^{-1}\xi+x_{\varepsilon}^{(1)})\big)^{2_{\mu}^\ast}}{|x-\xi|^{\mu}}d\xi\Big)C_{\varepsilon}\Big((\lambda_{\varepsilon}^{(1)})^{-1}x+x_{\varepsilon}^{(1)}\Big),
		\\&E_{\varepsilon,2}(x)=\frac{2_{\mu}^{\ast}}{(\lambda_{\varepsilon}^{(1)})^{N-\mu+2}}\Big(\int_{\Omega}\frac{D_{\varepsilon}\Big((\lambda_{\varepsilon}^{(1)})^{-1}\xi+x_{\varepsilon}^{(1)}\Big)\tilde{\eta}_{\varepsilon}\big(y\big)}{|x-\xi|^{\mu}}d\xi\Big)\Big(u_{\varepsilon}^{(2)}((\lambda_{\varepsilon}^{(1)})^{-1}x+x_{\varepsilon}^{(1)})\Big)^{2_{\mu}^\ast-1}.
	\end{split}
\end{equation*}
Then for any $\varphi(x)\in C_{0}^{\infty}(\R^N)$ with $supp~\varphi(x)\subset B_{\lambda_{\varepsilon}^{(1)}\delta}(x_{\varepsilon}^{(1)})$ for a small fixed $\delta$,
we have
\begin{equation}\label{e416}
	\int_{B_{\lambda_{\varepsilon}^{(1)}\delta}(x_{\varepsilon}^{(1)})}\nabla\tilde{\eta}_{\varepsilon}(x)\nabla\varphi(x)dx=\int_{B_{\lambda_{\varepsilon}^{(1)}\delta}(x_{\varepsilon}^{(1)})}
	\big(E_{\varepsilon,1}\tilde{\eta}_{\varepsilon}(x)+ E_{\varepsilon,2}(x)\big)\varphi(x)dx+\frac{\varepsilon}{(\lambda_{\varepsilon}^{(1)})^2}\int_{B_{\lambda_{\varepsilon}^{(1)}\delta}(x_{\varepsilon}^{(1)})}\tilde{\eta}_{\varepsilon}(x)\varphi(x)dx.
\end{equation}
In view of the elementary inequality \eqref{ele} in Appendix A, we know
\begin{equation}\label{e417}
	\begin{split}
		&\int_{B_{\lambda_{\varepsilon}^{(1)}\delta}(x_{\varepsilon}^{(1)})}
		E_{\varepsilon,1}\tilde{\eta}_{\varepsilon}(x)\varphi(x)dx\\&=\frac{2_{\mu}^{\ast}-1}{(\lambda_{\varepsilon}^{(1)})^{N-\mu+2}}\int_{B_{\lambda_{\varepsilon}^{(1)}\delta}(x_{\varepsilon}^{(1)})}\int_{B_{\lambda_{\varepsilon}^{(1)}\delta}(x_{\varepsilon}^{(1)})}\frac{\big(u_{\varepsilon}^{(1)}(\frac{\xi}{\lambda_{\varepsilon}^{(1)}}+x_{\varepsilon}^{(1)})\big)^{2_{\mu}^\ast}C_{\varepsilon}\Big(\frac{x}{\lambda_{\varepsilon}^{(1)}}+x_{\varepsilon}^{(1)}\Big)\varphi(x)}{|x-\xi|^{\mu}}dxd\xi\\&
		=\frac{2_{\mu}^{\ast}-1}{(\lambda_{\varepsilon}^{(1)})^{N-\mu+2}}\mathcal{F}_1+\frac{2_{\mu}^{\ast}(2_{\mu}^{\ast}-1)}{(\lambda_{\varepsilon}^{(1)})^{N-\mu+2}}\mathcal{F}_2+\frac{2_{\mu}^{\ast}(2_{\mu}^{\ast}-1)^2}{(2\lambda_{\varepsilon}^{(1)})^{N-\mu+2}}\mathcal{F}_3+O\Big(\frac{2_{\mu}^{\ast}-1}{(\lambda_{\varepsilon}^{(1)})^{N-\mu+2}}\mathcal{F}_4\Big),
	\end{split}
\end{equation}
where
\begin{equation*}
	\begin{split}
		&\mathcal{F}_1=\int_{B_{\lambda_{\varepsilon}^{(1)}\delta}(x_{\varepsilon}^{(1)})}\int_{B_{\lambda_{\varepsilon}^{(1)}\delta}(x_{\varepsilon}^{(1)})}\frac{\big(PU_{x_{\varepsilon}^{(1)},\lambda_{\varepsilon}^{(1)}}(\frac{\xi}{\lambda_{\varepsilon}^{(1)}}+x_{\varepsilon}^{(1)})\big)^{2_{\mu}^\ast}C_{\varepsilon}\big(\frac{x}{\lambda_{\varepsilon}^{(1)}}+x_{\varepsilon}^{(1)}\big)\tilde{\eta}_{\varepsilon}(x)\varphi(x)}{|x-\xi|^{\mu}}dxd\xi,
		\\&\mathcal{F}_2=\int_{B_{\lambda_{\varepsilon}^{(1)}\delta}(x_{\varepsilon}^{(1)})}\int_{B_{\lambda_{\varepsilon}^{(1)}\delta}(x_{\varepsilon}^{(1)})}\frac{\big(PU_{x_{\varepsilon}^{(1)},\lambda_{\varepsilon}^{(1)}}(\frac{\xi}{\lambda_{\varepsilon}^{(1)}}+x_{\varepsilon}^{(1)})\big)^{2_{\mu}^\ast-1}w_{\varepsilon}^{(1)}(\frac{\xi}{\lambda_{\varepsilon}^{(1)}}+x_{\varepsilon}^{(1)})C_{\varepsilon}\big(\frac{x}{\lambda_{\varepsilon}^{(1)}}+x_{\varepsilon}^{(1)}\big)\tilde{\eta}_{\varepsilon}(x)\varphi(x)}{|x-\xi|^{\mu}}dxd\xi,\\&
		\mathcal{F}_3=\int_{B_{\lambda_{\varepsilon}^{(1)}\delta}(x_{\varepsilon}^{(1)})}\int_{B_{\lambda_{\varepsilon}^{(1)}\delta}(x_{\varepsilon}^{(1)})}\frac{\big(PU_{x_{\varepsilon}^{(1)},\lambda_{\varepsilon}^{(1)}}(\frac{\xi}{\lambda_{\varepsilon}^{(1)}}+x_{\varepsilon}^{(1)})\big)^{2_{\mu}^\ast-2}(w_{\varepsilon}^{(1)}(\frac{\xi}{\lambda_{\varepsilon}^{(1)}}+x_{\varepsilon}^{(1)}))^{2}C_{\varepsilon}\big(\frac{x}{\lambda_{\varepsilon}^{(1)}}+x_{\varepsilon}^{(1)}\big)\tilde{\eta}_{\varepsilon}(x)\varphi(x)}{|x-\xi|^{\mu}}dxd\xi,\\&
		\mathcal{F}_4=\int_{B_{\lambda_{\varepsilon}^{(1)}\delta}(x_{\varepsilon}^{(1)})}\int_{B_{\lambda_{\varepsilon}^{(1)}\delta}(x_{\varepsilon}^{(1)})}\frac{\big(w_{\varepsilon}^{(1)}(\frac{\xi}{\lambda_{\varepsilon}^{(1)}}+x_{\varepsilon}^{(1)})\big)^{2_{\mu}^\ast}C_{\varepsilon}\big(\frac{x}{\lambda_{\varepsilon}^{(1)}}+x_{\varepsilon}^{(1)}\big)\tilde{\eta}_{\varepsilon}(x)\varphi(x)}{|x-\xi|^{\mu}}dxd\xi.
	\end{split}
\end{equation*}
By the definition of $U_{x_{\varepsilon}^{(1)},\lambda_{\varepsilon}^{(1)}}$, Lemma~\ref{L4.2} and Lemma \ref{Lem53}, we have
\begin{equation}\label{e414}
	\begin{split}
		\mathcal{F}_1&=\int_{B_{\lambda_{\varepsilon}^{(1)}\delta}(x_{\varepsilon}^{(1)})}\int_{B_{\lambda_{\varepsilon}^{(1)}\delta}(x_{\varepsilon}^{(1)})}\frac{\big(PU_{x_{\varepsilon}^{(1)},\lambda_{\varepsilon}^{(1)}}(\frac{\xi}{\lambda_{\varepsilon}^{(1)}}+x_{\varepsilon}^{(1)})\big)^{2_{\mu}^\ast}\big(U_{x_{\varepsilon}^{(1)},\lambda_{\varepsilon}^{(1)}}(\frac{x}{\lambda_{\varepsilon}^{(1)}}+x_{\varepsilon}^{(1)})\big)^{2_{\mu}^{\ast}-2}\tilde{\eta}_{\varepsilon}(x)\varphi(x)}{|x-\xi|^{\mu}}dxd\xi
		+o\Big(\frac{1}{\lambda_{\varepsilon}^{(1)}}\Big)
		\\&=(\lambda_{\varepsilon}^{(1)})^{N-\mu+2}\int_{\R^N}\int_{\R^N}\frac{U_{0,1}^{2_{\mu}^\ast}(\xi)U_{0,1}^{2_{\mu}^\ast-2}(x)\tilde{\eta}_{\varepsilon}(x)\varphi(x)}{|x-\xi|^{\mu}}dxd\xi+o\Big(\frac{1}{\lambda_{\varepsilon}^{(1)}}\Big).
	\end{split}
\end{equation}
And we have the following rest other estimates for which the proof is left in Lemma~\ref{A3} in Appendix A:
\begin{equation}\label{e418}
	\frac{1}{(\lambda_{\varepsilon}^{(1)})^{N-\mu+2}}\mathcal{F}_2=o\big(\frac{1}{\lambda_{\varepsilon}^{(1)}}\big),~~~\frac{1}{(\lambda_{\varepsilon}^{(1)})^{N-\mu+2}}\mathcal{F}_3=o\big(\frac{1}{\lambda_{\varepsilon}^{(1)}}\big),~~~\frac{1}{(\lambda_{\varepsilon}^{(1)})^{N-\mu+2}}\mathcal{F}_4=o\Big(\frac{1}{\lambda_{\varepsilon}^{(1)}}\Big).
\end{equation}
Now similar to the calculations of \eqref{e417}, by  inequality \eqref{ele}, we can deduce
\begin{equation}\label{e419}
	\begin{split}
		&\int_{B_{\lambda_{\varepsilon}^{(1)}\delta}(a_{\varepsilon}^{(1)})}
		E_{\varepsilon,2}\varphi(x)dx\\& ~~=\frac{2_{\mu}^{\ast}}{(\lambda_{\varepsilon}^{(1)})^{N-\mu+2}}\int_{B_{\lambda_{\varepsilon}^{(1)}\delta}(x_{\varepsilon}^{(1)})}\int_{B_{\lambda_{\varepsilon}^{(1)}\delta}(x_{\varepsilon}^{(1)})}\frac{D_{\varepsilon}\big(\frac{\xi}{\lambda_{\varepsilon}^{(1)}}+x_{\varepsilon}^{(1)}\big)\tilde{\eta}_{\varepsilon}(\xi)\big(u_{\varepsilon}^{(2)}(\frac{x}{\lambda_{\varepsilon}^{(1)}}+x_{\varepsilon}^{(1)})\big)^{2_{\mu}^\ast-1}\varphi(x)}{|x-\xi|^{\mu}}dxd\xi\\&
		~~=2_{\mu}^{\ast}\int_{\R^N}\int_{\R^N}\frac{U_{0,1}^{2_{\mu}^\ast}(\xi)\tilde{\eta}_{\varepsilon}(\xi)U_{0,1}^{2_{\mu}^\ast-1}(x)\varphi(x)}{|x-\xi|^{\mu}}dxd\xi+o\Big(\frac{1}{\lambda_{\varepsilon}^{(1)}}\Big).
	\end{split}
\end{equation}

On the other hand, from $\|\eta_{\varepsilon}\|=1$, we have
\begin{equation}\label{e421}
	\frac{\varepsilon}{(\lambda_{\varepsilon}^{(1)})^2}\int_{B_{\lambda_{\varepsilon}^{(1)}\delta}(x_{\varepsilon}^{(1)})}\tilde{\eta}_{\varepsilon}(x)\varphi(x)dx=o\Big(\frac{1}{\lambda_{\varepsilon}^{(1)}}\Big).
\end{equation}
Consequently, in view of \eqref{e416}-\eqref{e421}, we obtain
\begin{equation}\label{e422}
	\begin{split}
		\int_{B_{\lambda_{\varepsilon}^{(1)}\delta}(x_{\varepsilon}^{(1)})}\nabla\tilde{\eta}_{\varepsilon}(x)\nabla\varphi(x)dx&=(2_{\mu}^{\ast}-1)\int_{\R^N}\int_{\R^N}\frac{U_{0,1}^{2_{\mu}^\ast}(\xi)U_{0,1}^{2_{\mu}^\ast-2}(x)\tilde{\eta}_{\varepsilon}(x)\varphi(x)}{|x-\xi|^{\mu}}dxd\xi\\&~~~+
		2_{\mu}^{\ast}\int_{\R^N}\int_{\R^N}\frac{U_{0,1}^{2_{\mu}^\ast}(\xi)\tilde{\eta}_{\varepsilon}(y)U_{0,1}^{2_{\mu}^\ast-1}(x)\varphi(x)}{|x-\xi|^{\mu}}dxd\xi+o\Big(\frac{1}{\lambda_{\varepsilon}^{(1)}}\Big).
	\end{split}
\end{equation}
Taking $\varepsilon\rightarrow0$ in \eqref{e422}, we find that $\tilde{\eta}_0$ satisfies
\begin{equation}\label{e423}
	-\Delta\tilde{\eta}_0=(2_{\mu}^{\ast}-1)\Big(\int_{\R^N}\frac{U_{0,1}^{2_{\mu}^\ast}(\xi)}{|x-\xi|^{\mu}}d\xi\Big)U_{0,1}^{2_{\mu}^{\ast}-2}(x)\tilde{\eta}_{0}(x)+ 2_{\mu}^\ast\Big(\int_{\R^N}\frac{U_{0,1}^{2_{\mu}^{\ast}-1}(\xi)\tilde{\eta}_{0}(\xi)}{|x-\xi|^{\mu}}d\xi\Big)U_{0,1}^{2_{\mu}^\ast-1},~~~\text{in}~~\R^N,
\end{equation}
From the non-degeneracy results (see Lemma \ref{Lemma4.5}), which gives $\tilde{\eta}_0=\sum\limits_{k=0}^{N}c_k\phi_k$. Hence the conclusion \eqref{6.3} follows by \eqref{e422} and \eqref{e423}.
\end{proof}
The proof of the following lemma is postponed to Section~$4$.
\begin{Lem}\label{Lem6.6}
Suppose that the exponents $N$, $\mu$ satisfy the assumptions $(**)$  in Theorem \ref{Thm1.4}, there holds
\begin{equation*}
c_k=0,~~\text{for}~~k=0,\cdots,N,
\end{equation*}
where $c_k$ are the constants in Lemma~\ref{Lem6.3}.
\end{Lem}

We are going to prove Theorem~\ref{Thm1.4} by using Lemma~\ref{Lem6.3} and  Lemma \ref{Lem6.6}.\\

\noindent
\textbf{Proof of~Theorem~\ref{Thm1.4}.}
From \eqref{538}, we find that
\begin{equation*}
|\eta_{\varepsilon}(x)|=O\big(\frac{1}{R^{N-3}}\big)+O\big(\varepsilon\big),~~~\text{for}~~x\in\Omega\setminus B_{R(\lambda_{\varepsilon}^{(1)})^{-1}}(x_{\varepsilon}^{(1)}),
\end{equation*}
which means that for any fixed $\theta\in(0,1)$ and small $\varepsilon$, there exists $\tilde{R}>0$ such that
\begin{equation*}
|\eta_{\varepsilon}(x)|\leq\theta,~~~\text{for}~~x\in\Omega\setminus B_{\tilde{R}(\lambda_{\varepsilon}^{(1)})^{-1}}(x_{\varepsilon}^{(1)}).
\end{equation*}
Also for above fixed $\tilde{R}$, in view of Lemma~\ref{Lem6.6}, we know
\begin{equation}\label{N1}
|\eta_{\varepsilon}(x)|=o(1),~~~\text{for}~~x\in B_{\tilde{R}(\lambda_{\varepsilon}^{(1)})^{-1}}(x_{\varepsilon}^{(1)}).
\end{equation}
Then for any fixed $\theta\in(0,1)$ and small $\varepsilon$, we can deduce that $|\eta_{\varepsilon}(x)|<\theta$ for all $x\in\Omega$. This is a contradiction to $\|\eta_{\varepsilon}\|_{L^{\infty}}=1$. So $u_{\varepsilon}^{(1)}=u_{\varepsilon}^{(2)}$ for small $\varepsilon$.  This finishes the proof of Theorem \ref{Thm1.4}.
$\hfill{} \Box$

\section{Proof of Lemma \ref{Lem6.6}}
This section is devoted to the proof of Lemma \ref{Lem6.6}.

\begin{Lem}\label{Lem6.4}
For $N\geq6$ and $\mu\in(0,4]$, let $\eta_{\varepsilon}(x)$ be the function defined by \eqref{6.1}. Then we have the following estimate:
\begin{equation}\label{422}
\begin{split}
\eta_{\varepsilon}(x)&=\Big((2_{\mu}^{\ast}-1)A_{\varepsilon}^{(1)}+2_{\mu}^{\ast}A_{\varepsilon}^{(2)}\Big)G(x_{\varepsilon}^{(1)},x)\\&
~~~+\sum_{k=1}^N\Big((2_{\mu}^{\ast}-1)B_{\varepsilon,k}^{(1)}+2_{\mu}^{\ast}B_{\varepsilon,k}^{(2)}\Big)\partial_kG(x_{\varepsilon}^{(1)},x)
+O\Big(\frac{\ln\lambda_{\varepsilon}^{(1)}}{(\lambda_{\varepsilon}^{(1)})^N}\Big),~~~\text{in}~~C^1\big(\Omega\setminus B_{2\delta}(x_{\varepsilon}^{(1)})\big),
\end{split}
\end{equation}
where $\delta>0$ is any small fixed constant, $\partial_{k}G(z,x)=\frac{G(z,x)}{\partial z_k}$,
\begin{equation}\label{423}
A_{\varepsilon}^{(1)}=\int_{B_{\delta}(x_{\varepsilon}^{(1)})}\int_{B_{\delta}(x_{\varepsilon}^{(1)})}\frac{(u_{\varepsilon}^{(1)}(\xi))^{2_{\mu}^\ast}C_{\varepsilon}(z)\eta_{\varepsilon}(z)}{|z-\xi|^{\mu}}dzd\xi,
\end{equation}
\begin{equation}\label{4233}
A_{\varepsilon}^{(2)}=\int_{B_{\delta}(x_{\varepsilon}^{(1)})}\int_{B_{\delta}(x_{\varepsilon}^{(1)})}\frac{D_{\varepsilon}(\xi)\eta_{\varepsilon}(\xi)(u_{\varepsilon}^{(2)}(z))^{2_{\mu}^\ast-1}}{|z-\xi|^{\mu}}dzd\xi,
\end{equation}
\begin{equation}\label{424}
B_{\varepsilon,k}^{(1)}=\int_{B_{\delta}(x_{\varepsilon}^{(1)})}\int_{B_{\delta}(x_{\varepsilon}^{(1)})}(z_k-x_{\varepsilon,k}^{(1)})\frac{(u_{\varepsilon}^{(1)}(\xi))^{2_{\mu}^\ast}C_{\varepsilon}(z)\eta_{\varepsilon}(z)}{|z-\xi|^{\mu}}dzd\xi,
\end{equation}
\begin{equation}\label{4244}
B_{\varepsilon,k}^{(2)}=\int_{B_{\delta}(x_{\varepsilon}^{(1)})}\int_{B_{\delta}(x_{\varepsilon}^{(1)})}(z_k-x_{\varepsilon,k}^{(1)})\frac{D_{\varepsilon}(\xi)\eta_{\varepsilon}(\xi)(u_{\varepsilon}^{(2)}(z))^{2_{\mu}^\ast-1}}{|z-\xi|^{\mu}}dzd\xi.
\end{equation}
\end{Lem}
\begin{proof}
By the potential theory and \eqref{6.2}, we have
\begin{equation}\label{eta1}
\begin{split}
&\eta_{\varepsilon}(x)=\int_{\Omega}G(z,x)g\big(z,\lambda^{(1)}_{\varepsilon}(z),\lambda_{\varepsilon}^{(2)}(z)\big)dz\\&
=(2_{\mu}^{\ast}-1)\int_{\Omega}G(z,x)\Big(\int_{\Omega}\frac{(u_{\varepsilon}^{(1)}(\xi))^{2_{\mu}^\ast}}{|z-\xi|^{\mu}}d\xi\Big)C_{\varepsilon}(z)\eta_{\varepsilon}(z)dz+2_{\mu}^{\ast}\int_{\Omega}G(z,x) \Big(\int_{\Omega}\frac{D_{\varepsilon}(\xi)\eta_{\varepsilon}(\xi)}{|x-\xi|^{\mu}}dz\Big)(u_{\varepsilon}^{(2)}(z))^{2_{\mu}^\ast-1}dz\\&~~~+\varepsilon\int_{\Omega}G(z,x) \eta_{\varepsilon}(z)dz.
\end{split}
\end{equation}

 According to Lemma~\ref{Lem6.2}, for any $z\in\Omega\setminus B_{2\delta}(a_{\varepsilon}^{(1)})$, we obtain the estimate of the third term in \eqref{eta1} as
 \begin{equation}\label{eta2}
\begin{split}
\int_{\Omega}G(z,x) \eta_{\varepsilon}(z)dz&=\int_{B_{\delta}(x_{\varepsilon}^{(1)})}G(z,x) \eta_{\varepsilon}(z)dz+\int_{\Omega\setminus B_{\delta}(x_{\varepsilon}^{(1)})}G(z,x) \eta_{\varepsilon}(z)dz\\&
=\int_{B_{\delta}(x_{\varepsilon}^{(1)})}\eta_{\varepsilon}(z)dz+O\Big(\int_{\Omega}G(z,x)\frac{\ln\lambda_{\varepsilon}^{(1)}}{(\lambda_{\varepsilon}^{(1)})^{N-2}}dz\Big)
=O\Big(\frac{\ln\lambda_{\varepsilon}^{(1)}}{\big(\lambda_{\varepsilon}^{(1)}\big)^{N-2}}\Big).
\end{split}
\end{equation}

Decomposing the first term of \eqref{eta1} by
\begin{equation*}
\begin{split}
\int_{\Omega}G(z,x)\Big(\int_{\Omega}\frac{(u_{\varepsilon}^{(1)}(\xi))^{2_{\mu}^\ast}}{|z-\xi|^{\mu}}d\xi\Big)&C_{\varepsilon}(z)\eta_{\varepsilon}(z)dz=\int_{B_{\delta}(x_{\varepsilon}^{(1)})}G(z,x)\Big(\int_{B_{\delta}(x_{\varepsilon}^{(1)})}\frac{(u_{\varepsilon}^{(1)}(\xi))^{2_{\mu}^\ast}}{|x-\xi|^{\mu}}d\xi\Big)C_{\varepsilon}(z)\eta_{\varepsilon}(z)dz\\&
~~~+2\int_{\Omega\setminus B_{\delta}(z_{\varepsilon}^{(1)})}G(z,x)\Big(\int_{ B_{\delta}(x_{\varepsilon}^{(1)})}\frac{(u_{\varepsilon}^{(1)}(\xi))^{2_{\mu}^\ast}}{|z-\xi|^{\mu}}d\xi\Big)C_{\varepsilon}(z)\eta_{\varepsilon}(z)dz\\&
~~~+\int_{\Omega\setminus B_{\delta}(x_{\varepsilon}^{(1)})}G(z,x)\Big(\int_{\Omega\setminus B_{\delta}(x_{\varepsilon}^{(1)})}\frac{(u_{\varepsilon}^{(1)}(\xi))^{2_{\mu}^\ast}}{|z-\xi|^{\mu}}d\xi\Big)C_{\varepsilon}(z)\eta_{\varepsilon}(z)dz\\&
:=G_1+2G_2+G_3,
\end{split}
\end{equation*}
We are going to estimate $G_1$, $G_2$ and $G_3$, respectively.

By using Hardy-Littlewood-Sobolev inequality, Lemma~\ref{Lem6.2} and \eqref{ced}, then we have
\begin{equation*}
\begin{split}
G_1&=A_{\varepsilon}^{(1)}G(x_{\varepsilon}^{(1)},x)+\int_{B_{\delta}(x_{\varepsilon}^{(1)})}\big(G(z,x)-G(x_{\varepsilon}^{(1)},x)\big)\Big(\int_{B_{\delta}(x_{\varepsilon}^{(1)})}\frac{(u_{\varepsilon}^{(1)}(\xi))^{2_{\mu}^\ast}}{|z-\xi|^{\mu}}d\xi\Big)C_{\varepsilon}(z)\eta_{\varepsilon}(z)dz\\&
=A_{\varepsilon}^{(1)}G(x_{\varepsilon}^{(1)},x)+\sum_{k=1}^{N}\partial_{k} G(x_{\varepsilon}^{(1)},x)\int_{B_{\delta}(x_{\varepsilon}^{(1)})}\int_{B_{\delta}(x_{\varepsilon}^{(1)})}(z_k-x_{\varepsilon,k}^{(1)})\frac{(u_{\varepsilon}^{(1)}(\xi))^{2_{\mu}^\ast}C_{\varepsilon}(z)\eta_{\varepsilon}(z)}{|z-\xi|^{\mu}}dzd\xi\\&
~~~+O\Big(\frac{1}{|x-x_{\varepsilon}^{(1)}|^N}\int_{B_{\delta}(x_{\varepsilon}^{(1)})}\big|z-x_{\varepsilon}^{(1)}\big|^2\Big(\int_{B_{\delta}(x_{\varepsilon}^{(1)})}\frac{(u_{\varepsilon}^{(1)}(\xi))^{2_{\mu}^\ast}}{|z-\xi|^{\mu}}d\xi\Big)C_{\varepsilon}(z)\eta_{\varepsilon}(z)dz\Big)\\&
=A_{\varepsilon}^{(1)}G(x_{\varepsilon}^{(1)},x)+\sum_{k=1}^{N}\partial_{k}G(x_{\varepsilon}^{(1)},x)B_{\varepsilon,k}^{(1)}
+O\Big(\frac{1}{|x-x_{\varepsilon}^{(1)}|^N}\int_{B_{\delta}(x_{\varepsilon}^{(1)})}\int_{B_{\delta}(x_{\varepsilon}^{(1)})}\frac{\big|z-x_{\varepsilon}^{(1)}\big|^2U_{x_\varepsilon^{(1)},\lambda_{\varepsilon}^{(1)}}^{2_{\mu}^\ast}(\xi)C_{\varepsilon}(z)\eta_{\varepsilon}(z)}{|z-\xi|^{\mu}}dzd\xi\Big)\\& =A_{\varepsilon}^{(1)}G(x_{\varepsilon}^{(1)},x)+\sum_{k=1}^{N}\partial_{k}G(x_{\varepsilon}^{(1)},x)B_{\varepsilon,k}^{(1)}+O\Big(\frac{1}{|x-x_{\varepsilon}^{(1)}|^N}\int_{B_{\delta}(x_{\varepsilon}^{(1)})}\big|z-x_{\varepsilon}^{(1)}\big|^2U_{x_\varepsilon^{(1)},\lambda_{\varepsilon}^{(1)}}^{\frac{4}{N-2}}(x)\eta_{\varepsilon}(z)dz\Big)\\&
=A_{\varepsilon}^{(1)}G(x_{\varepsilon}^{(1)},x)+\sum_{k=1}^{N}\partial_{k}G(x_{\varepsilon}^{(1)},x)B_{\varepsilon,k}^{(1)}+O\Big(\frac{\ln\lambda_{\varepsilon}^{(1)}}{\big(\lambda_{\varepsilon}^{(1)}\big)^{N}}\frac{1}{|x-x_{\varepsilon}^{(1)}|^N}\Big),
\end{split}
\end{equation*}
where $A_{\varepsilon}^{(1)}$ and $B_{\varepsilon,k}^{(1)}$ are defined in \eqref{423} and \eqref{424}. Moreover, we can also find
\begin{equation*}
\begin{split}
G_2&\leq C\int_{\Omega\setminus B_{\delta}(x_{\varepsilon}^{(1)})}\int_{ B_{\delta}(x_{\varepsilon}^{(1)})}\frac{U_{x_{\varepsilon}^{(1)},\lambda_{\varepsilon}^{(1)}}^{2_{\mu}^\ast}(\xi)G(z,x)C_{\varepsilon}(z)\eta_{\varepsilon}(z)}{|z-\xi|^{\mu}}dzd\xi\\&
\leq C\frac{\ln\lambda_{\varepsilon}^{(1)}}{(\lambda_{\varepsilon}^{(1)})^{N}}\Big(\int_{\Omega\setminus B_{\delta}(z_{\varepsilon}^{(1)})\setminus B_{2\delta}(x)}\frac{1}{|z-x|^{N-2}}\frac{1}{|z-x_{\varepsilon}^{(1)|^{2}}}dz+\int_{\Omega\setminus B_{\delta}(x_{\varepsilon}^{(1)})\cap B_{2\delta}(x)}\frac{1}{|z-x|^{N-2}}\frac{1}{|z-x_{\varepsilon}^{(1)|^{2}}}dz\Big)\\&
\leq C\frac{\ln\lambda_{\varepsilon}^{(1)}}{(\lambda_{\varepsilon}^{(1)})^{N}\delta^{N-2}}\int_{\Omega\setminus B_{\delta}(x_{\varepsilon}^{(1)})}\frac{1}{|z-x_{\varepsilon}^{(1)}|^{2}}dz+C\frac{\ln\lambda_{\varepsilon}^{(1)}}{(\lambda_{\varepsilon}^{(1)})^{N}\delta^{2}}\int_{ B_{2\delta}(x)}\frac{1}{|z-x|^{N-2}}dz\\&
=O\Big(\frac{\ln\lambda_{\varepsilon}^{(1)}}{(\lambda_{\varepsilon}^{(1)})^{N}}\Big).
\end{split}
\end{equation*}
Analogously, we also have
\begin{equation*}
G_3=O\Big(\frac{\ln\lambda_{\varepsilon}^{(1)}}{(\lambda_{\varepsilon}^{(1)})^{N}}\Big).
\end{equation*}

For the second term in \eqref{eta1}, we decompose it by
\begin{equation*}
\begin{split}
\int_{\Omega}&G(z,x) \Big(\int_{\Omega}\frac{D_{\varepsilon}(\xi)\eta_{\varepsilon}(\xi)}{|z-\xi|^{\mu}}d\xi\Big)(u_{\varepsilon}^{(2)}(z))^{2_{\mu}^\ast-1}dz\\
&=\int_{B_{\delta}(x_{\varepsilon}^{(1)})}G(z,x)\Big(\int_{B_{\delta}(x_{\varepsilon}^{(1)})}\frac{D_{\varepsilon}(\xi)\eta_{\varepsilon}(\xi)}{|z-\xi|^{\mu}}d\xi\Big)(u_{\varepsilon}^{(2)}(z))^{2_{\mu}^\ast-1}dz\\&
\hspace{4mm}+2\int_{B_{\delta}(x_{\varepsilon}^{(1)})}G_{\varepsilon}(z,x)\Big(\int_{\Omega\setminus B_{\delta}(x_{\varepsilon}^{(1)})}\frac{D_{\varepsilon}(\xi)\eta_{\varepsilon}(\xi)}{|z-\xi|^{\mu}}d\xi\Big)(u_{\varepsilon}^{(2)}(z))^{2_{\mu}^\ast-1}dz\\&
\hspace{6mm}+\int_{\Omega\setminus B_{\delta}(x_{\varepsilon}^{(1)})}G(z,x)\Big(\int_{\Omega\setminus B_{\delta}(x_{\varepsilon}^{(1)})}\frac{D_{\varepsilon}(\xi)\eta_{\varepsilon}(\xi)}{|z-\xi|^{\mu}}d\xi\Big)(u_{\varepsilon}^{(2)}(z))^{2_{\mu}^\ast-1}dz\\&
:=L_1+L_2+L_3.
\end{split}
\end{equation*}
Similar to the estimate for $G_1$, by Hardy-Littlewood-
Sobolev inequality and the fact $|\eta_{\varepsilon}(x)|\leq1$, \eqref{ced} and \eqref{u12}, a direct calculation shows that
\begin{equation*}
\begin{split}
L_1&
=A_{\varepsilon}^{(2)}G(x_{\varepsilon}^{(1)},x)+\sum_{k=1}^{N}\partial_{k} G(x_{\varepsilon}^{(1)},x)\int_{B_{\delta}(x_{\varepsilon}^{(1)})}\int_{B_{\delta}(x_{\varepsilon}^{(1)})}(z_k-x_{\varepsilon,k}^{(1)})\frac{D_{\varepsilon}(\xi)\eta_{\varepsilon}(\xi)(u_{\varepsilon}^{(2)}(z))^{2_{\mu}^\ast-1}}{|z-\xi|^{\mu}}dzd\xi\\
&\hspace{4mm}+O\Big(\frac{1}{|x-x_{\varepsilon}^{(1)}|^N}\int_{B_{\delta}(x_{\varepsilon}^{(1)})}\int_{B_{\delta}(x_{\varepsilon}^{(1)})}\big|z-x_{\varepsilon}^{(1)}\big|^2\frac{D_{\varepsilon}(\xi)\eta_{\varepsilon}(\xi)(u_{\varepsilon}^{(2)}(z))^{2_{\mu}^\ast-1}}{|z-\xi|^{\mu}}dzd\xi\Big)\\&
=A_{\varepsilon}^{(2)}G(x_{\varepsilon}^{(1)},x)+\sum_{k=1}^{N}\partial_{k}G(x_{\varepsilon}^{(1)},x)B_{\varepsilon,k}^{(2)}\\&
\hspace{4mm}+O\Big(\frac{(\lambda_{\varepsilon}^{(1)})^{N-\mu+2}}{|x-x_{\varepsilon}^{(1)}|^N}\Big)\int_{ B_{\delta}(x_{\varepsilon}^{(1)})}\int_{ B_{\delta}(x_{\varepsilon}^{(1)})}\frac{1}{\big(1+\lambda_{\varepsilon}^{(1)}|\xi-x_{\varepsilon}^{(1)}|\big)^{N-\mu+2}}\frac{1}{|z-\xi|^\mu}\frac{|z-x_{\varepsilon}^{(1)}|^2}{\big(1+\lambda_{\varepsilon}^{(1)}|z-x_{\varepsilon}^{(1)}|\big)^{N-\mu+2}}dzd\xi\\& =A_{\varepsilon}^{(2)}G(x_{\varepsilon}^{(1)},x)+\sum_{k=1}^{N}\partial_{k}G(x_{\varepsilon}^{(1)},x)B_{\varepsilon,k}^{(2)}+\\&
\hspace{4mm}+O\Big(\frac{1}{(\lambda_{\varepsilon}^{(1)})^{N}}\frac{1}{|x-x_{\varepsilon}^{(1)}|^N}\int_{ B_{\delta}(0)}\int_{ B_{\delta}(0)}\frac{1}{\big(1+|\xi|\big)^{N-\mu+2}}\frac{1}{|z-\xi|^\mu}\frac{|z|^2}{\big(1+|z|\big)^{N-\mu+2}}dzd\xi\\&
=A_{\varepsilon}^{(2)}G(x_{\varepsilon}^{(1)},x)+\sum_{k=1}^{N}B_{\varepsilon,k}^{(2)}\partial_{k}G(x_{\varepsilon}^{(1)},x)+O\Big(\frac{1}{(\lambda_{\varepsilon}^{(1)})^{N}}\frac{1}{|x-x_{\varepsilon}^{(1)}|^N}\Big),
\end{split}
\end{equation*}
where $A_{\varepsilon}^{(2)}$ and $B_{\varepsilon,k}^{(2)}$ are defined in \eqref{4233}and \eqref{4244}.
By Lemma~\ref{Lem6.2}, we can calculate that
\begin{equation}\label{eq424}
\begin{split}
L_2&\leq C\int_{\Omega\setminus B_{\delta}(x_{\varepsilon}^{(1)})}\int_{ B_{\delta}(x_{\varepsilon}^{(1)})}\frac{U_{x_{\varepsilon}^{(1)},\lambda_{\varepsilon}^{(1)}}^{2_{\mu}^\ast-1}(\xi)\eta_{\varepsilon}(\xi)U_{x_{\varepsilon}^{(1)},\lambda_{\varepsilon}^{(1)}}^{2_{\mu}^\ast-1}(z)G(z,x)}{|z-\xi|^{\mu}}dzd\xi\\&
\leq \frac{C\ln \lambda_{\varepsilon}^{{(1)}}}{(\lambda_{\varepsilon}^{(1)})^{N-2}}\Big(\int_{\Omega\setminus B_{\delta}(x_{\varepsilon}^{(1)})}U_{x_\varepsilon^{(1)},\lambda_{\varepsilon}^{(1)}}^{\frac{2N(2_{\mu}^\ast-1)}{2N-\mu}}(\xi)d\xi\Big)^{\frac{2N-\mu}{2N}}\Big(\int_{B_{\delta}(x_{\varepsilon}^{(1)})}\Big|\frac{1}{|z-x|^{N-2}}U_{x_\varepsilon^{(1)},\lambda_{\varepsilon}^{(1)}}^{2_{\mu}^\ast-1}(z)\Big|^{\frac{2N}{2N-\mu}}dz\Big)^{\frac{2N-\mu}{2N}}\\&
\leq\frac{C\ln \lambda_{\varepsilon}^{(1)}}{(\lambda_{\varepsilon}^{(1)})^{N-2}}\frac{1}{(\lambda_{\varepsilon}^{(1)})^{\frac{N-2}{2}}}\Big(\int_{B_{\delta}(x_{\varepsilon}^{(1)})}\frac{(\lambda_{\varepsilon}^{(1)})^\frac{N(N-\mu+2)}{2N-\mu}}{\big(1+(\lambda_{\varepsilon}^{(1)})^2|z-x_{\varepsilon}^{(1)}|^2\big)^\frac{N(N-\mu+2)}{2N-\mu}}dz\Big)^{\frac{2N-\mu}{2N}}\\&
=O\Big(\frac{\ln\lambda_{\varepsilon}^{(1)}}{(\lambda_{\varepsilon}^{(1)})^{2N-4}}\Big).
\end{split}
\end{equation}
And similar to the estimate of \eqref{eq424}, we can also find
\begin{equation*}
\begin{split}
L_3&\leq C\int_{\Omega\setminus B_{\delta}(x_{\varepsilon}^{(1)})}\int_{\Omega\setminus B_{\delta}(x_{\varepsilon}^{(1)})}\frac{U_{x_{\varepsilon}^{(1)},\lambda_{\varepsilon}^{(1)}}^{2_{\mu}^\ast-1}(\xi)\eta_{\varepsilon}(\xi)U_{x_{\varepsilon}^{(1)},\lambda_{\varepsilon}^{(1)}}^{2_{\mu}^\ast-1}(z)G(z,x)}{|z-\xi|^{\mu}}dzd\xi\\&
\leq \frac{C\ln \lambda_{\varepsilon}^{(1)}}{(\lambda_{\varepsilon}^{(1)})^{N-2}}\Big(\int_{\Omega\setminus B_{\delta}(x_{\varepsilon}^{(1)})}U_{x_\varepsilon^{(1)},\lambda_{\varepsilon}^{(1)}}^{\frac{2N(2_{\mu}^\ast-1)}{2N-\mu}}(\xi)d\xi\Big)^{\frac{2N-\mu}{2N}}\Big(\int_{B_{\delta}(x_{\varepsilon}^{(1)})}\Big|\frac{1}{|z-x|^{N-2}}U_{x_\varepsilon^{(1)},\lambda_{\varepsilon}^{(1)}}^{2_{\mu}^\ast-1}(z)\Big|^{\frac{2N}{2N-\mu}}dz\Big)^{\frac{2N-\mu}{2N}}\\&
\leq\frac{C\ln \lambda_{\varepsilon}^{(1)}}{(\lambda_{\varepsilon}^{(1)})^{N-2}}\frac{1}{(\lambda_{\varepsilon}^{(1)})^{\frac{N-2}{2}}}\Big(\int_{\big(\Omega\setminus B_{\delta}(x_{\varepsilon}^{(1)})\big)\setminus B_{2\delta}(x)}\Big|\frac{1}{|z-x|^{N-2}}\frac{(\lambda_{\varepsilon}^{(1)})^\frac{N-\mu+2}{2}}{\big(1+(\lambda_{\varepsilon}^{(1)})^2|z-x_{\varepsilon}^{(1)}|^2\big)^\frac{N-\mu+2}{2}}\Big|^{\frac{2N}{2N-\mu}}dz\Big)^{\frac{2N-\mu}{2N}}\\&~~~+\frac{C\ln \lambda_{\varepsilon}^{(1)}}{\tilde{\lambda}_{\varepsilon}^{N-2}}\frac{1}{(\lambda_{\varepsilon}^{(1)})^{\frac{N-2}{2}}}\int_{\big(\Omega\setminus B_{\delta}(x_{\varepsilon}^{(1)})\big)\cap B_{2\delta}(x)}\Big|\frac{1}{|z-x|^{N-2}}\frac{(\lambda_{\varepsilon}^{(1)})^\frac{N-\mu+2}{2}}{\big(1+(\lambda_{\varepsilon}^{(1)})^2|z-x_{\varepsilon}^{(1)}|^2\big)^\frac{N-\mu+2}{2}}\Big|^{\frac{2N}{2N-\mu}}dz\Big)^{\frac{2N-\mu}{2N}}\\&
\leq\frac{C\ln \lambda_{\varepsilon}^{(1)}}{(\lambda_{\varepsilon}^{(1)})^{N-2}}\frac{1}{(\lambda_{\varepsilon}^{(1)})^{N-\frac{\mu}{2}}}\Big[\Big(\int_{\Omega\setminus B_{\delta}(x_{\varepsilon}^{(1)})}\frac{1}{\big|z-x_{\varepsilon}^{(1)}\big|^\frac{2N(N-\mu+2)}{2N-\mu}}dz\Big)^{\frac{2N-\mu}{2N}}+\int_{ B_{2\delta}(x)}\frac{1}{\big|z-x\big|^{\frac{2N(N-2)}{2N-\mu}}}dz\Big)^{\frac{2N-\mu}{2N}}\Big]\\&
=O\Big(\frac{\ln \lambda_{\varepsilon}^{(1)}}{(\lambda_{\varepsilon}^{(1)})^{2N-\mu}}\Big)+O\Big(\frac{\ln \lambda_{\varepsilon}^{(1)}}{(\lambda_{\varepsilon}^{(1)})^{2N-4}}\Big),
\end{split}
\end{equation*}
where we using $\frac{2N(N-\mu+2)}{2N-\mu}>N$ and $\frac{2N(N-2)}{2N-\mu}<N$.
Combining \eqref{eta1}-\eqref{eta2} and estimates of $G_1$, $G_2$, $G_3$, $L_1$, $L_2$, $L_3$, then we get
\begin{equation*}
\begin{split}
\eta_{\varepsilon}(x)&=\Big((2_{\mu}^{\ast}-1)A_{\varepsilon}^{(1)}+2_{\mu}^{\ast}A_{\varepsilon}^{(2)}\Big)G(x_{\varepsilon}^{(1)},x)+\sum_{k=1}^N\Big((2_{\mu}^{\ast}-1)B_{\varepsilon,k}^{(1)}+2_{\mu}^{\ast}B_{\varepsilon,k}^{(2)}\Big)\partial_kG(x_{\varepsilon}^{(1)},x)\\&
~~~+O\Big(\frac{\ln\lambda_{\varepsilon}^{(1)}}{(\lambda_{\varepsilon}^{(1)})^N}+\frac{\varepsilon\ln\lambda_{\varepsilon}^{(1)}}{(\lambda_{\varepsilon}^{(1)})^{N-2}}
+\frac{1}{(\lambda_{\varepsilon}^{(1)})^{N}}+\frac{\ln \lambda_{\varepsilon}^{(1)}}{(\lambda_{\varepsilon}^{(1)})^{2N-4}}+\frac{\ln \lambda_{\varepsilon}^{(1)}}{(\lambda_{\varepsilon}^{(1)})^{2N-\mu}}\Big)\\&
=\Big((2_{\mu}^{\ast}-1)A_{\varepsilon}^{(1)}+2_{\mu}^{\ast}A_{\varepsilon}^{(2)}\Big)G(x_{\varepsilon}^{(1)},x)+\sum_{k=1}^N\Big((2_{\mu}^{\ast}-1)B_{\varepsilon,k}^{(1)}+2_{\mu}^{\ast}B_{\varepsilon,k}^{(2)}\Big)\partial_kG(x_{\varepsilon}^{(1)},x)+O\Big(\frac{\ln\lambda_{\varepsilon}^{(1)}}{(\lambda_{\varepsilon}^{(1)})^N}\Big),
\\&~~~\text{for}~~x\in\Omega\setminus B_{2\delta}(x_{\varepsilon}^{(1)}),
\end{split}
\end{equation*}
in the last step we have used $\varepsilon=O\big(\frac{1}{(\lambda_{\varepsilon}^{(1)})^{N-4}}\big)=O\big(\frac{1}{(\lambda_{\varepsilon}^{(1)})^{2}}\big)$.

On the other hand, from \eqref{eta1}, we obtain
\begin{equation*}
\begin{split}
\frac{\partial\eta_{\varepsilon}(x)}{\partial x_i}&
=(2_{\mu}^{\ast}-1)\int_{\Omega}D_{x_i}G(z,x)\Big(\int_{\Omega}\frac{(u_{\varepsilon}^{(1)}(\xi))^{2_{\mu}^\ast}}{|z-\xi|^{\mu}}d\xi\Big)C_{\varepsilon}(z)\eta_{\varepsilon}(z)dz\\&~~~+2_{\mu}^{\ast}\int_{\Omega}D_{x_i}G(z,x) \Big(\int_{\Omega}\frac{D_{\varepsilon}(\xi)\eta_{\varepsilon}(\xi)}{|z-\xi|^{\mu}}d\xi\Big)(u_{\varepsilon}^{(2)}(z))^{2_{\mu}^\ast-1}dz+\varepsilon\int_{\Omega}D_{x_i}G(z,x) \eta_{\varepsilon}(z)dz.
\end{split}
\end{equation*}
Similar to the above estimates of $\eta_{\varepsilon}(x)$, we know for $N\geq6$,
\begin{equation*}
\varepsilon\int_{\Omega}D_{x_i}G(z,x) \eta_{\varepsilon}(z)dz
=O\Big(\frac{\ln\lambda_{\varepsilon}^{(1)}}{\big(\lambda_{\varepsilon}^{(1)}\big)^{N}}\Big).
\end{equation*}
By Hardy-Littlewood-Sobolev inequality and the fact that $|\eta_{\varepsilon}(x)|\leq1$, Lemma~\ref{Lem6.2}, \eqref{ced} and \eqref{u12}, then we can get
\begin{equation*}
\begin{split}
\int_{\Omega}D_{x_i}G(z,x)\Big(\int_{\Omega}\frac{(u_{\varepsilon}^{(1)}(\xi))^{2_{\mu}^\ast}}{|z-\xi|^{\mu}}d\xi\Big)&C_{\varepsilon}(z)\eta_{\varepsilon}(z)dz
=A_{\varepsilon}^{(1)}D_{x_i}G(x_{\varepsilon}^{(1)},x)+\sum_{k=1}^{N}D_{x_i}(\partial_{k}G(x_{\varepsilon}^{(1)},x)B_{\varepsilon,k}^{(1)})+O\Big(\frac{\ln\lambda_{\varepsilon}^{(1)}}{(\lambda_{\varepsilon}^{(1)})^{N}}\Big),
\end{split}
\end{equation*}
and
\begin{equation*}
\int_{\Omega}D_{x_i}G(z,x) \Big(\int_{\Omega}\frac{D_{\varepsilon}(\xi)\eta_{\varepsilon}(\xi)}{|z-\xi|^{\mu}}d\xi\Big)(u_{\varepsilon}^{(2)}(z))^{2_{\mu}^\ast-1}dz
=A_{\varepsilon}^{(2)}D_{x_i}G(x_{\varepsilon}^{(1)},x)+\sum_{k=1}^{N}B_{\varepsilon,k}^{(2)}D_{x_i}(\partial_{k}G(x_{\varepsilon}^{(1)},x))+O\Big(\frac{\ln\lambda_{\varepsilon}^{(1)}}{(\lambda_{\varepsilon}^{(1)})^N}\Big).
\end{equation*}
Therefore, we deduce
\begin{equation*}
\begin{split}
\frac{\partial\eta_{\varepsilon}(x)}{\partial x_i}&
=\Big((2_{\mu}^{\ast}-1)A_{\varepsilon}^{(1)}+2_{\mu}^{\ast}A_{\varepsilon}^{(2)}\Big)D_{x_i}G(x_{\varepsilon}^{(1)},x)+\sum_{k=1}^N\Big((2_{\mu}^{\ast}-1)B_{\varepsilon,k}^{(1)}+2_{\mu}^{\ast}B_{\varepsilon,k}^{(2)}\Big)D_{x_i}(\partial_kG(x_{\varepsilon}^{(1)},x))+O\Big(\frac{\ln\lambda_{\varepsilon}^{(1)}}{(\lambda_{\varepsilon}^{(1)})^N}\Big),
\\&~~~\text{for}~~x\in\Omega\setminus B_{2\delta}(x_{\varepsilon}^{(1)}),
\end{split}
\end{equation*}
According to the above argument of $\eta_{\varepsilon}(x)$ and $\frac{\partial\eta_{\varepsilon}(x)}{\partial x_i}$. Then we can finish the proof of Lemma~\ref{Lem6.4}.
\end{proof}

\begin{Lem}\label{le6.4}
Assume that $N\geq6$, $\mu\in(0,4]$ and $u_{\varepsilon}^{(j)}$  with $j=1,2$ be the solutions of \eqref{eq1.1}. Then we have
\begin{equation}\label{le641} u_{\varepsilon}^{(j)}(x)=\frac{G(x_{\varepsilon}^{(1)},x)}{(\lambda_{\varepsilon}^{(1)})^{\frac{N-2}{2}}}A_{N,\mu}
+O\Big(\frac{\ln\lambda_{\varepsilon}^{(1)}}{(\lambda_{\varepsilon}^{(1)})^{\frac{N+2}{2}}}\Big)~~~~~\text{in}~~ C^1\Big(\Omega\setminus B_{2\delta}(x_\varepsilon^{(1)})\Big),
\end{equation}
where $A_{N,\mu}$ is from Lemma~\ref{Lem2.2}.
\end{Lem}
\begin{proof}
First, in view of Lemma~\ref{Lem5.2}, we know that \eqref{le641} holds for $j=1$ and
\begin{equation}\label{le642} u_{\varepsilon}^{(2)}(x)=\frac{G(x_{\varepsilon}^{(2)},x)}{(\lambda_{\varepsilon}^{(2)})^{\frac{N-2}{2}}}A_{N,\mu}
+O\Big(\frac{\ln\lambda_{\varepsilon}^{(1)}}{(\lambda_{\varepsilon}^{(1)})^{\frac{N+2}{2}}}\Big)~~~~~~~\text{in}~~ C^1\Big(\Omega\setminus B_{2\delta}(x_\varepsilon^{(2)})\Big).
\end{equation}
By a direct calculate shows that
\begin{equation*}
\frac{G(x_{\varepsilon}^{(2)},x)}{(\lambda_{\varepsilon}^{(2)})^{\frac{N-2}{2}}}=\frac{G(x_{\varepsilon}^{(1)},x)}{(\lambda_{\varepsilon}^{(1)})^{\frac{N-2}{2}}}
+O\Big(\frac{|x_{\varepsilon}^{(1)}-x_{\varepsilon}^{(2)}|}{(\lambda_{\varepsilon}^{(1)})^{\frac{N-2}{2}}}\Big)+O\Big(\frac{|\lambda_{\varepsilon}^{(1)}-\lambda_{\varepsilon}^{(2)}|}{(\lambda_{\varepsilon}^{(1)})^{\frac{N-2}{2}}}\Big).
\end{equation*}
Since $B_{\delta}(x_\varepsilon^{(1)})\subset B_{2\delta}(x_\varepsilon^{(1)})$ for small $\varepsilon$, we deduce that \eqref{le641} for $j=2$ from Lemma~\ref{Lem5.4} and \eqref{le642}.
\end{proof}

\begin{Lem}\label{Lem6.5}
For $\eta_{\varepsilon}(x)$ defined by \eqref{6.1}, we have the following pohozaev identities:
\begin{equation}\label{ph1}
\begin{split}
-&\int_{\partial\Omega^{\prime}}\frac{\partial u_{\varepsilon}^{(1)}}{\partial x_j}\frac{\partial\eta_{\varepsilon}}{\partial\nu}ds-\int_{\partial\Omega^{\prime}}\frac{\partial u_{\varepsilon}^{(2)}}{\partial \nu}\frac{\partial\eta_{\varepsilon}}{\partial x_j}ds+\frac{1}{2}\int_{\partial\Omega^{\prime}}\big\langle\nabla (u_{\varepsilon}^{(1)}+u_{\varepsilon}^{(2)}),\nabla\eta_{\varepsilon}\big\rangle\nu_jds
\\&=\frac{1}{2_{\mu}^\ast}
\int_{\partial\Omega^{\prime}}\int_{ \Omega}\frac{|u_{\varepsilon}^{(2)}(\xi)|^{2_{\mu}^\ast}\tilde{C}_{\varepsilon}(x)\eta_{\varepsilon}(x)\nu_j}
{|x-\xi|^{\mu}}d\xi ds+\frac{1}{2_{\mu}^\ast}\int_{\partial\Omega^{\prime}}\int_{ \Omega}\frac{D_{\varepsilon}(\xi)\eta_{\varepsilon}(\xi)|u_{\varepsilon}^{(1)}(x)|^{2_{\mu}^\ast}\nu_j}
{|x-\xi|^{\mu}}d\xi ds\\&\hspace{4mm}
+\frac{1}{2_{\mu}^\ast}
\int_{\partial\Omega^{\prime}}\int_{ \Omega^{\prime}}\frac{|u_{\varepsilon}^{(2)}(\xi)|^{2_{\mu}^\ast}\tilde{C}_{\varepsilon}(x)\eta_{\varepsilon}(x)\nu_j}
{|x-\xi|^{\mu}}d\xi ds+\frac{1}{2_{\mu}^\ast}\int_{\partial\Omega^{\prime}}\int_{ \Omega^{\prime}}\frac{D_{\varepsilon}(\xi)\eta_{\varepsilon}(\xi)|u_{\varepsilon}^{(1)}(x)|^{2_{\mu}^\ast}\nu_j}
{|x-\xi|^{\mu}}d\xi ds\\&\hspace{6mm}
+\frac{\mu}{2_{\mu}^\ast}\int_{\Omega^{\prime}}\int_{\Omega\setminus\Omega^{\prime}}(x_j-\xi_j)\frac{|u_{\varepsilon}^{(2)}(\xi)|^{2_{\mu}^\ast}\tilde{C}_{\varepsilon}(x)\eta_{\varepsilon}(x)}
{|x-\xi|^{\mu+2}}dxd\xi+\frac{\mu}{2_{\mu}^\ast}\int_{\Omega^{\prime}}\int_{\Omega\setminus\Omega^{\prime}}(x_j-\xi_j)\frac{D_{\varepsilon}(\xi)\eta_{\varepsilon}(\xi)|u_{\varepsilon}^{(1)}(x)|^{2_{\mu}^\ast}}
{|x-\xi|^{\mu+2}}dxd\xi\\&\hspace{8mm}+\frac{\varepsilon}{2}\int_{\partial\Omega^{\prime}} \big(u_{\varepsilon}^{(1)}+u_{\varepsilon}^{(2)}\big)\eta_{\varepsilon}\nu_jds,
\end{split}
\end{equation}
and
\begin{equation}\label{ph2}
\begin{split}
-&\int_{\partial\Omega^{\prime}}\frac{\partial\eta_{\varepsilon}}{\partial\nu}\big\langle x-x_{\varepsilon}^{(1)},\nabla u_{\varepsilon}^{(1)}\big\rangle ds-\int_{\partial\Omega^{\prime}}\frac{\partial u_{\varepsilon}^{(2)}}{\partial \nu}\big\langle x-x_{\varepsilon}^{(1)},\nabla\eta_{\varepsilon}\big\rangle ds\\&\hspace{2mm}+\frac{1}{2}\int_{\partial\Omega^{\prime}}\big\langle\nabla \big(u_{\varepsilon}^{(1)}+u_{\varepsilon}^{(2)}\big),\nabla\eta_{\varepsilon}\big\rangle\big\langle x-x_{\varepsilon}^{(1)},\nu\big\rangle ds+\frac{2-N}{2}\int_{\partial\Omega^{\prime}}\Big[\frac{\partial \eta_{\varepsilon}}{\partial \nu}u_{\varepsilon}^{(1)}+\frac{\partial u_{\varepsilon}^{(2)}}{\partial \nu}\eta_{\varepsilon}\Big]ds
\\&
=-\frac{\mu}{2}\Big[\int_{ \Omega^{\prime}}\int_{\Omega\setminus\Omega^{\prime}}\frac{|u_{\varepsilon}^{(2)}(\xi)|^{2_{\mu}^\ast}\tilde{C}_{\varepsilon}(x)\eta_{\varepsilon}(x)}
{|x-\xi|^{\mu}}d\xi dx
+\int_{ \Omega^{\prime}}\int_{\Omega\setminus\Omega^{\prime}}\frac{D_{\varepsilon}(\xi)\eta_{\varepsilon}(\xi)|u_{\varepsilon}^{(1)}(x)|^{2_{\mu}^\ast}}
{|x-\xi|^{\mu}}d\xi dx\Big]\\&\hspace{4mm}
+\mu\Big[\int_{ \Omega^{\prime}}\int_{\Omega\setminus\Omega^{\prime}}x\cdot(x-\xi)\frac{|u_{\varepsilon}^{(2)}(\xi)|^{2_{\mu}^\ast}\tilde{C}_{\varepsilon}(x)\eta_{\varepsilon}(x)}
{|x-\xi|^{\mu+2}}d\xi dx+\int_{ \Omega^{\prime}}\int_{\Omega\setminus\Omega^{\prime}}x\cdot(x-\xi)\frac{D_{\varepsilon}(\xi)\eta_{\varepsilon}(\xi)|u_{\varepsilon}^{(1)}(x)|^{2_{\mu}^\ast}}
{|x-\xi|^{\mu+2}}d\xi dx\Big]
\\&
\hspace{6mm}+\int_{\partial\Omega^{\prime}}\int_{ \Omega\setminus\Omega^{\prime}}\frac{|u_{\varepsilon}^{(2)}(\xi)|^{2_{\mu}^\ast}\tilde{C}_{\varepsilon}(x)\eta_{\varepsilon}(x)}
{|x-\xi|^{\mu}}\big\langle x-x_{\varepsilon}^{(1)},\nu\big\rangle d\xi ds+\int_{\partial\Omega^{\prime}}\int_{ \Omega\setminus\Omega^{\prime}}\frac{D_{\varepsilon}(\xi)\eta_{\varepsilon}(\xi)|u_{\varepsilon}^{(1)}(x)|^{2_{\mu}^\ast}}
{|x-\xi|^{\mu}}\big\langle x-x_{\varepsilon}^{(1)},\nu\big\rangle d\xi ds
\\&\hspace{8mm}+2\int_{\partial\Omega^{\prime}}\int_{ \Omega^{\prime}}\frac{|u_{\varepsilon}^{(2)}(\xi)|^{2_{\mu}^\ast}\tilde{C}_{\varepsilon}(x)\eta_{\varepsilon}(x)}
{|x-\xi|^{\mu}}\big\langle x-x_{\varepsilon}^{(1)},\nu\big\rangle d\xi ds+2\int_{\partial\Omega^{\prime}}\int_{ \Omega^{\prime}}\frac{D_{\varepsilon}(\xi)\eta_{\varepsilon}(\xi)|u_{\varepsilon}^{(1)}(x)|^{2_{\mu}^\ast}}
{|x-\xi|^{\mu}}\big\langle x-x_{\varepsilon}^{(1)},\nu\big\rangle d\xi ds\\&
\hspace{10mm}+\frac{\varepsilon}{2}\int_{\partial\Omega^{\prime}} \big(u_{\varepsilon}^{(1)}+u_{\varepsilon}^{(2)}\big)\eta_{\varepsilon}\big\langle x-x_{\varepsilon}^{(1)},\nu\big\rangle ds-\varepsilon\int_{\Omega^{\prime}} \big(u_{\varepsilon}^{(1)}+u_{\varepsilon}^{(2)}\big)\eta_{\varepsilon}dx,
\end{split}
\end{equation}
where $\Omega^{\prime}\subset\Omega$ is a smooth domain, $\nu=\nu(x)$ denotes the unit outward normal to the boundary $\partial \Omega^{\prime}$ and
\begin{equation*}
\tilde{C}_{\varepsilon}(x)=\int_{0}^1\big(tu_{\varepsilon}^{(1)}(x)+(1-t)u_{\varepsilon}^{(2)}(x)\big)^{2_{\mu}^{\ast}-1}dt,~~~D_{\varepsilon}(\xi)=\int_{0}^1\big(tu_{\varepsilon}^{(1)}(\xi)+(1-t)u_{\varepsilon}^{(2)}(\xi)\big)^{2_{\mu}^{\ast}-1}dt.
\end{equation*}
\end{Lem}
\begin{proof}
In view of Lemma~\ref{Lem2.1}, taking $u_{\varepsilon}=u_{\varepsilon}^{(j)}$ with $j=1,2$ in \eqref{PH}, and then making a difference between those respectively.
By direct calculations, we can obtain \eqref{ph1}.
Similarly, taking $u_{\varepsilon}=u_{\varepsilon}^{(j)}$ with $j=1,2$ in \eqref{eq2.6}, and then making a difference between those respectively, we can also derive that \eqref{ph2}.
\end{proof}

Now we are ready to prove Lemma~\ref{Lem6.6} by using the Local Pohozaev identities.\\
\textbf{Proof of~Lemma~\ref{Lem6.6}.} We divide the argument into two steps:\\
{\bf Step 1.}
We prove that $c_k=0,k=1,\cdots,N$. \\
We define the following quadratic form
\begin{equation*}
\mathcal{P}(\eta,u,\tau)=\mathcal{P}(\eta,u,B_{\tau}(x_{\varepsilon}^{(1)}))=-\int_{\partial B_{\tau}(x_{\varepsilon}^{(1)})}\frac{\partial u}{\partial x_j}\frac{\partial\eta}{\partial\nu}ds-\int_{\partial B_{\tau}(x_{\varepsilon}^{(1)})}\frac{\partial u}{\partial \nu}\frac{\partial\eta}{\partial x_j}ds+\int_{\partial B_{\tau}(x_{\varepsilon}^{(1)})}\big\langle\nabla \eta,\nabla u\big\rangle\nu_jds.
\end{equation*}
For $N\geq6$, taking $\Omega^{\prime}=B_{\tau}(x_{\varepsilon}^{(1)})$ in \eqref{ph1},  by \eqref{422} and \eqref{le641}, we know
\begin{equation}\label{LHS}
\begin{split}
\text{LHS of}~\eqref{ph1}&=\frac{A_{N,\mu}\big((2_{\mu}^{\ast}-1)A_{\varepsilon}^{(1)}+2_{\mu}^{\ast}A_{\varepsilon}^{(2)}\big)\mathcal{P}\Big(G(x_{\varepsilon}^{(1)},x),G(x_{\varepsilon}^{(1)},x),\tau\Big)}{(\lambda_{\varepsilon}^{(1)})^{\frac{N-2}{2}}}
\\&~~~+\sum_{l=1}^{N}\frac{A_{N,\mu}\big((2_{\mu}^{\ast}-1)B_{\varepsilon,l}^{(1)}+2_{\mu}^{\ast}B_{\varepsilon,l}^{(2)}\big)\mathcal{P}\Big(G(x_{\varepsilon}^{(1)},x),\partial_{l}G(x_{\varepsilon}^{(1)},x),\tau\Big)}{(\lambda_{\varepsilon}^{(1)})^{\frac{N-2}{2}}}+O\Big(\frac{\ln \lambda_{\varepsilon}^{(1)}}{(\lambda_{\varepsilon}^{(1)})^{\frac{3N-2}{2}}}\Big).
\end{split}
\end{equation}
We next estimate $A_{\varepsilon}^{(1)}$ and $A_{\varepsilon}^{(2)}$, respectively.
In fact,
\begin{equation*}
\begin{split}
A_{\varepsilon}^{(1)}&=O\Big(\int_{B_{\tau}(x_{\varepsilon}^{(1)})}\int_{B_{\tau}(x_{\varepsilon}^{(1)})}\frac{U_{x_{\varepsilon}^{(1)},\lambda_{\varepsilon}^{(1)}}^{2_{\mu}^\ast}(\xi)C_{\varepsilon}(x)\eta_{\varepsilon}(x)}{|x-\xi|^{\mu}}dxd\xi\Big)
=O\Big(\frac{1}{(\lambda_{\varepsilon}^{(1)})^{N-2}}\Big).
\end{split}
\end{equation*}
since
\begin{equation*}
\begin{split}
\int_{B_{\tau}(x_{\varepsilon}^{(1)})}\frac{(\lambda_{\varepsilon}^{(1)})^2}{(1+\lambda_{\varepsilon}^{(1)}|x-x_{\varepsilon}^{(1)}|)^{N+2}}dx
&=O\Big(\frac{1}{(\lambda_{\varepsilon}^{(1)})^{N-2}}\int_{0}^{\lambda_{\varepsilon}^{(1)}\tau}\frac{r^{N-1}}{(1+r)^{N+2}}dr\Big)=O\Big(\frac{1}{(\lambda_{\varepsilon}^{(1)})^{N-2}}\Big),
\end{split}
\end{equation*}
Also, we have
\begin{equation*}
\begin{split}
A_{\varepsilon}^{(2)}&=O\Big(\int_{B_{\tau}(x_{\varepsilon}^{(1)})}\int_{B_{\tau}(x_{\varepsilon}^{(1)})}\frac{D_{\varepsilon}(\xi)\eta_{\varepsilon}(\xi)(U_{x_\varepsilon^{(2)},\lambda_{\varepsilon}^{(2)}}(x))^{2_{\mu}^\ast-1}}{|x-\xi|^{\mu}}dxd\xi\Big)\\&
=O\Big(\big(\int_{B_{\tau}(x_{\varepsilon}^{(1)})}(D_{\varepsilon}(\xi)\eta_{\varepsilon}(\xi))^{\frac{2N}{2N-\mu}}d\xi\big)^{\frac{2N-\mu}{2N}}\big(\int_{B_{\tau}(x_{\varepsilon}^{(1)})}U_{x_\varepsilon^{(1)},\lambda_{\varepsilon}^{(1)}}^{\frac{2N(2_{\mu}^\ast-1)}{2N-\mu}}(x)dx\big)^{\frac{2N-\mu}{2N}}\Big)\\&
=O\Big(\frac{1}{(\lambda_{\varepsilon}^{(1)})^{N-2}}\big(\int_{B_{\lambda_{\varepsilon}^{(1)}\tau}(0)}\frac{1}{(1+|x|^2)^{\frac{N(N-\mu+2)}{2N-\mu}}}dx\big)^{\frac{2N-\mu}{N}}\Big)\\&
=O\Big(\frac{1}{(\lambda_{\varepsilon}^{(1)})^{N-2}}\Big),
\end{split}
\end{equation*}
since $\frac{2N(N-\mu+2)}{2N-\mu}>N$.
On the other hand, from \cite{Cao-Peng-Yan2021}, we know
\begin{equation}\label{G1}
\begin{split}
\mathcal{P}\Big(G(x_{\varepsilon}^{(1)},x),G(x_{\varepsilon}^{(1)},x),\tau\Big)=-\frac{\partial \mathcal{R}(x_{\varepsilon}^{(1)})}{\partial x_i}.
\end{split}
\end{equation}
Hence it follow from \eqref{R1} that
\begin{equation}\label{e538}
\frac{\big((2_{\mu}^{\ast}-1)A_{\varepsilon}^{(1)}+2_{\mu}^{\ast}A_{\varepsilon}^{(2)}\big)\mathcal{P}\Big(G(x_{\varepsilon}^{(1)},x),G(x_{\varepsilon}^{(1)},x),\tau\Big)}{(\lambda_{\varepsilon}^{(1)})^{\frac{N-2}{2}}}
=O\Big(\frac{\ln\lambda_{\varepsilon}^{(1)}}{(\lambda_{\varepsilon}^{(1)})^{\frac{3N-2}{2}}}\Big).
\end{equation}

Next we are going to estimate each term of the right hand side of \eqref{ph1} with $\Omega^{\prime}=B_{\tau}(x_{\varepsilon}^{(1)})$. We define
\begin{equation*}
\begin{split}
&P_1=\frac{1}{2_{\mu}^\ast}\int_{\partial B_{\tau}(x_{\varepsilon}^{(1)})}\int_{ \Omega}\frac{|u_{\varepsilon}^{(2)}(\xi)|^{2_{\mu}^\ast}\tilde{C}_{\varepsilon}(x)\eta_{\varepsilon}(x)\nu_j}
{|x-\xi|^{\mu}}d\xi ds,~~~P_2=\frac{1}{2_{\mu}^\ast}\int_{\partial B_{\tau}(x_{\varepsilon}^{(1)})}\int_{ \Omega}\frac{D_{\varepsilon}(\xi)\eta_{\varepsilon}(\xi)|u_{\varepsilon}^{(1)}(x)|^{2_{\mu}^\ast}\nu_j}
{|x-\xi|^{\mu}}d\xi ds,\\&
P_3=\frac{1}{2_{\mu}^\ast}\int_{\partial B_{\tau}(x_{\varepsilon}^{(1)})}\int_{ B_{\tau}(x_{\varepsilon}^{(1)})}\Big[\frac{|u_{\varepsilon}^{(2)}(\xi)|^{2_{\mu}^\ast}\tilde{C}_{\varepsilon}(x)\eta_{\varepsilon}(x)\nu_j}
{|x-\xi|^{\mu}}+\frac{D_{\varepsilon}(\xi)\eta_{\varepsilon}(\xi)|u_{\varepsilon}^{(1)}(x)|^{2_{\mu}^\ast}\nu_j}
{|x-\xi|^{\mu}}\Big]d\xi ds,\\&
P_4=\frac{\mu}{2_{\mu}^\ast}\int_{B_{\tau}(x_{\varepsilon}^{(1)})}\int_{\Omega\setminus B_{\tau}(x_{\varepsilon}^{(1)})}(x_j-\xi_j)\frac{|u_{\varepsilon}^{(2)}(\xi)|^{2_{\mu}^\ast}\tilde{C}_{\varepsilon}(x)\eta_{\varepsilon}}
{|x-\xi|^{\mu+2}}dxd\xi,~P_5=\frac{\varepsilon}{2}\int_{\partial B_{\tau}(x_{\varepsilon}^{(1)})} \big(u_{\varepsilon}^{(1)}+u_{\varepsilon}^{(2)}\big)\eta_{\varepsilon}\nu_jds,\\&P_6=\frac{\mu}{2_{\mu}^\ast}\int_{B_{\tau}(x_{\varepsilon}^{(1)})}\int_{\Omega\setminus B_{\tau}(x_{\varepsilon}^{(1)})}(x_j-\xi_j)\frac{D_{\varepsilon}(\xi)\eta_{\varepsilon}(\xi)|u_{\varepsilon}^{(1)}(x)|^{2_{\mu}^\ast}}
{|x-\xi|^{\mu+2}}dxd\xi.
\end{split}
\end{equation*}
Firstly, we can deduce
\begin{equation*}
\tilde{C}_{\varepsilon}(x)\leq C\frac{(\lambda_{\varepsilon}^{(1)})^\frac{N-\mu+2}{2}}{\big(1+(\lambda_{\varepsilon}^{(1)})^2|x-x_{\varepsilon}^{(1)}|^2\big)^\frac{N-\mu+2}{2}}~~~\text{and}~~~D_{\varepsilon}(\xi)\leq C\frac{(\lambda_{\varepsilon}^{(1)})^\frac{N-\mu+2}{2}}{\big(1+(\lambda_{\varepsilon}^{(1)})^2|\xi-x_{\varepsilon}^{(1)}|^2\big)^\frac{N-\mu+2}{2}}.
\end{equation*}
Then we have
\begin{equation}\label{C1} \tilde{C}_{\varepsilon}(x)=O\Big(\frac{1}{(\lambda_{\varepsilon}^{(1)})^{\frac{N-\mu+2}{2}}}\Big)~~~\text{and}~~~D_{\varepsilon}(\xi)=O\Big(\frac{1}{(\lambda_{\varepsilon}^{(1)})^{\frac{N-\mu+2}{2}}}\Big),~~\text{in}~\Omega\setminus B_{\tau}(x_{\varepsilon}^{(1)}).
\end{equation}
Together with \eqref{e29}, \eqref{UUA} and  \eqref{e6.3}, we obtain
\begin{equation*}
\begin{split}
P_{1}&=O\Big(\int_{\partial B_{\tau}(x_{\varepsilon}^{(1)})}\int_{ \Omega}\frac{|U_{x_\varepsilon^{(1)},\lambda_\varepsilon^{(1)}}(\xi)|^{2_{\mu}^\ast}\tilde{C}_{\varepsilon}(x)\eta_{\varepsilon}(x)\nu_j}
{|x-\xi|^{\mu}}d\xi ds\Big)\\&
=O\Big(\int_{\partial B_{\tau}(x_{\varepsilon}^{(1)})}\frac{(\lambda_{\varepsilon}^{(1)})^\frac{N+2}{2}}{\big(1+\lambda_{\varepsilon}^{(1)}|x-x_{\varepsilon}^{(1)}|\big)^{N+2}}\eta_{\varepsilon}(x)\nu_jds\Big)\\&
=O\Big(\frac{\ln\lambda_{\varepsilon}^{(1)}}{(\lambda_{\varepsilon}^{(1)})^{\frac{3N-2}{2}}}\Big),
\end{split}
\end{equation*}
and
\begin{equation*}
\begin{split}
P_{2}&=
O\Big(\big(\int_{\partial B_{\tau}(x_{\varepsilon}^{(1)})}\big(U_{x_{\varepsilon}^{(1)},\lambda_{\varepsilon}^{(1)}}^{2_{\mu}^{\ast}}\nu_i\big)^{\frac{2N}{2N-\mu}}ds\big)^{\frac{2N-\mu}{2N}}\big(\int_{ \Omega}|D_{\varepsilon}(\xi)\eta_{\varepsilon}(\xi)|^{\frac{2N}{2N-\mu}}dx\big)^{\frac{2N-\mu}{2N}}\Big)\\&
=O\Big(\frac{1}{(\lambda_{\varepsilon}^{(1)})^N}\Big)\Big(\int_{ \Omega}\frac{(\lambda_{\varepsilon}^{(1)})^\frac{N(N-\mu+2)}{2N-\mu}}{\big(1+(\lambda_{\varepsilon}^{(1)})|x-x_{\varepsilon}^{(1)}|\big)^\frac{2N(N-\mu+2)}{2N-\mu}}dx\Big)^{\frac{2N-\mu}{2N}}\\&
=O\Big(\frac{1}{(\lambda_{\varepsilon}^{(1)})^{\frac{3N-2}{2}}}\Big).
\end{split}
\end{equation*}
Similar to the above estimates, we can also prove
\begin{equation*}
P_3=O\Big(\frac{\ln\lambda_{\varepsilon}^{(1)}}{(\lambda_{\varepsilon}^{(1)})^{\frac{3N-2}{2}}}\Big).
\end{equation*}
Because of oddness, we can find
\begin{equation*}
\begin{split}
\int_{B_{\tau}(x_{\varepsilon}^{(1)})}(x_j-\xi_j)\frac{|U_{x_{\varepsilon}^{(1)},\lambda_{\varepsilon}^{(1)}}(\xi)|^{2_{\mu}^\ast}}
{|x-\xi|^{\mu+2}}d\xi=\int_{\Omega\setminus B_{\tau}(x_{\varepsilon}^{(1)})}\frac{x_j-\xi_j}{\big(1+|\xi-x_{\varepsilon}^{(1)}|\big)^{2N-\mu}}d\xi=0.
\end{split}
\end{equation*}
This means that $P_4=P_6=0$.
Moreover, note that $\varepsilon=O\big(\frac{1}{(\lambda_{\varepsilon}^{(1)})^{N-4}}\big)=O\big(\frac{1}{(\lambda_{\varepsilon}^{(1)})^{2}}\big)$ if $N\geq6$, so we have
\begin{equation*}
P_5=O\Big(\frac{1}{(\lambda_{\varepsilon}^{(1)})^{\frac{3N-2}{2}}}\Big).
\end{equation*}
Hence we know that
\begin{equation*}
\text{RHS of}~\eqref{ph1}=O\Big(\frac{\ln\lambda_{\varepsilon}^{(1)}}{(\lambda_{\varepsilon}^{(1)})^{\frac{3N-2}{2}}}\Big).
\end{equation*}
Then it follows from \eqref{LHS} that
\begin{equation*}
\sum_{l=1}^{N}\frac{A_{N,\mu}\big((2_{\mu}^{\ast}-1)B_{\varepsilon,l}^{(1)}+2_{\mu}^{\ast}B_{\varepsilon,l}^{(2)}\big)\mathcal{P}\Big(G(x_{\varepsilon}^{(1)},x),\partial_{l}G(x_{\varepsilon}^{(1)},x),\tau\Big)}{(\lambda_{\varepsilon}^{(1)})^{\frac{N-2}{2}}}=o\Big(\frac{1}{(\lambda_{\varepsilon}^{(1)})^{\frac{3N-4}{2}}}\Big).
\end{equation*}
Using
 the estimate (see \cite{Cao-Peng-Yan2021})
\begin{equation}\label{G2}
\begin{split}
\mathcal{P}\Big(G(x_{\varepsilon}^{(1)},x),\partial_{l}G(x_{\varepsilon}^{(1)},x),\tau\Big)=-\frac{\partial^2\mathcal{R}(x_{\varepsilon}^{(1)})}{\partial x_{i}\partial x_l}
\end{split}
\end{equation}
and $x_0$ is a nondegenerate critical point of Robin function $\mathcal{R}(x)$, we see that
\begin{equation}\label{b546}
(2_{\mu}^{\ast}-1)B_{\varepsilon,l}^{(1)}+2_{\mu}^{\ast}B_{\varepsilon,l}^{(2)}=o\Big(\frac{1}{(\lambda_{\varepsilon}^{(1)})^{N-1}}\Big).
\end{equation}

On the other hand, we consider that the estimates of $B_{\varepsilon,l}^{(1)}$ and $B_{\varepsilon,l}^{(2)}$ in \eqref{424}-\eqref{4244}. Using the elementary inequality \eqref{ele} in Appendix A,
then we know that
\begin{equation}\label{B11}
\begin{split}
B_{\varepsilon,l}^{(1)}&=\int_{B_{\tau}(x_{\varepsilon}^{(1)})}\int_{B_{\tau}(x_{\varepsilon}^{(1)})}(z_l-x_{\varepsilon,l}^{(1)})\frac{(u_{\varepsilon}^{(1)}(\xi))^{2_{\mu}^\ast}C_{\varepsilon}(z)\eta_{\varepsilon}(z)}{|z-\xi|^{\mu}}dzd\xi\\&
=\frac{1}{(\lambda_{\varepsilon}^{(1)})^{2N-\mu+1}}\Big(\mathcal{G}_1+2_{\mu}^{\ast}\mathcal{G}_2+\frac{2_{\mu}^{\ast}(2_{\mu}^{\ast}-1)}{2}\mathcal{G}_3+O\big(\mathcal{G}_4\big)\Big),
\end{split}
\end{equation}
where
\begin{equation}\label{B12}
\begin{split}
&\mathcal{G}_1=\int_{B_{\lambda_{\varepsilon}^{(1)}\tau}(0)}\int_{B_{\lambda_{\varepsilon}^{(1)}\tau}(0)}z_l\frac{\big(PU_{x_{\varepsilon}^{(1)},\lambda_{\varepsilon}^{(1)}}(\frac{\xi}{\lambda_{\varepsilon}^{(1)}}+x_{\varepsilon}^{(1)})\big)^{2_{\mu}^\ast}C_{\varepsilon}(\frac{z}{\lambda_{\varepsilon}^{(1)}}+x_{\varepsilon}^{(1)})\tilde{\eta}_{\varepsilon}(z)}{|z-\xi|^{\mu}}dzd\xi,\\&
\mathcal{G}_2=\int_{B_{\lambda_{\varepsilon}^{(1)}\tau}(0)}\int_{B_{\lambda_{\varepsilon}^{(1)}\tau}(0)}z_l\frac{\big(PU_{x_{\varepsilon}^{(1)},\lambda_{\varepsilon}^{(1)}}(\frac{\xi}{\lambda_{\varepsilon}^{(1)}}+x_{\varepsilon}^{(1)})\big)^{2_{\mu}^\ast-1}w_{\varepsilon}^{(1)}(\frac{\xi}{\lambda_{\varepsilon}^{(1)}}+x_{\varepsilon}^{(1)})C_{\varepsilon}(\frac{z}{\lambda_{\varepsilon}^{(1)}}+x_{\varepsilon}^{(1)})\tilde{\eta}_{\varepsilon}(z)}{|z-\xi|^{\mu}}dzd\xi,\\&
\mathcal{G}_3=\int_{B_{\lambda_{\varepsilon}^{(1)}\tau}(0)}\int_{B_{\lambda_{\varepsilon}^{(1)}\tau}(0)}z_l\frac{\big(PU_{x_{\varepsilon}^{(1)},\lambda_{\varepsilon}^{(1)}}(\frac{\xi}{\lambda_{\varepsilon}^{(1)}}+x_{\varepsilon}^{(1)})\big)^{2_{\mu}^\ast-2}\big(w_{\varepsilon}^{(1)}(\frac{\xi}{\lambda_{\varepsilon}^{(1)}}+x_{\varepsilon}^{(1)})\big)^2C_{\varepsilon}(\frac{z}{\lambda_{\varepsilon}^{(1)}}+x_{\varepsilon}^{(1)})\tilde{\eta}_{\varepsilon}(z)}{|z-\xi|^{\mu}}dzd\xi,\\&
\mathcal{G}_4=\int_{B_{\lambda_{\varepsilon}^{(1)}\tau}(0)}\int_{B_{\lambda_{\varepsilon}^{(1)}\tau}(0)}z_l\frac{\big(w_{\varepsilon}^{(1)}(\frac{\xi}{\lambda_{\varepsilon}^{(1)}}+x_{\varepsilon}^{(1)})\big)^{2_{\mu}^{\ast}}C_{\varepsilon}(\frac{z}{\lambda_{\varepsilon}^{(1)}}+x_{\varepsilon}^{(1)})\tilde{\eta}_{\varepsilon}(z)}{|z-\xi|^{\mu}}dzd\xi.
\end{split}
\end{equation}
Then by Lemmas~\ref{LC1}, \ref{LC2} and \ref{LC3} in Appendix B, we get
\begin{equation}\label{B13}
B_{\varepsilon,l}^{(1)}=-\frac{c_l}{2^{\ast}-1}\frac{N(N-2)}{A_{H,L}}\frac{1}{\big(\lambda_{\varepsilon}^{(1)}\big)^{N-1}}\int_{\R^N}U_{0,1}^{\frac{N+2}{N-2}}(z)dz+o\Big(\frac{1}{(\lambda_{\varepsilon}^{(1)})^{N-1}}\Big), ~~~\text{for}~l=1,2,\cdots,N.
\end{equation}

Noting that
\begin{equation}\label{A11}
\begin{split}
B_{\varepsilon,l}^{(2)}&=\int_{B_{\delta}(x_{\varepsilon}^{(1)})}\int_{B_{\delta}(x_{\varepsilon}^{(1)})}(z_l-x_{\varepsilon,l}^{(1)})\frac{D_{\varepsilon}(\xi)\eta_{\varepsilon}(\xi)(u_{\varepsilon}^{(2)}(z))^{2_{\mu}^\ast-1}}{|z-\xi|^{\mu}}dzd\xi\\&
=\mathcal{H}_1+(2_{\mu}^{\ast}-1)\mathcal{H}_2+O\big(\mathcal{H}_3\big),
\end{split}
\end{equation}
where
\begin{equation}\label{A12}
\begin{split}
&\mathcal{H}_1=\int_{B_{\delta}(x_{\varepsilon}^{(1)})}\int_{B_{\delta}(x_{\varepsilon}^{(1)})}(z_l-x_{\varepsilon,l}^{(1)})\frac{D_{\varepsilon}(\xi)\eta_{\varepsilon}(\xi)\big(PU_{x_{\varepsilon}^{(2)},\lambda_{\varepsilon}^{(2)}}(z)\big)^{2_{\mu}^\ast-1}}{|z-\xi|^{\mu}}dzd\xi,\\&
\mathcal{H}_2=\int_{B_{\delta}(x_{\varepsilon}^{(1)})}\int_{B_{\delta}(x_{\varepsilon}^{(1)})}(z_l-x_{\varepsilon,l}^{(1)})\frac{D_{\varepsilon}(\xi)\eta_{\varepsilon}(\xi)\big(PU_{x_{\varepsilon}^{(2)},\lambda_{\varepsilon}^{(2)}}(z)\big)^{2_{\mu}^\ast-2}w_{\varepsilon}^{(2)}(z)}{|z-\xi|^{\mu}}dzd\xi,\\&
\mathcal{H}_3=\int_{B_{\delta}(x_{\varepsilon}^{(1)})}\int_{B_{\delta}(x_{\varepsilon}^{(1)})}(z_l-x_{\varepsilon,l}^{(1)})\frac{D_{\varepsilon}(\xi)\eta_{\varepsilon}(\xi)\big(w_{\varepsilon}^{(2)}(z)\big)^{2_{\mu}^{\ast}-1}}{|z-\xi|^{\mu}}dzd\xi.
\end{split}
\end{equation}
Then by Lemma~\ref{LD1} in Appendix C, we get
\begin{equation}\label{aa13}
B_{\varepsilon,l}^{(2)}=o\Big(\frac{1}{(\lambda_{\varepsilon}^{(1)})^{N-1}}\Big).
\end{equation}
Thus by \eqref{b546}, \eqref{B13} and \eqref{aa13} imply $c_k=0$, $k=1,2,\cdots,N$.

{\bf Step 2.}
We prove that $c_0=0$. \\
First we define the following quadratic form
\begin{equation*}
\begin{split}
\mathcal{Q}(\eta,u,\tau)&=-\int_{\partial B_{\tau}(x_{\varepsilon}^{(1)})}\big\langle\nabla\eta,\nu\big\rangle\big\langle x-x_{\varepsilon}^{(1)},\nabla u\big\rangle ds\\&~~+\frac{1}{2}\int_{\partial B_{\tau}(x_{\varepsilon}^{(1)})}\big\langle\nabla\eta,\nabla u\big\rangle\big\langle x-x_{\varepsilon}^{(1)},\nu\big\rangle ds+\frac{2-N}{2}\int_{\partial B_{\tau}(x_{\varepsilon}^{(1)})}\big\langle\nabla \eta,\nu\big\rangle uds.
\end{split}
\end{equation*}
Taking $\Omega^{\prime}=B_{\tau}(x_{\varepsilon}^{(1)})$ in \eqref{ph2}, from \eqref{422} and \eqref{le641}, we have
\begin{equation*}
\begin{split}
\text{LHS of}~\eqref{ph2}&=\frac{2A_{N,\mu}\big((2_{\mu}^{\ast}-1)A_{\varepsilon}^{(1)}+2_{\mu}^{\ast}A_{\varepsilon}^{(2)}\big)\mathcal{Q}\Big(G(x_{\varepsilon}^{(1)},x),G(x_{\varepsilon}^{(1)},x),\tau\Big)}{(\lambda_{\varepsilon}^{(1)})^{\frac{N-2}{2}}}.
\end{split}
\end{equation*}
Since we have the estimate (see \cite{Cao-Peng-Yan2021})
\begin{equation*}
\mathcal{Q}\Big(G(x_{\varepsilon}^{(1)},x),G(x_{\varepsilon}^{(1)},x),\tau\Big)=-\frac{(N-2)}{2}\mathcal{R}(x_{\varepsilon}^{(1)}),
\end{equation*}
which implies that
\begin{equation*}
\text{LHS of}~\eqref{ph2}=-\frac{A_{N,\mu}\big((2_{\mu}^{\ast}-1)A_{\varepsilon}^{(1)}+2_{\mu}^{\ast}A_{\varepsilon}^{(2)}\big)(N-2)\mathcal{R}(x_{\varepsilon}^{(1)})}{(\lambda_{\varepsilon}^{(1)})^{\frac{N-2}{2}}}.
\end{equation*}
Note that by \eqref{R15}, we know
\begin{equation*}
A_{N,\mu}=\frac{N(N-2)}{A_{H,L}}\int_{\mathbb{R}^N}U_{0,1}^{2^\ast-1}(z)dz
+o\big(1\big).
\end{equation*}
On the other hand, from Lemma~\ref{LE1} in Appendix D, we can find
\begin{equation*}
\begin{split}
A_{\varepsilon}^{(1)}+A_{\varepsilon}^{(2)}&=\int_{B_{\delta}(x_{\varepsilon}^{(1)})}\int_{B_{\delta}(x_{\varepsilon}^{(1)})}\frac{(u_{\varepsilon}^{(1)}(\xi))^{2_{\mu}^\ast}C_{\varepsilon}(z)\eta_{\varepsilon}(z)}{|z-\xi|^{\mu}}dzd\xi\\&
=\frac{1}{\big(\lambda_{\varepsilon}^{(1)}\big)^{N-2}}\frac{N(N-2)}{A_{H,L}}\int_{\R^N}U_{0,1}^{\frac{4}{N-2}}(z)c_0\phi_0dz+o\Big(\frac{1}{(\lambda_{\varepsilon}^{(1)})^{N-2}}\Big).
\end{split}
\end{equation*}
A direct calculation, we can also find
\begin{equation*}
(2^{\ast}-1)\int_{\R^N}U_{0,1}^{\frac{4}{N-2}}\phi_{0}dz=-\frac{N-2}{2}\int_{\R^N}U_{0,1}^{2^{\ast}-1}(z)dz.
\end{equation*}
Therefore, together with the above estimates, we can deduce
\begin{equation*}
\begin{split}
\text{LHS of}~\eqref{ph2}&
=\frac{N^2(N-2)^4(N-\mu+2)\mathcal{R}(x_{\varepsilon}^{(1)})}{2(A_{H,L})^2(N+2)}\frac{1}{(\lambda_{\varepsilon}^{(1)})^{\frac{3N-6}{2}}}\Big(\int_{\mathbb{R}^N}U_{0,1}^{2^\ast-1}(z)dz\Big)^2c_{0}+o\Big(\frac{1}{(\lambda_{\varepsilon}^{(1)})^{\frac{N-2}{2}}}\Big)
\end{split}
\end{equation*}
From Lemma~\ref{LF1} in Appendix D, we know
\begin{equation*}
\text{RHS of}~\eqref{ph2}=\frac{2\varepsilon}{\big(\lambda_{\varepsilon}^{(1)}\big)^{\frac{N+2}{2}}}\Big(\int_{\R^N}U_{0,1}^2(z)dz\Big)c_0+o\Big(\frac{1}{(\lambda_{\varepsilon}^{(1)})^{\frac{N-2}{2}}}\Big).
\end{equation*}
As a result,
\begin{equation}\label{559}
\begin{split}
\frac{N^2(N-2)^4(N-\mu+2)\mathcal{R}(x_{\varepsilon}^{(1)})}{2(A_{H,L})^2(N+2)}\frac{1}{(\lambda_{\varepsilon}^{(1)})^{N-2}}&\Big(\int_{\mathbb{R}^N}U_{0,1}^{2^\ast-1}(z)dz\Big)^2c_{0}
\\&=\frac{2\varepsilon}{\big(\lambda_{\varepsilon}^{(1)}\big)^{2}}\Big(\int_{\R^N}U_{0,1}^2(z)dz\Big)c_0+o\big(1\big).
\end{split}
\end{equation}
Noting that, from the proof of Lemma~\ref{R1} in section~$3$, we can find the basic estimate
\begin{equation}\label{5510}
\begin{split}
\frac{N^2(N-2)^3\mathcal{R}(x_{\varepsilon}^{(1)})}{2(A_{H,L})^2}\frac{1}{(\lambda_{\varepsilon}^{(1)})^{N-2}}\Big(\int_{\mathbb{R}^N}U_{0,1}^{2^\ast-1}(z)dz
+O\big(\frac{1}{(\lambda_{\varepsilon}^{(1)})^2}\big)\Big)^2&=\frac{\varepsilon}{(\lambda_{\varepsilon}^{(1)})^2}\Big(\int_{\R^N}U_{0,1}^2(z)dz+O\big(\frac{1}{(\lambda_{\varepsilon}^{(1)})^2}\big)\Big)
\\&
\hspace{9mm}+O\Big(\frac{\varepsilon }{(\lambda_{\varepsilon}^{(1)})^{N-2}}+\frac{1}{(\lambda_{\varepsilon}^{(1)})^{N}}\Big).
\end{split}
\end{equation}
Then \eqref{559} and \eqref{5510} imply that $c_0=0$. This finishes the proof of Lemma.
$\hfill{} \Box$

\appendix
\section{Estimates of $A_{N,\mu}$ and $\mathcal{F}_2$, $\mathcal{F}_3$, $\mathcal{F}_{4}$ in \eqref{e417}}
In this section, we give that have been used in the previous sections.
Let recall that
\begin{equation*}
\psi_{z,\lambda}(x)=U_{z,\lambda}-PU_{z,\lambda},~~~ U_{z,\lambda}(x)=\big(\frac{\lambda}{1+\lambda^2|x-z|^2}\big)^{\frac{N-2}{2}}.
\end{equation*}
Some basic estimates as follow:
\begin{Lem}\label{Lem9.1}
\begin{equation*}
\begin{split}
\frac{\partial U_{z,\lambda}(x)}{\partial{z_j}}&=(N-2)\lambda^{\frac{N+2}{2}}
\frac{x_{j}-z_j}{\big(1+\lambda^2|x-z|^2\big)^{\frac{N}{2}}}
=O\big(\lambda U_{z,\lambda}\big),
\end{split}
\end{equation*}
\begin{equation*}
\begin{split}
\frac{\partial U_{z,\lambda}(x)}{\partial{\lambda}}&=\frac{N-2}{2}\lambda^{\frac{N-4}{2}}
\frac{1-\lambda^2|x-z|^2}{(1+\lambda^2|x-z|^2)^{\frac{N}{2}}}
=O\big(\frac{U_{z,\lambda}}{\lambda}\big),
\end{split}
\end{equation*}
\begin{equation*}
\|\psi_{z,\lambda}\|_{L^\infty}=O\big(\frac{1}{\lambda^{\frac{N-2}{2}}d^{N-2}}\big).
\end{equation*}
where where $d=dist(x, \partial\Omega)$ is the distance between $x$ and the boundary of $\Omega$.
\end{Lem}
\begin{proof}
This follows from the definition of $U_{z,\lambda}$, $PU_{z,\lambda}$, $\psi_{z,\lambda}$ and direct computations. See also \cite{Rey-1990}.
\end{proof}

\begin{Lem}\label{FPU} It holds
\begin{equation*}
\psi_{x_{\varepsilon},\lambda_{\varepsilon}}=O\Big(\frac{1}{(\lambda_{\varepsilon})^{\frac{N-2}{2}}}\Big),~~~\text{in}~~C^{1}(\Omega) ~~~\text{and}~~PU_{x_{\varepsilon},\lambda_{\varepsilon}}=O\Big(\frac{1}{(\lambda_{\varepsilon})^{\frac{N-2}{2}}}\Big),~~~\text{in}~~C^{1}\big(\Omega\setminus B_{\delta}(x_{\varepsilon})\big),
\end{equation*}
where $\delta>0$ is any small fixed constant.
\end{Lem}
\begin{proof}
For a proof of this lemma, we refer to \cite{Rey-1990}.
\end{proof}
\begin{Lem}\label{AA3}
 It holds
\begin{equation}\label{eqw}
	\|w_\varepsilon\|_{H_{0}^{1}}=\begin{cases}
		O\bigg(\frac{1}{\lambda_\varepsilon ^{N-2}}+\frac{\varepsilon}{\lambda_\varepsilon^{\frac{N-2}{2}}}\bigg),~~~&\text{if}~N<6-\mu,
		\\
		O\bigg(\frac{(\ln\lambda_\varepsilon )^{\frac{8-2\mu}{12-\mu}}}{\lambda_\varepsilon ^{4-\mu}}+\frac{\varepsilon(\ln\lambda_\varepsilon)^{\frac{4-\mu}{6-\mu}}}{\lambda_\varepsilon^{\frac{4-\mu}{2}}}\bigg),~~~&\text{if}~N=6-\mu,
		\\
		O\bigg(\frac{1}{\lambda_\varepsilon ^{\frac{N-\mu+2}{2}}}+\frac{\varepsilon}{\lambda_\varepsilon^{\frac{4-\mu}{2}}}\bigg),~~~&\text{if}~N>6-\mu.
	\end{cases}
\end{equation}
\end{Lem}
\begin{proof}
See Lemma~4.1 in \cite{Yang-Zhao-2021}.
\end{proof}

\begin{Lem}\label{inequality}
For any $a>0$, $b>0$, one has
\begin{equation}\label{ele}
\begin{split}
&~~~~~(a+b)^r=a^r+ra^{r-1}b+O\big(b^r\big),~~~\text{if}~~1<r\leq2,\\&
(a+b)^r=a^r+ra^{r-1}b+\frac{r(r-1)}{2}a^{r-2}b^2+O\big(b^r\big),~~~\text{if}~~r>2.
\end{split}
\end{equation}
\end{Lem}
\begin{proof}
This follows from a direct calculation.
\end{proof}

\begin{Lem}\label{Lem53}
For $N\geq4$, $\mu\in(0,4]$, it holds
\begin{equation}\label{R151}
A_{N,\mu}=\frac{N(N-2)}{A_{H,L}}\int_{\mathbb{R}^N}U_{0,1}^{2^\ast-1}dx
+O\Big(\frac{1}{\lambda_{\varepsilon}^2}\Big),
\end{equation}
where $A_{N,\mu}$ and $A_{H,L}$ from \eqref{AHL} and Lemma~\ref{Lem2.2}, respectively.
\end{Lem}
\begin{proof}
We have
\begin{equation}\label{R7}
\begin{split}
A_{N,\mu}&=\int_{ B_{\tau\lambda_{\varepsilon}}(0)}\int_{B_{\tau\lambda_{\varepsilon}}(0)}\frac{v_{\varepsilon}^{2_{\mu}^\ast}(\xi)v_{\varepsilon}^{2_{\mu}^\ast-1}(x)}{|x-\xi|^{\mu}}d\xi dx\\&
=\lambda_{\varepsilon}^{\frac{N-2}{2}}\Big(B_{1}-2B_{2}-B_{3}\Big),
\end{split}
\end{equation}
where
\begin{equation*}
\begin{split}
&B_{1}=\int_{\mathbb{R}^N}\int_{\mathbb{R}^N}\frac{u_{\varepsilon}^{2_{\mu}^\ast}(\xi)u_{\varepsilon}^{2_{\mu}^\ast-1}(z)}{|x-\xi|^{\mu}}d\xi dx,~~~B_{2}=\int_{\mathbb{R}^N}\int_{\mathbb{R}^N\setminus B_{\tau}(x_{\varepsilon})}\frac{u_{\varepsilon}^{2_{\mu}^\ast}(\xi)u_{\varepsilon}^{2_{\mu}^\ast-1}(x)}{|x-\xi|^{\mu}}d\xi dx,\\&
B_{3}=\int_{\mathbb{R}^N\setminus{B_{\tau}(x_{\varepsilon})}}\int_{\mathbb{R}^N\setminus{B_{\tau}(x_{\varepsilon})}}\frac{u_{\varepsilon}^{2_{\mu}^\ast}(\xi)u_{\varepsilon}^{2_{\mu}^\ast-1}(x)}{|x-\xi|^{\mu}}d\xi dx.
\end{split}
\end{equation*}
Combining \eqref{ele} and a direct calculation shows that
\begin{equation}\label{R8}
\begin{split}
&B_{1}
=\frac{N(N-2)}{A_{H,L}}\cdot\frac{1}{\lambda_{\varepsilon}^{\frac{N-2}{2}}}\int_{\mathbb{R}^N}U_{0,1}^{2^\ast-1}dx
+O\Big(\frac{1}{\lambda_{\varepsilon}^\frac{N+2}{2}}\Big),
\end{split}
\end{equation}
where the estimate of \eqref{R8} follows by the following some computations.
First we remark that
\begin{equation}\label{R9}
\begin{split}
\int_{\mathbb{R}^N}\int_{\mathbb{R}^N}&\frac{( PU_{x_{\varepsilon},\lambda_{\varepsilon}}(\xi))^{2_{\mu}^\ast}( PU_{x_{\varepsilon},\lambda_{\varepsilon}}(x))^{2_{\mu}^\ast-1}}{|x-\xi|^{\mu}}d\xi dx
=\frac{N(N-2)}{A_{H,L}}\cdot\frac{1}{\lambda_{\varepsilon}^{\frac{N-2}{2}}}\int_{\mathbb{R}^N}U_{0,1}^{2^\ast-1}dx
+O\Big(\frac{1}{\lambda_{\varepsilon}^\frac{N+2}{2}}\Big),
\end{split}
\end{equation}
since
\begin{equation*}
\begin{split}
\int_{\mathbb{R}^N}\int_{\mathbb{R}^N}\frac{ U_{x_{\varepsilon},\lambda_{\varepsilon}}^{2_{\mu}^\ast}(\xi) U_{x_{\varepsilon},\lambda_{\varepsilon}}^{2_{\mu}^\ast-2}(x)\psi_{x_{\varepsilon},\lambda_{\varepsilon}}}{|x-\xi|^{\mu}}d\xi dx&
=\frac{N(N-2)}{A_{H,L}}\int_{\mathbb{R}^N}U_{x_{\varepsilon},\lambda_{\varepsilon}}^{2^\ast-2}(x)\psi_{x_{\varepsilon},\lambda_{\varepsilon}}dx
=O\Big(\frac{1}{\lambda_{\varepsilon}^\frac{N+2}{2}}\Big),
\end{split}
\end{equation*}
\begin{equation*}
\begin{split}
\int_{\mathbb{R}^N}\int_{\mathbb{R}^N}\frac{ U_{x_{\varepsilon},\lambda_{\varepsilon}}^{2_{\mu}^\ast-1}(\xi)\psi_{x_{\varepsilon},\lambda_{\varepsilon}} U_{x_{\varepsilon},\lambda_{\varepsilon}}^{2_{\mu}^\ast-2}(x)\psi_{x_{\varepsilon},\lambda_{\varepsilon}}}{|x-\xi|^{\mu}}d\xi dx&=
\int_{\mathbb{R}^N}\int_{\mathbb{R}^N}\frac{U_{x_{\varepsilon},\lambda_{\varepsilon}}^{2_{\mu}^\ast-2}(\xi)\psi_{x_{\varepsilon},\lambda_{\varepsilon}}^2 U_{x_{\varepsilon},\lambda_{\varepsilon}}^{2_{\mu}^\ast-1}(x)}{|x-y|^{\mu}}d\xi dx=O\Big(\frac{1}{\lambda_{\varepsilon}^\frac{N+2}{2}}\Big),
\end{split}
\end{equation*}
and
\begin{equation*}
\int_{\mathbb{R}^N}\int_{\mathbb{R}^N}\frac{ U_{x_{\varepsilon},\lambda_{\varepsilon}}^{2_{\mu}^\ast-2}(\xi)\psi_{x_{\varepsilon},\lambda_{\varepsilon}}^2 U_{x_{\varepsilon},\lambda_{\varepsilon}}^{2_{\mu}^\ast-2}(x)\psi_{x_{\varepsilon},\lambda_{\varepsilon}}}{|x-\xi|^{\mu}}d\xi dx
=O\Big(\frac{1}{\lambda_{\varepsilon}^\frac{N+2}{2}}\Big).
\end{equation*}
Second, we have
\begin{equation}\label{R10}
\begin{split}
\int_{\mathbb{R}^N}\int_{\mathbb{R}^N}\frac{(PU_{x_{\varepsilon},\lambda_{\varepsilon}}(\xi))^{2_{\mu}^\ast}(PU_{x_{\varepsilon},\lambda_{\varepsilon}}(x))^{2_{\mu}^\ast-2}w_{\varepsilon}}{|x-\xi|^{\mu}}d\xi dx
=O\Big(\|w_{\varepsilon}\|_{H_{0}^{1}}\Big).
\end{split}
\end{equation}
Moreover, we deduce
\begin{equation}\label{R11}
\begin{split}
\int_{\mathbb{R}^N}\int_{\mathbb{R}^N}\frac{(PU_{x_{\varepsilon},\lambda_{\varepsilon}}(\xi))^{2_{\mu}^\ast-1}w_{\varepsilon}(PU_{x_{\varepsilon},\lambda_{\varepsilon}}(x))^{2_{\mu}^\ast-1}}{|x-\xi|^{\mu}}d\xi dx&
=O\Big(\frac{1}{\lambda_{\varepsilon}^{\frac{N-2}{2}}}\big(\int_{0}^{+\infty}\frac{r^{N-1}}{(1+r^2)^\frac{N(N-\mu+2)}{2N-\mu}}dr\big)^{\frac{2N-\mu}{2N}}\|w_{\varepsilon}\|_{H_{0}^1}\Big)\\&
=O\Big(\frac{1}{\lambda_{\varepsilon}^{N}}\Big),
\end{split}
\end{equation}
and
\begin{equation}\label{R12}
\begin{split}
\int_{\mathbb{R}^N}\int_{\mathbb{R}^N}&\frac{(PU_{x_{\varepsilon},\lambda_{\varepsilon}}(\xi))^{2_{\mu}^\ast-1}w_{\varepsilon}(PU_{x_{\varepsilon},\lambda_{\varepsilon}}(x))^{2_{\mu}^\ast-2}w_{\varepsilon}}{|x-\xi|^{\mu}}d\xi dx
=O\Big(\|w_{\varepsilon}\|_{H_{0}^1}^2\Big).
\end{split}
\end{equation}
Similar to the calculation of \eqref{R10}, \eqref{R11} and \eqref{R12}, we find
\begin{equation*}
\begin{split}
&\int_{\mathbb{R}^N}\int_{\mathbb{R}^N}\Big[\frac{(PU_{x_{\varepsilon},\lambda_{\varepsilon}})^{2_{\mu}^\ast-2}w_{\varepsilon}^2(PU_{x_{\varepsilon},\lambda_{\varepsilon}}(x))^{2_{\mu}^\ast-1}}{|x-\xi|^{\mu}}
+\frac{(PU_{x_{\varepsilon},\lambda_{\varepsilon}})^{2_{\mu}^\ast-2}w_{\varepsilon}^2(PU_{x_{\varepsilon},\lambda_{\varepsilon}}(x))^{2_{\mu}^\ast-2}w_{\varepsilon}}{|x-\xi|^{\mu}}\Big]d\xi dx=O\Big(\frac{1}{\lambda_{\varepsilon}^\frac{N+2}{2}}\Big),\\&
\int_{\mathbb{R}^N}\int_{\mathbb{R}^N}\frac{w_{\varepsilon}^{2_{\mu}^{\ast}}(PU_{x_{\varepsilon},\lambda_{\varepsilon}}(x))^{2_{\mu}^\ast-1}}{|x-\xi|^{\mu}}d\xi dx+\int_{\mathbb{R}^N}\int_{\mathbb{R}^N}\frac{w_{\varepsilon}^{2_{\mu}^{\ast}}(PU_{x_{\varepsilon},\lambda_{\varepsilon}}(x))^{2_{\mu}^\ast-2}w_\varepsilon}{|x-\xi|^{\mu}}d\xi dx=O\Big(\frac{1}{\lambda_{\varepsilon}^\frac{N+2}{2}}\Big),\\&
\int_{\mathbb{R}^N}\int_{\mathbb{R}^N}\frac{(PU_{x_{\varepsilon},\lambda_{\varepsilon}}(\xi))^{2_{\mu}^\ast}w_{\varepsilon}^{2_{\mu}^\ast-1}}{|x-\xi|^{\mu}}d\xi dx+\int_{\mathbb{R}^N}\int_{\mathbb{R}^N}\frac{(PU_{x_{\varepsilon},\lambda_{\varepsilon}}(\xi))^{2_{\mu}^\ast-1}w_{\varepsilon}w_{\varepsilon}^{2_{\mu}^\ast-1}}{|x-\xi|^{\mu}}d\xi dx=O\Big(\frac{1}{\lambda_{\varepsilon}^\frac{N+2}{2}}\Big),\\&
\int_{\mathbb{R}^N}\int_{\mathbb{R}^N}\frac{(PU_{x_{\varepsilon},\lambda_{\varepsilon}}(\xi))^{2_{\mu}^\ast-2}w_{\varepsilon}^2(\xi)w_{\varepsilon}^{2_{\mu}^\ast-1}}{|x-\xi|^{\mu}}d\xi dx+\int_{\mathbb{R}^N}\int_{\mathbb{R}^N}\frac{w_{\varepsilon}^{2_{\mu}^\ast}(\xi)w_{\varepsilon}^{2_{\mu}^\ast-1}}{|x-\xi|^{\mu}}d\xi dx=O\Big(\frac{1}{\lambda_{\varepsilon}^\frac{N+2}{2}}\Big).\\&
\end{split}
\end{equation*}
Combining \eqref{R9} and \eqref{R10}-\eqref{R12}, the estimate \eqref{R8} is reached.

Using $|u_{\varepsilon}(x)|\leq CU_{x_{\varepsilon},\lambda_{\varepsilon}}$, we compute
\begin{equation}\label{R13}
\begin{split}
B_2&\leq C\int_{\mathbb{R}^N}\int_{\mathbb{R}^N\setminus B_{\tau}(x_{\varepsilon})}\frac{U_{x_{\varepsilon},\lambda_{\varepsilon}}^{2_{\mu}^\ast}(\xi)U_{x_{\varepsilon},\lambda_{\varepsilon}}^{2_{\mu}^\ast-1}(x)}{|x-\xi|^{\mu}}d\xi dx=O\Big(\int_{\mathbb{R}^N\setminus B_{\tau}(x_{\varepsilon})}U_{x_{\varepsilon},\lambda_{\varepsilon}}^{2^\ast-1}(x)dx\Big)\\&
=O\Big(\frac{1}{\lambda_{\varepsilon}^{\frac{N+2}{2}}}\int_{\mathbb{R}^N\setminus B_{\tau}(x_{\varepsilon})}\frac{1}{|x-x_{\varepsilon}|^{N+2}}dx\Big)=O\Big(\frac{1}{\lambda_{\varepsilon}^{\frac{N+2}{2}}}\Big).
\end{split}
\end{equation}
And similar to the estimate of \eqref{R13}, we can also obtain
\begin{equation}\label{R14}
B_3=O\Big(\frac{1}{\lambda_{\varepsilon}^{\frac{N+2}{2}}}\Big).
\end{equation}
Then \eqref{R7}, \eqref{R8} and \eqref{R13}-\eqref{R14} imply that \eqref{R151}.
\end{proof}

\begin{Lem}\label{A3}
For any fixed small $\delta>0$, it holds
\begin{equation*}
\frac{1}{(\lambda_{\varepsilon}^{(1)})^{N-\mu+2}}\mathcal{F}_{2}=o\Big(\frac{1}{\lambda_{\varepsilon}^{(1)}}\Big),~~~\frac{1}{(\lambda_{\varepsilon}^{(1)})^{N-\mu+2}}\mathcal{F}_{3}=o\Big(\frac{1}{\lambda_{\varepsilon}^{(1)}}\Big),~~~
\frac{1}{(\lambda_{\varepsilon}^{(1)})^{N-\mu+2}}\mathcal{F}_{4}=o\Big(\frac{1}{\lambda_{\varepsilon}^{(1)}}\Big).
\end{equation*}
\end{Lem}
\begin{proof}
Noting that
\begin{equation*}
(PU_{z,\lambda})^{2_{\mu}^{\ast}-1}=U_{z,\lambda}^{2_{\mu}^{\ast}-1}+O\Big(U_{z,\lambda}^{2_{\mu}^{\ast}-2}\psi_{z,\lambda}\Big). \end{equation*}
Let us write $\mathcal{F}_2=\mathcal{F}_{2,1}+\mathcal{F}_{2,2}$.
Now by Lemma~\ref{L4.2} and \eqref{eqw}, we can calculate that
\begin{equation}\label{AA2}
\begin{split}
&\frac{1}{(\lambda_{\varepsilon}^{(1)})^{N-\mu+2}}\mathcal{F}_{2,1}\\&=\int_{B_{\lambda_{\varepsilon}^{(1)}\delta}(x_{\varepsilon}^{(1)})}\int_{B_{\lambda_{\varepsilon}^{(1)}\delta}(x_{\varepsilon}^{(1)})}\frac{U_{x_{\varepsilon}^{(1)},\lambda_{\varepsilon}^{(1)}}^{2_{\mu}^\ast-1}(\frac{y}{\lambda_{\varepsilon}^{(1)}}+x_{\varepsilon}^{(1)})w_{\varepsilon}^{(1)}(\frac{\xi}{\lambda_{\varepsilon}^{(1)}}+x_{\varepsilon}^{(1)})U_{x_{\varepsilon}^{(1)},\lambda_{\varepsilon}^{(1)}}^{2_{\mu}^{\ast}-2}\big(\frac{x}{\lambda_{\varepsilon}^{(1)}}+x_{\varepsilon}^{(1)}\big)\tilde{\eta}_{\varepsilon}(x)\varphi(x)}{|x-\xi|^{\mu}}dxd\xi\\&
+O\big(\frac{\ln \lambda_{\varepsilon}^{(1)} }{\lambda_{\varepsilon}^{(1)}}\big)\int_{B_{\lambda_{\varepsilon}^{(1)}\delta}(x_{\varepsilon}^{(1)})}\int_{B_{\lambda_{\varepsilon}^{(1)}\delta}(x_{\varepsilon}^{(1)})}\frac{U_{x_{\varepsilon}^{(1)},\lambda_{\varepsilon}^{(1)}}^{2_{\mu}^\ast-1}(\frac{\xi}{\lambda_{\varepsilon}^{(1)}}+x_{\varepsilon}^{(1)})w_{\varepsilon}^{(1)}(\frac{\xi}{\lambda_{\varepsilon}^{(1)}}+x_{\varepsilon}^{(1)})U_{x_{\varepsilon}^{(1)},\lambda_{\varepsilon}^{(1)}}^{2_{\mu}^{\ast}-2}\big(\frac{x}{\lambda_{\varepsilon}^{(1)}}+x_{\varepsilon}^{(1)}\big)\tilde{\eta}_{\varepsilon}(x)\varphi(x)}{|x-\xi|^{\mu}}dxd\xi\\&
+O\Big(\int_{B_{\lambda_{\varepsilon}^{(1)}\delta}(x_{\varepsilon}^{(1)})}\int_{B_{\lambda_{\varepsilon}^{(1)}\delta}(x_{\varepsilon}^{(1)})}\frac{U_{x_{\varepsilon}^{(1)},\lambda_{\varepsilon}^{(1)}}^{2_{\mu}^\ast-1}(\frac{\xi}{\lambda_{\varepsilon}^{(1)}}+x_{\varepsilon}^{(1)})w_{\varepsilon}^{(1)}(\frac{\xi}{\lambda_{\varepsilon}^{(1)}}+x_{\varepsilon}^{(1)})\sum\limits_{j=1}^{2}|w_{\varepsilon}^{(j)}|^{2_{\mu}^{\ast}-2}\tilde{\eta}_{\varepsilon}(x)\varphi(x)}{|x-\xi|^{\mu}}dxd\xi\Big)\\&
=o\Big(\frac{1}{\lambda_{\varepsilon}^{(1)}}\Big).
\end{split}
\end{equation}
Next, similar to the calculations of \eqref{AA2}, by Lemma~\ref{FPU}, we can also get
\begin{equation*}
\frac{1}{(\lambda_{\varepsilon}^{(1)})^{N-\mu+2}}\mathcal{F}_{2,2}=o\Big(\frac{1}{\lambda_{\varepsilon}^{(1)}}\Big).
\end{equation*}
Hence we prove that $\frac{1}{(\lambda_{\varepsilon}^{(1)})^{N-\mu+2}}\mathcal{F}_{2}=o\Big(\frac{1}{\lambda_{\varepsilon}^{(1)}}\Big)$.
Analogously, we have
\begin{equation*}
\frac{1}{(\lambda_{\varepsilon}^{(1)})^{N-\mu+2}}\mathcal{F}_{3}=o\Big(\frac{1}{\lambda_{\varepsilon}^{(1)}}\Big),~~~
\frac{1}{(\lambda_{\varepsilon}^{(1)})^{N-\mu+2}}\mathcal{F}_{4}=o\Big(\frac{1}{\lambda_{\varepsilon}^{(1)}}\Big).
\end{equation*}
This finishes the proof.
\end{proof}

\section{Estimates of $\mathcal{G}_{1}$, $\mathcal{G}_{2}$, $\mathcal{G}_{3}$ and $\mathcal{G}_{4}$ in \eqref{B12}}
\begin{Lem}\label{LC1}
For any $N\geq6$ and $\mu\in(0,4)$, it holds
\begin{equation}\label{C0}
\frac{1}{(\lambda_{\varepsilon}^{(1)})^{2N-\mu+1}}\mathcal{G}_{1}=-\frac{c_l}{2^{\ast}-1}\frac{N(N-2)}{A_{H,L}}\frac{1}{\big(\lambda_{\varepsilon}^{(1)}\big)^{N-1}}\int_{\R^N}U_{0,1}^{\frac{N+2}{N-2}}(z)dz+o\Big(\frac{1}{(\lambda_{\varepsilon}^{(1)})^{N-1}}\Big),~~~~\text{for}~~l=1,2,\cdots,N.
\end{equation}
\end{Lem}
\begin{proof}
In view of $PU_{z,\lambda}=U_{z,\lambda}-\psi_{z,\lambda}$, we know $$(PU_{z,\lambda})^{2_{\mu}^{\ast}}=U_{z,\lambda}^{2_{\mu}^{\ast}}-2_{\mu}^{\ast}U_{z,\lambda}^{2_{\mu}^{\ast}-1}\psi_{z,\lambda}+O\Big(U_{z,\lambda}^{2_{\mu}^{\ast}-2}\psi_{z,\lambda}^2\Big).$$ Then $\mathcal{G}_1$ can be written as follows:
\begin{equation}\label{C1}
\mathcal{G}_1=\mathcal{G}_{1,1}-2_{\mu}^{\ast}\mathcal{G}_{1,2}+O\big(\mathcal{G}_{1,3}\big),
\end{equation}
where
\begin{equation*}
\begin{split}
&\mathcal{G}_{1,1}=\int_{B_{\lambda_{\varepsilon}^{(1)}\tau}(0)}\int_{B_{\lambda_{\varepsilon}^{(1)}\tau}(0)}z_l\frac{U_{x_{\varepsilon}^{(1)},\lambda_{\varepsilon}^{(1)}}^{2_{\mu}^\ast}(\frac{\xi}{\lambda_{\varepsilon}^{(1)}}+x_{\varepsilon}^{(1)})C_{\varepsilon}(\frac{z}{\lambda_{\varepsilon}^{(1)}}+x_{\varepsilon}^{(1)})\tilde{\eta}_{\varepsilon}(z)}{|z-\xi|^{\mu}}dzd\xi,\\&
\mathcal{G}_{1,2}=\int_{B_{\lambda_{\varepsilon}^{(1)}\tau}(0)}\int_{B_{\lambda_{\varepsilon}^{(1)}\tau}(0)}z_l\frac{U_{x_{\varepsilon}^{(1)},\lambda_{\varepsilon}^{(1)}}^{2_{\mu}^\ast-1}(\frac{\xi}{\lambda_{\varepsilon}^{(1)}}+x_{\varepsilon}^{(1)})\psi_{x_{\varepsilon}^{(1)},\lambda_{\varepsilon}^{(1)}}(\frac{\xi}{\lambda_{\varepsilon}^{(1)}}+x_{\varepsilon}^{(1)})C_{\varepsilon}(\frac{z}{\lambda_{\varepsilon}^{(1)}}+x_{\varepsilon}^{(1)})\tilde{\eta}_{\varepsilon}(z)}{|z-\xi|^{\mu}}dzd\xi,\\&
\mathcal{G}_{1,3}=\int_{B_{\lambda_{\varepsilon}^{(1)}\tau}(0)}\int_{B_{\lambda_{\varepsilon}^{(1)}\tau}(0)}z_l\frac{U_{x_{\varepsilon}^{(1)},\lambda_{\varepsilon}^{(1)}}^{2_{\mu}^\ast-2}(\frac{\xi}{\lambda_{\varepsilon}^{(1)}}+x_{\varepsilon}^{(1)})\big(\psi_{x_{\varepsilon}^{(1)},\lambda_{\varepsilon}^{(1)}}(\frac{\xi}{\lambda_{\varepsilon}^{(1)}}+x_{\varepsilon}^{(1)})\big)^2C_{\varepsilon}(\frac{z}{\lambda_{\varepsilon}^{(1)}}+x_{\varepsilon}^{(1)})\tilde{\eta}_{\varepsilon}(z)}{|z-\xi|^{\mu}}dzd\xi.
\end{split}
\end{equation*}
Combining \eqref{e29}, \eqref{6.3}, \eqref{6.91} and oddness of the function, we can prove that as $\varepsilon\rightarrow0$
\begin{equation}\label{C3}
\begin{split}
\frac{1}{(\lambda_{\varepsilon}^{(1)})^{2N-\mu+1}}\mathcal{G}_{1,1}&
\rightarrow\frac{1}{\big(\lambda_{\varepsilon}^{(1)}\big)^{N-1}}\int_{\R^N}\int_{\R^N}y_l\frac{U_{0,1}^{2_{\mu}^\ast}(\xi)U_{0,1}^{\frac{4-\mu}{N-2}}(z)\big(\sum\limits_{k=0}^{N}c_k\phi_k(z)\big)}{|x-y|^{\mu}}dzd\xi\\&
=\frac{N(N-2)}{\mathcal{A}_{H,L}}\frac{1}{\big(\lambda_{\varepsilon}^{(1)}\big)^{N-1}}c_{l}\int_{\R^N}z_lU_{0,1}^{\frac{4}{N-2}}\frac{\partial U_{0,1}(z)}{\partial z_l}dz\\&
=-\frac{c_l}{2^{\ast}-1}\frac{N(N-2)}{A_{H,L}}\frac{1}{\big(\lambda_{\varepsilon}^{(1)}\big)^{N-1}}\int_{\R^N}U_{0,1}^{\frac{N+2}{N-2}}(z)dz,~~~~\text{for}~~l=1,2,\cdots,N.
\end{split}
\end{equation}
Together with \eqref{6.91}, \eqref{eqw}, Lemma~\ref{Lem9.1}, oddness of the function, Hardy-Littlewood-Sobolev inequality, H\"{o}lder inequality and Sobolev embedding theorem, we can prove
\begin{equation}\label{C4}
\begin{split}
\frac{1}{(\lambda_{\varepsilon}^{(1)})^{2N-\mu+1}}\mathcal{G}_{1,2}&\leq\underbrace{\frac{\|\psi_{x_{\varepsilon}^{(1)},\lambda_{\varepsilon}^{(1)}}\|_{L^\infty}}{\big(\lambda_{\varepsilon}^{(1)}\big)^{\frac{3N}{2}-2}}\int_{B_{\lambda_{\varepsilon}^{(1)}\tau}(0)}\int_{B_{\lambda_{\varepsilon}^{(1)}\tau}(0)}z_l\frac{U_{0,1}^{2_{\mu}^\ast-1}(\xi)U_{0,1}^{\frac{4-\mu}{N-2}}(z)\big(\sum\limits_{k=1}^Nc_k\frac{\partial U_{0,1}}{\partial z_k}\big)}{|z-\xi|^{\mu}}dzd\xi}\limits_{=:\mathcal{G}_{1,2,1}}\\&+\underbrace{O\Big(\frac{\|\psi_{x_{\varepsilon}^{(1)},\lambda_{\varepsilon}^{(1)}}\|_{L^\infty}\ln\lambda_{\varepsilon}^{(1)}}{\big(\lambda_{\varepsilon}^{(1)}\big)^{\frac{3N}{2}-1}}\Big)\int_{B_{\lambda_{\varepsilon}^{(1)}\tau}(0)}\int_{B_{\lambda_{\varepsilon}^{(1)}\tau}(0)}z_l\frac{U_{0,1}^{2_{\mu}^\ast-1}(\xi)U_{0,1}^{\frac{4-\mu}{N-2}}(z)\big(\sum\limits_{k=1}^Nc_k\frac{\partial U_{0,1}}{\partial z_k}\big)}{|z-\xi|^{\mu}}dzd\xi}\limits_{=:\mathcal{G}_{1,2,2}}\\&+\underbrace{O\Big(\frac{\|\psi_{x_{\varepsilon}^{(1)},\lambda_{\varepsilon}^{(1)}}\|_{L^\infty}}{\big(\lambda_{\varepsilon}^{(1)}\big)^{\frac{3N}{2}-\frac{\mu}{2}}}\Big)\int_{B_{\lambda_{\varepsilon}^{(1)}\tau}(0)}\int_{B_{\lambda_{\varepsilon}^{(1)}\tau}(0)}z_l\frac{U_{0,1}^{2_{\mu}^\ast-1}(\xi)\big(\sum\limits_{j=1}^{2}|w_{\varepsilon}^{(j)}|^{2_{\mu}^{\ast}-2}\big)\big(\sum\limits_{k=1}^Nc_k\frac{\partial U_{0,1}}{\partial z_k}\big)}{|z-\xi|^{\mu}}dxdy}\limits_{=:\mathcal{G}_{1,2,3}}\\&
=O\Big(\frac{1}{(\lambda_{\varepsilon}^{(1)})^{2N-3}}\Big)+O\Big(\frac{\ln\lambda_{\varepsilon}^{(1)}}{(\lambda_{\varepsilon}^{(1)})^{2N-2}}\Big)+O\Big(\frac{(\ln\lambda_{\varepsilon}^{(1)})^{\frac{N-2}{N}}}{(\lambda_{\varepsilon}^{(1)})^{2N-\frac{\mu}{2}-1+\frac{(N-\mu+2)(4-\mu)}{2(N-2)}}}\Big),\\&
\end{split}
\end{equation}
because
\begin{equation*}
\begin{split}
\mathcal{G}_{1,2,1}&=O\Big(\frac{1}{\big(\lambda_{\varepsilon}^{(1)}\big)^{2N-3}}\Big)\Big(\int_{0}^{\lambda_{\varepsilon}^{(1)}\tau}\frac{r^{N-1}}{\big(1+r\big)^{\frac{2N(N-\mu+2)}{2N-\mu}}}dr\Big)^{\frac{2N-\mu}{2N}}\Big(\int_{0}^{\lambda_{\varepsilon}^{(1)}\tau}\frac{r^{\frac{4N}{2N-\mu}}\cdot r^{N-1}}{\big(1+r\big)^{\frac{2N(N-\mu+4)}{2N-\mu}}}dr\Big)^{\frac{2N-\mu}{2N}}\\&
=O\Big(\frac{1}{(\lambda_{\varepsilon}^{(1)})^{2N-3}}\Big),
\end{split}
\end{equation*}
where we have used $\frac{2N(N-\mu+2)}{2N-\mu}>N$ and $\frac{2N(N-\mu+4)}{2N-\mu}>N+\frac{4N}{2N-\mu}$.
\begin{equation*}
\begin{split}
\mathcal{G}_{1,2,2}&=O\Big(\frac{\ln\lambda_{\varepsilon}^{(1)}}{\big(\lambda_{\varepsilon}^{(1)}\big)^{2N-2}}\Big)\Big(\int_{0}^{\lambda_{\varepsilon}^{(1)}\tau}\frac{r^{N-1}}{\big(1+r\big)^{\frac{2N(N-\mu+2)}{2N-\mu}}}dr\Big)^{\frac{2N-\mu}{2N}}\Big(\int_{0}^{\lambda_{\varepsilon}^{(1)}\tau}\frac{r^{\frac{4N}{2N-\mu}}\cdot r^{N-1}}{\big(1+r\big)^{\frac{2N(N-\mu+4)}{2N-\mu}}}dr\Big)^{\frac{2N-\mu}{2N}}\\&
=O\Big(\frac{\ln\lambda_{\varepsilon}^{(1)}}{(\lambda_{\varepsilon}^{(1)})^{2N-2}}\Big),
\end{split}
\end{equation*}
and
\begin{equation*}
\begin{split}
\mathcal{G}_{1,2,3}&=O\Big(\frac{\sum_{j=1}^2\big\|w_{\varepsilon}^{(j)}\big\|_{H_{0}^1}^{\frac{4-\mu}{N-2}}}{\big(\lambda_{\varepsilon}^{(1)}\big)^{2N-\frac{\mu}{2}-1}}\Big)\Big(\int_{B_{\lambda_{\varepsilon}^{(1)}\tau}(0)}\frac{1}{\big(1+|z|\big)^{\frac{2N(N-\mu+2)}{2N-\mu}}}dz\Big)^{\frac{2N-\mu}{2N}}\Big(\int_{B_{\lambda_{\varepsilon}^{(1)}\tau}(0)}\frac{|z|^\frac{2N}{N-2}}{\big(1+|z|\big)^{\frac{N^2}{N-2}}}dz\Big)^{\frac{N-2}{N}}\\&
=O\Big(\frac{(\ln\lambda_{\varepsilon}^{(1)})^{\frac{N-2}{N}}}{(\lambda_{\varepsilon}^{(1)})^{2N-\frac{\mu}{2}-1+\frac{(N-\mu+2)(4-\mu)}{2(N-2)}}}\Big).
\end{split}
\end{equation*}
And analogously, from $0\leq\psi_{x_{\varepsilon}^{(1)},\lambda_{\varepsilon}^{(1)}}\leq U_{x_{\varepsilon}^{(1)},\lambda_{\varepsilon}^{(1)}}$, we have
\begin{equation}\label{C5}
\begin{split}
\frac{1}{(\lambda_{\varepsilon}^{(1)})^{2N-\mu+1}}\mathcal{G}_{1,3}
=o\Big(\frac{1}{(\lambda_{\varepsilon}^{(1)})^{N-1}}\Big),
\end{split}
\end{equation}
since $\frac{2N(N-\mu+2)}{2N-\mu}>N$ and $\frac{2N(N-\mu+4)}{2N-\mu}>N+\frac{4N}{2N-\mu}$.
Then \eqref{C1}, \eqref{C3}, \eqref{C4} and \eqref{C5} imply \eqref{C0}.

\end{proof}

\begin{Lem}\label{LC2}
For any $N\geq6$ and $\mu\in(0,4)$, it holds
\begin{equation}\label{CC5}
\frac{1}{(\lambda_{\varepsilon}^{(1)})^{2N-\mu+1}}\mathcal{G}_{2}=o\Big(\frac{1}{(\lambda_{\varepsilon}^{(1)})^{N-1}}\Big).
\end{equation}
\end{Lem}
\begin{proof}
First, $\mathcal{G}_2$ can be written as follows:
\begin{equation}\label{C6}
\mathcal{G}_2=\mathcal{G}_{2,1}+O\big(\mathcal{G}_{2,2}\big),
\end{equation}
where
\begin{equation*}
\begin{split}
&\mathcal{G}_{2,1}=\int_{B_{\lambda_{\varepsilon}^{(1)}\tau}(0)}\int_{B_{\lambda_{\varepsilon}^{(1)}\tau}(0)}z_l\frac{U_{x_{\varepsilon}^{(1)},\lambda_{\varepsilon}^{(1)}}^{2_{\mu}^\ast-1}(\frac{\xi}{\lambda_{\varepsilon}^{(1)}}+x_{\varepsilon}^{(1)})w_{\varepsilon}^{(1)}(\frac{\xi}{\lambda_{\varepsilon}^{(1)}}+x_{\varepsilon}^{(1)})C_{\varepsilon}(\frac{z}{\lambda_{\varepsilon}^{(1)}}+x_{\varepsilon}^{(1)})\tilde{\eta}_{\varepsilon}(z)}{|z-\xi|^{\mu}}dzd\xi,\\&
\mathcal{G}_{2,2}=\int_{B_{\lambda_{\varepsilon}^{(1)}\tau}(0)}\int_{B_{\lambda_{\varepsilon}^{(1)}\tau}(0)}z_l\frac{U_{x_{\varepsilon}^{(1)},\lambda_{\varepsilon}^{(1)}}^{2_{\mu}^\ast-2}(\frac{\xi}{\lambda_{\varepsilon}^{(1)}}+x_{\varepsilon}^{(1)})\psi_{x_{\varepsilon}^{(1)},\lambda_{\varepsilon}^{(1)}}(\frac{\xi}{\lambda_{\varepsilon}^{(1)}}+x_{\varepsilon}^{(1)})w_{\varepsilon}^{(1)}(\frac{\xi}{\lambda_{\varepsilon}^{(1)}}+x_{\varepsilon}^{(1)})C_{\varepsilon}(\frac{z}{\lambda_{\varepsilon}^{(1)}}+x_{\varepsilon}^{(1)})\tilde{\eta}_{\varepsilon}(z)}{|z-\xi|^{\mu}}dzd\xi.
\end{split}
\end{equation*}
Now by \eqref{6.91}, and Lemma~\ref{Lem9.1}, we have
\begin{equation}\label{C7}
\begin{split}
\frac{1}{(\lambda_{\varepsilon}^{(1)})^{2N-\mu+1}}\mathcal{G}_{2,1}&=\underbrace{\frac{1}{\big(\lambda_{\varepsilon}^{(1)}\big)^{\frac{3N}{2}-2}}\int_{B_{\lambda_{\varepsilon}^{(1)}\tau}(0)}\int_{B_{\lambda_{\varepsilon}^{(1)}\tau}(0)}z_l\frac{U_{0,1}^{2_{\mu}^\ast-1}(\xi)w_{\varepsilon}^{(1)}U_{0,1}^{\frac{4-\mu}{N-2}}(z)\big(\sum\limits_{k=1}^Nc_k\frac{\partial U_{0,1}}{\partial z_k}\big)}{|z-\xi|^{\mu}}dzd\xi}\limits_{=:\mathcal{G}_{2,1,1}}\\&
+\underbrace{O\Big(\frac{\ln\lambda_{\varepsilon}^{(1)}}{\big(\lambda_{\varepsilon}^{(1)}\big)^{\frac{3N}{2}-1}}\Big)\int_{B_{\lambda_{\varepsilon}^{(1)}\tau}(0)}\int_{B_{\lambda_{\varepsilon}^{(1)}\tau}(0)}y_l\frac{U_{0,1}^{2_{\mu}^\ast-1}(\xi)w_{\varepsilon}^{(1)}U_{0,1}^{\frac{4-\mu}{N-2}}(z)\big(\sum\limits_{k=1}^Nc_k\frac{\partial U_{0,1}}{\partial z_k}\big)}{|z-\xi|^{\mu}}dzd\xi}\limits_{=:\mathcal{G}_{2,1,2}}\\&
+\underbrace{O\Big(\frac{1}{\big(\lambda_{\varepsilon}^{(1)}\big)^{\frac{3N}{2}-\frac{\mu}{2}}}\Big)\int_{B_{\lambda_{\varepsilon}^{(1)}\tau}(0)}\int_{B_{\lambda_{\varepsilon}^{(1)}\tau}(0)}y_l\frac{U_{0,1}^{2_{\mu}^\ast-1}(\xi)w_{\varepsilon}^{(1)}\big(\sum\limits_{j=1}^{2}|w_{\varepsilon}^{(j)}|^{2_{\mu}^{\ast}-2}\big)\big(\sum\limits_{k=1}^Nc_k\frac{\partial U_{0,1}}{\partial z_k}\big)}{|z-\xi|^{\mu}}dzd\xi}\limits_{=:\mathcal{G}_{2,1,3}}\\&
=o\Big(\frac{1}{(\lambda_{\varepsilon}^{(1)})^{N-1}}\Big),
\end{split}
\end{equation}
because
\begin{equation*}
\begin{split}
\mathcal{G}_{2,1,1}&=O\Big(\frac{\big\|w_{\varepsilon}^{(1)}\big\|_{H_{0}^1}}{\big(\lambda_{\varepsilon}^{(1)}\big)^{\frac{3N}{2}-2}}\Big)\Big(\int_{B_{\lambda_{\varepsilon}^{(1)}\tau}(0)}U_{0,1}^{2^{\ast}}(\xi)d\xi\Big)^{\frac{N-\mu+2}{2N}}\Big(\int_{B_{\lambda_{\varepsilon}^{(1)}\tau}(0)}\frac{|z|^{\frac{4N}{2N-\mu}}}{\big(1+|z|\big)^{\frac{2N(N-\mu+4)}{2N-\mu}}}dz\Big)^{\frac{2N-\mu}{2N}}\\&
=o\Big(\frac{1}{(\lambda_{\varepsilon}^{(1)})^{N-1}}\Big),
\end{split}
\end{equation*}
\begin{equation*}
\mathcal{G}_{2,1,2}=o\Big(\frac{1}{(\lambda_{\varepsilon}^{(1)})^{N-1}}\Big),
\end{equation*}
and
\begin{equation*}
\begin{split}
\mathcal{G}_{2,1,3}&=O\Big(\frac{\big\|w_{\varepsilon}^{(1)}\big\|_{H_{0}^1}}{\big(\lambda_{\varepsilon}^{(1)}\big)^{\frac{3N}{2}-\frac{\mu}{2}}}\Big)\Big(\int_{B_{\lambda_{\varepsilon}^{(1)}\tau}(0)}U_{0,1}^{2^{\ast}}(\xi)d\xi\Big)^{\frac{N-\mu+2}{2N}}\Big(\int_{B_{\lambda_{\varepsilon}^{(1)}\tau}(0)}\big(\frac{\partial U_{0,1}(z)}{\partial z_l}|z|\big)^{\frac{N}{N-2}}dy\Big)^{\frac{N-2}{N}}\Big(\sum_{j=1}^2\big\|w_{\varepsilon}^{(j)}\big\|_{H_{0}^1}^{\frac{4-\mu}{N-2}}\Big)\\&
=o\Big(\frac{1}{(\lambda_{\varepsilon}^{(1)})^{N-1}}\Big).
\end{split}
\end{equation*}
Next similar to the calculations of $\mathcal{A}_{2,1}$, by $0\leq\psi_{x_{\varepsilon}^{(1)},\lambda_{\varepsilon}^{(1)}}\leq U_{x_{\varepsilon}^{(1)},\lambda_{\varepsilon}^{(1)}}$, we know
\begin{equation}\label{C8}
\frac{1}{(\lambda_{\varepsilon}^{(1)})^{2N-\mu+1}}\mathcal{G}_{2,2}=o\Big(\frac{1}{(\lambda_{\varepsilon}^{(1)})^{N-1}}\Big).
\end{equation}
Then \eqref{C6}, \eqref{C7} and \eqref{C8} imply \eqref{CC5}.
\end{proof}
Similar to the proof of Lemmas~\ref{LC1} and \ref{LC2}, we can find following two estimates.
\begin{Lem}\label{LC3}
For any $N\geq6$ and $\mu\in(0,4)$, it holds
\begin{equation*}
\frac{1}{(\lambda_{\varepsilon}^{(1)})^{2N-\mu+1}}\mathcal{G}_{3}=o\Big(\frac{1}{(\lambda_{\varepsilon}^{(1)})^{N-1}}\Big),~~~\frac{1}{(\lambda_{\varepsilon}^{(1)})^{2N-\mu+1}}\mathcal{G}_{4}=o\Big(\frac{1}{(\lambda_{\varepsilon}^{(1)})^{N-1}}\Big).
\end{equation*}
\end{Lem}

\section{Estimates of $\mathcal{H}_{1}$, $\mathcal{H}_{2}$ and $\mathcal{H}_{3}$ in \eqref{A12}}
\begin{Lem}\label{LD1}
For any $N\geq6$ and $\mu\in(0,4)$, it holds
\begin{equation}\label{D0}
\mathcal{H}_{1}=o\Big(\frac{1}{(\lambda_{\varepsilon}^{(1)})^{N-1}}\Big),~~~\mathcal{H}_{2}=o\Big(\frac{1}{(\lambda_{\varepsilon}^{(1)})^{N-1}}\Big),~~~\mathcal{H}_{3}=o\Big(\frac{1}{(\lambda_{\varepsilon}^{(1)})^{N-1}}\Big).
\end{equation}
\end{Lem}
\begin{proof}
First, let us  write $\mathcal{B}_1$
$
\mathcal{H}_1=\mathcal{H}_{1,1}+O\big(\mathcal{H}_{1,2}\big),
$
where
\begin{equation*}
\begin{split}
&\mathcal{H}_{1,1}=\int_{B_{\tau}(x_{\varepsilon}^{(1)})}\int_{B_{\tau}(x_{\varepsilon}^{(1)})}(z_l-x_{\varepsilon,l}^{(1)})\frac{D_{\varepsilon}(\xi)\eta_{\varepsilon}(\xi)U_{x_{\varepsilon}^{(2)},\lambda_{\varepsilon}^{(2)}}^{2_{\mu}^\ast-1}(z)}{|z-\xi|^{\mu}}dzd\xi,\\&
\mathcal{H}_{1,2}=\int_{B_{\tau}(x_{\varepsilon}^{(1)})}\int_{B_{\tau}(x_{\varepsilon}^{(1)})}(z_l-x_{\varepsilon,l}^{(1)})\frac{D_{\varepsilon}(\xi)\eta_{\varepsilon}(\xi)U_{x_{\varepsilon}^{(2)},\lambda_{\varepsilon}^{(2)}}^{2_{\mu}^\ast-2}(z)\psi_{x_{\varepsilon}^{(2)},\lambda_{\varepsilon}^{(2)}}}{|z-\xi|^{\mu}}dzd\xi.
\end{split}
\end{equation*}
Now, let us write
\begin{equation}\label{D8}
\begin{split}
\mathcal{H}_{1,1}&=\underbrace{O\Big(\frac{1}{(\lambda_{\varepsilon}^{(1)})^{2N-\mu+1}}\Big)\int_{B_{\lambda_{\varepsilon}^{(1)}\tau}(0)}\int_{B_{\lambda_{\varepsilon}^{(1)}\tau}(0)}z_l\frac{U_{x_{\varepsilon}^{(1)},\lambda_{\varepsilon}^{(1)}}^{2_{\mu}^{\ast}-1}(\frac{x}{\lambda_{\varepsilon}^{(1)}}+x_{\varepsilon}^{(1)})\tilde{\eta}_{\varepsilon}(x)U_{x_{\varepsilon}^{(1)},\lambda_{\varepsilon}^{(1)}}^{2_{\mu}^\ast-1}(\frac{z}{\lambda_{\varepsilon}^{(1)}}+x_{\varepsilon}^{(1)})}{|z-\xi|^{\mu}}dzd\xi}\limits_{=:\mathcal{H}_{1,1,1}}\\&
+\underbrace{O\Big(\frac{\ln\lambda_{\varepsilon}^{(1)}}{(\lambda_{\varepsilon}^{(1)})^{2N-\mu+2}}\Big)\int_{B_{\lambda_{\varepsilon}^{(1)}\tau}(0)}\int_{B_{\lambda_{\varepsilon}^{(1)}\tau}(0)}z_l\frac{U_{x_{\varepsilon}^{(1)},\lambda_{\varepsilon}^{(1)}}^{2_{\mu}^{\ast}-1}(\frac{\xi}{\lambda_{\varepsilon}^{(1)}}+x_{\varepsilon}^{(1)})\tilde{\eta}_{\varepsilon}(\xi)U_{x_{\varepsilon}^{(1)},\lambda_{\varepsilon}^{(1)}}^{2_{\mu}^\ast-1}(\frac{z}{\lambda_{\varepsilon}^{(1)}}+x_{\varepsilon}^{(1)})}{|z-\xi|^{\mu}}dzd\xi}\limits_{:=\mathcal{H}_{1,1,2}}\\&
+\underbrace{O\Big(\frac{1}{(\lambda_{\varepsilon}^{(1)})^{2N-\mu+1}}\Big)\int_{B_{\lambda_{\varepsilon}^{(1)}\tau}(0)}\int_{B_{\lambda_{\varepsilon}^{(1)}\tau}(0)}z_l\frac{\big(\sum\limits_{j=1}^{2}|w_{\varepsilon}^{(j)}|^{2_{\mu}^{\ast}-1}\big)\tilde{\eta}_{\varepsilon}U_{x_{\varepsilon}^{(1)},\lambda_{\varepsilon}^{(1)}}^{2_{\mu}^\ast-1}(\frac{z}{\lambda_{\varepsilon}^{(1)}}+x_{\varepsilon}^{(1)})}{|z-\xi|^{\mu}}dzd\xi}\limits_{=:\mathcal{H}_{1,1,3}}
\\&=o\Big(\frac{1}{(\lambda_{\varepsilon}^{(1)})^{N-1}}\Big).
\end{split}
\end{equation}
Because of oddness, we have
\begin{equation*}
\int_{B_{\lambda_{\varepsilon}^{(1)}\tau}(0)}\frac{z_lU_{x_{\varepsilon}^{(1)},\lambda_{\varepsilon}^{(1)}}^{2_{\mu}^\ast-1}(\frac{z}{\lambda_{\varepsilon}^{(1)}}+x_{\varepsilon}^{(1)})}{|z-\xi|^{\mu}}d\xi=0,
\end{equation*}
which imply
\begin{equation}\label{D4}
\mathcal{H}_{1,1,1}=\mathcal{H}_{1,1,2}=0.
\end{equation}
On the other hand, we divide our argument into three cases:\\
$\mathbf{(1).}$ For $0<\mu<2$, and $\frac{2N(N-\mu+2)}{2N-\mu}>\frac{2N}{2N-\mu}+N$,
\begin{equation*}
\Big(\int_{B_{\lambda_{\varepsilon}^{(1)}\tau}(0)}\big||z|U_{0,1}^{2_{\mu}^{\ast}-1}(z)\big|^{\frac{2N}{2N-\mu}}dz\Big)^{\frac{2N-\mu}{2N}}\leq C.
\end{equation*}
Using Hardy-Littlewood Sobolev, H\"{o}lder inequality and \eqref{eqw}, then we have
\begin{equation}\label{D4}
\begin{split}
\mathcal{H}_{1,1,3}&=O\Big(\frac{\sum\limits_{j=1}^2\big\|w_{\varepsilon}^{(j)}\big\|_{H_{0}^1}^{\frac{N-\mu+2}{N-2}}}{(\lambda_{\varepsilon}^{(1)})^{\frac{3N-\mu+2}{2}}}\Big(\int_{B_{\lambda_{\varepsilon}^{(1)}\tau}(0)}\big||z|U_{0,1}^{2_{\mu}^{\ast}-1}(z)\big|^{\frac{2N}{2N-\mu}}dz\Big)^{\frac{2N-\mu}{2N}}\Big)
\\&
=O\Big(\frac{1}{(\lambda_{\varepsilon}^{(1)})^{\frac{3N-\mu+2}{2}+\frac{(N-\mu+2)^2}{2(N-2)}}}\Big).
\end{split}
\end{equation}
$\mathbf{(2).}$ For $\mu=2$, and $\frac{2N(N-\mu+2)}{2N-\mu}=\frac{2N}{2N-\mu}+N$,
\begin{equation*}
\Big(\int_{B_{\lambda_{\varepsilon}^{(1)}\tau}(0)}\big||z|U_{0,1}^{2_{\mu}^{\ast}-1}(z)\big|^{\frac{2N}{2N-\mu}}dz\Big)^{\frac{2N-\mu}{2N}}=O\Big(\ln \lambda_{\varepsilon}^{(1)}\Big).
\end{equation*}
Using the definition of $\mathcal{H}_{1,1,3}$ and \eqref{eqw}, we can also obtain
\begin{equation}\label{D5}
\mathcal{H}_{1,1,3}=o\Big(\frac{1}{(\lambda_{\varepsilon}^{(1)})^{N-1}}\Big).
\end{equation}
$\mathbf{(3).}$ For $2<\mu<4$, and $\frac{2N(N-\mu+2)}{2N-\mu}<\frac{2N}{2N-\mu}+N$,
\begin{equation*}
\Big(\int_{B_{\lambda_{\varepsilon}^{(1)}\tau}(0)}\big||z|U_{0,1}^{2_{\mu}^{\ast}-1}(z)\big|^{\frac{2N}{2N-\mu}}dz\Big)^{\frac{2N-\mu}{2N}}=O\Big((\lambda_{\varepsilon}^{(1)})^{\frac{\mu-2}{2}}\Big).
\end{equation*}
Using the definition of $\mathcal{H}_{1,1,3}$ and \eqref{eqw}, we can also get
\begin{equation}\label{D6}
\mathcal{H}_{1,1,3}=o\Big(\frac{1}{(\lambda_{\varepsilon}^{(1)})^{N-1}}\Big).
\end{equation}
Thus, from \eqref{D4}, \eqref{D5} and \eqref{D6} imply $\mathcal{H}_{1,1}=o\big(\frac{1}{(\lambda_{\varepsilon}^{(1)})^{N-1}}\big).$ Similarly, we can also prove
\begin{equation*}
\mathcal{H}_{1,2}=o\Big(\frac{1}{(\lambda_{\varepsilon}^{(1)})^{N-1}}\Big).
\end{equation*}
Similar to the above argument of $\mathcal{H}_{1}$, we can also get
\begin{equation}\label{D00}
\mathcal{H}_{2}=o\Big(\frac{1}{(\lambda_{\varepsilon}^{(1)})^{N-1}}\Big),~~~\mathcal{H}_{3}=o\Big(\frac{1}{(\lambda_{\varepsilon}^{(1)})^{N-1}}\Big).
\end{equation}
Then the conclusion follows by the above estimates.
\end{proof}

\section{Estimates of $A_{\varepsilon}^{(1)}$ and  $A_{\varepsilon}^{(2)}$  in \eqref{423}-\eqref{4233} and $\text{RHS of}~\eqref{ph2}$ when $\Omega^{\prime}=B_{\tau}(x_{\varepsilon}^{(1)})$}
\begin{Lem}\label{LE1}
For any $N\geq6$ and $\mu\in(0,4)$, it holds
\begin{equation}\label{LE11}
A_{\varepsilon}^{(1)}+A_{\varepsilon}^{(2)}=\frac{1}{\big(\lambda_{\varepsilon}^{(1)}\big)^{N-2}}\frac{N(N-2)}{A_{H,L}}\int_{\R^N}U_{0,1}^{\frac{4}{N-2}}(z)c_0\phi_0dz+o\Big(\frac{1}{(\lambda_{\varepsilon}^{(1)})^{N-2}}\Big).
\end{equation}
\end{Lem}
\begin{proof}
The proof is similar to that of Lemmas \ref{LC1},~\ref{LC2} and \ref{LC3}.
Then we can estimate \eqref{LE11} by \eqref{e29}, \eqref{6.91}, \eqref{6.102}, \eqref{u12},
\eqref{e6.3}, \eqref{6.3}, \eqref{423}, \eqref{4233} and \eqref{AA3}.
\end{proof}

\begin{Lem}\label{LF1}
For any $N\geq6$ and $\mu\in(0,4)$, it holds
\begin{equation*}
\text{RHS of}~\eqref{ph2}=\frac{2\varepsilon}{\big(\lambda_{\varepsilon}^{(1)}\big)^{\frac{N+2}{2}}}\Big(\int_{\R^N}U_{0,1}^2(z)dz\Big)c_0+o\Big(\frac{1}{(\lambda_{\varepsilon}^{(1)})^{\frac{N-2}{2}}}\Big).
\end{equation*}
\end{Lem}
\begin{proof}
Taking $\Omega^{\prime}=B_{\tau}(x_{\varepsilon}^{(1)})$ in \eqref{ph2},
$\text{RHS of}~\eqref{ph2}$ can be written as follows:
\begin{equation*}
\text{RHS of}~\eqref{ph2}=\mathcal{J}_1+\mathcal{J}_2+\mathcal{J}_3+\mathcal{J}_4+\mathcal{J}_5+\mathcal{J}_6+\mathcal{E}_7+\mathcal{J}_8+\mathcal{J}_9,
\end{equation*}
where
 \begin{equation*}
 \begin{split}
&\mathcal{J}_1+\mathcal{J}_2:=-\frac{\mu}{2}\int_{ B_{\tau}(x_{\varepsilon}^{(1)})}\int_{\Omega\setminus B_{\tau}(x_{\varepsilon}^{(1)})}\Big[\frac{|u_{\varepsilon}^{(2)}(\xi)|^{2_{\mu}^\ast}\tilde{C}_{\varepsilon}(x)\eta_{\varepsilon}(x)}
{|x-\xi|^{\mu}}+\frac{D_{\varepsilon}(\xi)\eta_{\varepsilon}(\xi)|u_{\varepsilon}^{(1)}(x)|^{2_{\mu}^\ast}}
{|x-\xi|^{\mu}}\Big]dxd\xi,\\&
\mathcal{J}_3+\mathcal{J}_4:=\mu\int_{ B_{\tau}(x_{\varepsilon}^{(1)})}\int_{\Omega\setminus B_{\tau}(x_{\varepsilon}^{(1)})}\Big[x\cdot(x-\xi)\frac{|u_{\varepsilon}^{(2)}(\xi)|^{2_{\mu}^\ast}\tilde{C}_{\varepsilon}(x)\eta_{\varepsilon}(x)}
{|x-\xi|^{\mu+2}}+x\cdot(x-\xi)\frac{D_{\varepsilon}(\xi)\eta_{\varepsilon}(\xi)|u_{\varepsilon}^{(1)}(x)|^{2_{\mu}^\ast}}
{|x-\xi|^{\mu+2}}\Big]dxd\xi,
\\&
\mathcal{J}_5+\mathcal{J}_6:=\int_{\partial B_{\tau}(x_{\varepsilon}^{(1)})}\int_{ \Omega\setminus B_{\tau}(x_{\varepsilon}^{(1)})}\Big[\frac{|u_{\varepsilon}^{(2)}(\xi)|^{2_{\mu}^\ast}\tilde{C}_{\varepsilon}(x)\eta_{\varepsilon}(x)}
{|x-\xi|^{\mu}}\big\langle x-x_{\varepsilon}^{(1)},\nu\big\rangle+
\frac{D_{\varepsilon}(\xi)\eta_{\varepsilon}(\xi)|u_{\varepsilon}^{(1)}(x)|^{2_{\mu}^\ast}}
{|x-\xi|^{\mu}}\big\langle x-x_{\varepsilon}^{(1)},\nu\big\rangle\Big]d\xi ds,
\\&
\mathcal{J}_7+\mathcal{J}_8:=2\int_{\partial B_{\tau}(x_{\varepsilon}^{(1)})}\int_{ B_{\tau}(x_{\varepsilon}^{(1)})}\frac{|u_{\varepsilon}^{(2)}(\xi)|^{2_{\mu}^\ast}\tilde{C}_{\varepsilon}(x)\eta_{\varepsilon}(x)}
{|x-\xi|^{\mu}}\big\langle x-x_{\varepsilon}^{(1)},\nu\big\rangle+\frac{D_{\varepsilon}(\xi)\eta_{\varepsilon}(\xi)|u_{\varepsilon}^{(1)}(x)|^{2_{\mu}^\ast}}
{|x-\xi|^{\mu}}\big\langle x-x_{\varepsilon}^{(1)},\nu\big\rangle\Big]d\xi ds,\\&
\mathcal{J}_9:=\frac{\varepsilon}{2}\int_{\partial B_{\tau}(x_{\varepsilon}^{(1)})} \big(u_{\varepsilon}^{(1)}+u_{\varepsilon}^{(2)}\big)\eta_{\varepsilon}\big\langle x-\xi_{\varepsilon}^{(1)},\nu\big\rangle ds-\varepsilon\int_{B_{\tau}(x_{\varepsilon}^{(1)})} \big(u_{\varepsilon}^{(1)}+u_{\varepsilon}^{(2)}\big)\eta_{\varepsilon}(x)dx.
\end{split}
\end{equation*}
Using Lemmas~\ref{e6.3}, \ref{Lem6.3}, \ref{Lem6.4}, \eqref{u12} and Hardy-Littlewood-Sobolev inequality, we can calculate that
\begin{equation*}
 \begin{split}
\mathcal{J}_1&=O\Big((\lambda_{\varepsilon}^{(1)})^\frac{3N-2\mu+2}{2}\Big)\int_{ B_{\tau}(x_{\varepsilon}^{(1)})}\int_{\Omega\setminus B_{\tau}(x_{\varepsilon}^{(1)})}\frac{1}{\big(1+\lambda_{\varepsilon}^{(1)}|\xi-x_{\varepsilon}^{(1)}|\big)^{2N-\mu}}\frac{1}{|x-\xi|^\mu}\frac{\eta_{\varepsilon}(x)}{\big(1+\lambda_{\varepsilon}^{(1)}|x-x_{\varepsilon}^{(1)}|\big)^{N-\mu+2}}dxd\xi\\&
=O\Big(\frac{\ln \lambda_{\varepsilon}^{(1)}}{(\lambda_{\varepsilon}^{(1)})^{\frac{3N-6}{2}}}\Big)\int_{ B_{\lambda_{\varepsilon}^{(1)}\tau}(0)}\int_{\Omega\setminus B_{\lambda_{\varepsilon}^{(1)}\tau}(0)}\frac{1}{\big(1+|\xi|\big)^{2N-\mu}}\frac{1}{|x-\xi|^\mu}\frac{1}{\big(1+|x|\big)^{N-\mu+2}}dxd\xi\\&
=O\Big(\frac{\ln\lambda_{\varepsilon}^{(1)}}{(\lambda_{\varepsilon}^{(1)})^\frac{3N-6}{2}}\Big),
\end{split}
\end{equation*}
and analogously
\begin{equation*}
\mathcal{J}_2=\mathcal{J}_3=\mathcal{J}_4=O\Big(\frac{\ln\lambda_{\varepsilon}^{(1)}}{(\lambda_{\varepsilon}^{(1)})^\frac{3N-6}{2}}\Big).
\end{equation*}
Using Hardy-Littlewood-Sobolev inequality, we obtain
\begin{equation*}
 \begin{split}
\mathcal{J}_5&=O\Big((\lambda_{\varepsilon}^{(1)})^\frac{3N-2\mu+2}{2}\Big)\int_{\partial B_{\tau}(x_{\varepsilon}^{(1)})}\int_{\Omega\setminus B_{\tau}(x_{\varepsilon}^{(1)})}\frac{1}{\big(1+\lambda_{\varepsilon}^{(1)}|\xi-x_{\varepsilon}^{(1)}|\big)^{2N-\mu}}\frac{1}{|x-y|^\mu}\frac{\eta_{\varepsilon}(x)\langle x-x_{\varepsilon}^{(1)}), \nu\rangle }{\big(1+\lambda_{\varepsilon}^{(1)}|x-x_{\varepsilon}^{(1)}|\big)^{N-\mu+2}}d\xi ds\\&
=O\Big(\frac{\ln\lambda_{\varepsilon}^{(1)}}{(\lambda_{\varepsilon}^{(1)})^\frac{3N-4}{2}}\Big)\int_{\partial B_{\lambda_{\varepsilon}^{(1)}\tau}(0)}\int_{\Omega\setminus B_{\lambda_{\varepsilon}^{(1)}\tau}(0)}\frac{1}{\big(1+|\xi|\big)^{2N-\mu}}\frac{1}{|x-\xi|^\mu}\frac{1}{\big(1+|x|\big)^{N-\mu+2}}\langle x,\nu\rangle d\xi ds\\&
=O\Big(\frac{\ln\lambda_{\varepsilon}^{(1)}}{(\lambda_{\varepsilon}^{(1)})^\frac{3N-4}{2}}\Big).
\end{split}
\end{equation*}
Similarly, we can also obtain
\begin{equation*}
\mathcal{J}_6=\mathcal{E}_7=\mathcal{E}_8=O\Big(\frac{\ln\lambda_{\varepsilon}^{((1))}}{(\lambda_{\varepsilon}^{(1)})^\frac{3N-4}{2}}\Big).
\end{equation*}
Moreover, we know
\begin{equation*}
\int_{\R^N}U_{0,1}(z)\phi_0(z)dz=-\int_{\R^N}U_{0,1}^2(z)dz.
\end{equation*}
Combining \eqref{6.3}, then we get
\begin{equation*}
\begin{split}
\mathcal{J}_9&=-\frac{2\varepsilon}{\big(\lambda_{\varepsilon}^{(1)}\big)^{\frac{N+2}{2}}}\Big(\int_{\R^N}U_{0,1}(z)c_0\phi_0dz+o(1)\Big).\\&
=\frac{2\varepsilon}{\big(\lambda_{\varepsilon}^{(1)}\big)^{\frac{N+2}{2}}}\Big(\int_{\R^N}U_{0,1}^2(z)dz\Big)c_0+o\Big(\frac{1}{(\lambda_{\varepsilon}^{(1)})^{\frac{N-2}{2}}}\Big).
\end{split}
\end{equation*}
The conclusion can be reached by the above estimates $\mathcal{J}_1$, $\mathcal{J}_2$, $\mathcal{J}_3$, $\mathcal{J}_4$, $\mathcal{J}_5$, $\mathcal{J}_6$, $\mathcal{J}_7$, $\mathcal{J}_8$, and $\mathcal{J}_9$.

\end{proof}


\end{document}